\documentclass[12pt,oneside,ngerman]{amsart}

\usepackage[T1]{fontenc}
\usepackage{a4}
\usepackage[latin9]{inputenc}
\usepackage{graphicx}
\usepackage{mathcomp}
\usepackage{amsfonts}
\usepackage{amsmath,amssymb}
\usepackage{cancel}

\usepackage{babel}

\makeatletter
\numberwithin{equation}{section} 
\numberwithin{figure}{section} 
 \theoremstyle{plain}    
 \newtheorem{thm}{Satz} 
 \theoremstyle{plain}    
 \newtheorem{lem}{Lemma} 
 \theoremstyle{plain}    
 \newtheorem{cor}{Korollar} 
 
 \theoremstyle{remark}
 
 \theoremstyle{definition}
 \newtheorem{defn}{Definition}
 \theoremstyle{definition}
  
  \theoremstyle{definition}
  
  \theoremstyle{plain}    
 \newtheorem{ek}{Euklidische Grundkonstruktion} 
  \theoremstyle{plain}    
 \newtheorem{oax}{Origami-Axiom} 
\makeatother
\newcommand{\N}{\mathbb{N}}
\newcommand{\E}{\mathbb{E}}
\newcommand{\R}{\mathbb{R}}

\newcommand{\Q}{\mathbb{Q}}

\newcommand{\Or}{\mathbb{O}}
\newcommand{\Lin}{\mathbb{L}}

\newcommand{\bpm}{\begin{pmatrix}}
\newcommand{\epm}{\end{pmatrix}}
\newcommand{\bi}{\begin{itemize}}
\newcommand{\ei}{\end{itemize}}
\newcommand{\be}{\begin{enumerate}}
\newcommand{\ee}{\end{enumerate}}

\begin{document}
\title{Geometrische Konstruktionen und Origami}

\author{Kay Paulus}

\maketitle
Zulassungsarbeit für das Lehramt an Gymnasien im Fach Mathematik\\
\\
Betreuer: Prof. Dr. Friedrich Knop\\

Eingereicht am 20. Juli 2012
\newpage

\tableofcontents
\newpage

{\huge \bfseries Widmung}\\
~\\
~\\
{\bfseries Für meine Oma, die während der Anfertigung dieser Arbeit verstorben ist.}

\newpage

Vor Beginn der eigentlichen Arbeit möchte ich einige Bemerkungen zur Notation machen.

\bi
\item Lateinische Großbuchstaben bezeichnen Punkte
\item Die Strecke zwischen den beiden Punkten $A$ und $B$ wird mit $\overline{AB}$ notiert. Gelegentlich werden Strecken auch mit überstrichenen Kleinbuchstaben abgekürzt.
\item $AB$ bezeichnet die Länge dieser Strecke. Streckenlängen werden gelegentlich lateinische Kleinbuchstaben zugewiesen.
\item $\overleftrightarrow{AB}$ bezeichne die Gerade, die durch die Punkte $A$ und $B$ geht. Geraden werden auch mit lateinischen Kleinbuchstaben bezeichnet, was gemeint ist, wird im Zusammenhang deutlich.
\item $\overrightarrow{AB}$ bezeichne die Halbgerade mit Endpunkt A.
\item $K(A,r)$ bezeichnet den Kreis um den Mittelpunkt $A$ mit Radius $r$, $K(A,B)$ den Kreis um Mittelpunkt $A$, der durch den Punkt $B$ geht.
\item $\angle(ABC)$ bezeichnet den Winkel mit Scheitelpunkt $B$, wobei $A$ ein Punkt auf dem ersten Schenkel und $C$ ein Punkt auf dem zweiten Schenkel ist. Winkel werden auch mit griechischen Kleinbuchstaben bezeichnet.
\item $\measuredangle(ABC)$ bezeichnet das Maß dieses Winkels.
\ei

\newpage

\section{Über die Geschichte geometrischer Konstruktionen}

Die alten Spiele sind oft die besten Spiele, und eines der ältesten Spiele sind die geometrischen Konstruktionen. \footnote{\cite{martin}, Einleitung} Die Griechen begannen vor ca. 2500 Jahren damit, sich mit geometrischen Fragestellungen, die freilich schon viel länger in alten Kulturen oder Situationen des täglichen Lebens entstanden waren, auf einer abstrakteren Ebene zu beschäftigen. \footnote{\cite{henn}, Kapitel 2} Euklid (360-280 v. Chr.) hat wohl eine der berühmtesten Abhandlungen hierzu geschrieben, {\em Die Elemente}, die das mathematische Wissen seiner Zeit zusammenfassten. Die Geometrie wurde in der Zeit der griechischen Mathematiker der Antike von einer konkreten Arbeit, wie die Landvermessung der alten Ägypter es noch war, zu einer rein theoretischen Wissenschaft.\\

Euklid baute seine Geometrie auf eine große Anzahl Definitionen, Postulaten und Axiomen auf, die man als anzuerkennende Grundlagen, als "`Spielregeln"' zu verstehen hat.\footnote{ebd}. Viele seiner Definitionen sind nicht mehr zeitgemäß, da er sich doch nicht vollständig von einer Anschauung lösen konnte. Seine Postulate und Axiome sind auch nicht unabhängig voneinander, wie man es von einem modernen Axiomensystem erwarten würde. Spätere Generationen haben jedoch seine Festlegungen zusammengefasst, zum Beispiel kann man folgende Grundkonstruktionen extrahieren:\footnote{zitiert nach\cite{ger95}, Abschnitt 2}

\begin{ek} Durch zwei nicht identische Punkte $P$ und $Q$ kann man mittels des Lineals eine eindeutig festgelegte Gerade $l=\overleftrightarrow{PQ}$ zeichnen, die beide Punkte enthält.
\end{ek}
\begin{ek} 
Gegeben sei ein Punkt $M$ und eine Strecke mit Länge $r>0$, dann kann man den eindeutig festgelegten Kreis $K(M,r)$ mit dem Zirkel zeichnen.
\end{ek}
\begin{ek} 
Gegeben seien zwei nicht parallele Geraden $g,h$. Dann haben diese den eindeutigen Schnittpunkt $P=g\cap h$.
\end{ek}
\begin{ek}
 Gegeben sei ein Kreis $K(M,r)$ und eine Gerade $g$, wobei die Entfernung zwischen $g$ und $K(M,r)$ kleiner oder gleich $r$ sein soll. Dann werde(n) dadurch die/der Schnittpunkt(e) von Kreis und Gerade festgelegt.
\end{ek}
\begin{ek}
 Gegeben seien zwei Kreise $K_1(M_1,r_1)$ und $K_2(M_2,r_s)$, und gelte entweder:
Keiner der Kreise enthält den Mittelpunkt des anderen in seinem Inneren, und $M_1M_2\le r_1+r_2$; oder aber, dass ein Kreis den Mittelpunkt des anderen in seinem Inneren enthält, und $M_1M_2\ge \|r_1-r_2\|$, dann lassen sich der/die Schnittpunkt(e) der beiden Kreise festlegen.
\end{ek}

Aus den Postulaten Euklids ergibt sich die Beschränkung auf die sogenannten "`Euklidischen Werkzeuge"', Zirkel und Lineal. \footnote{\cite{martin}, Seite 6} Ein Lineal im euklidischen Sinne ist ein Werkzeug, mit dem man eine Strecke durch zwei beliebige gegebene Punkte zeichnen kann. Ein "`reales"' Modell eines solchen Lineals wäre also eines, auf dem keine Längenskala aufgedruckt ist. Wenn in dieser Arbeit von einem Lineal gesprochen wird, ist immer ein solches euklidisches Lineal gemeint. Ein Zirkel ist ein Werkzeug, mit dem man einem Kreis mit einem gegebenen Mittelpunkt durch einen beliebigen anderen Punkt. Unsere modernen Zirkel sind hierfür ein adäquates physikalisches Modell.\\

\subsection{Mit Zirkel und Lineal konstruierbare Zahlen}
Eine wichtige Frage der Geometrie des alten Griechenlandes war die Frage, welche Strecken man mit diesen Werkzeugen konstruieren könne. In der heutigen Zeit ist diese Frage längst beantwortet. Da die Beantwortung dieser Frage nicht Hauptthema der Zulassungsarbeit ist, möchte ich die entsprechende Theorie nur kurz skizzieren, der Beweis wird ebenfalls nur skizziert\footnote{Ein exakter Beweis findet sich zum Beispiel in \cite{martin}, Kapitel 1 und 2}.

Wir identifizieren die Zeichenebene mit dem ${\R}^2$, versehen mit karthesischen Koordinaten, gegeben sei eine Strecke, die oBdA die durch die Punkte $(0,0);(0,1)$ festgelegte Einheitsstrecke ist. \footnote{\cite{martin}, S. 30ff}

\begin{defn}{\em ZuL-Punkt\\}
Ein Punkt ist mit Zirkel und Lineal konstruierbar (kurz: ZuL-Punkt), wenn dieser Punkt Ergebnis einer endlichen Sequenz von Punkten $P_1, P_2, \dots, P_n$ von Punkten, sodass jeder Punkt ist entweder einer der Startmenge oder aber in einer der folgenden Arten entstanden ist:

\be
\item Schnittpunkt zweier Geraden, die jeweils durch 2 Punkte gehen, die vorher in der Sequenz aufgetreten sind.
\item Schnittpunkt einer Geraden durch 2 Punkte, die vorher bereits aufgetreten sind, mit einem Kreis, der einen vorher aufgetretenen Punkt als Mittelpunkt hat und durch einen weiteren vorher aufgetretenen Punkt geht.
\item Schnittpunkt zweier Kreise, die jeweils durch einen vorher aufgetretenen Mittelpunkt und einen vorher aufgetretenen Punkt auf der Kreislinie festgelegt sind.
\ee

Eine {\em ZuL-Gerade} ist eine Gerade durch zwei ZuL-Punkte, ein {\em ZuL-Kreis} ist einer, mit einem ZuL-Punkt als Mittelpunkt durch einen ZuL-Punkt. Eine Zahl $x$ ist eine ZuL-Zahl, wenn $(x,0)$ ein ZuL-Punkt ist.
\end{defn}

Insbesondere sind damit die Schnittpunkte von zwei ZuL-Geraden, einer ZuL-Gerade und einem ZuL-Kreis und die Schnittpunkte von zwei ZuL-Kreisen ZuL-Punkte und die Koordinatenachsen ZuL-Geraden. Es wird aber ausdrücklich nicht gesagt, dass jeder Punkt auf einer ZuL-Geraden ein ZuL-Punkt ist!

Da man mit Hilfe eines Zirkels die Länge einer Strecke an eine andere Strecke antragen kann, sind mit $A$ und $B$ auch die Zahlen $A+B$ und $A-B$ ZuL-Punkte. Der Strahlensatz liefert eine Möglichkeit, aus den Zahlen $a$ und $b$ die Zahlen $ab$ und $\frac{a}{b}$ für $b\ne 0$ zu konstruieren. \footnote{\cite{henn}, Seite 49}

\begin{figure}[h]
\includegraphics[width=8cm]{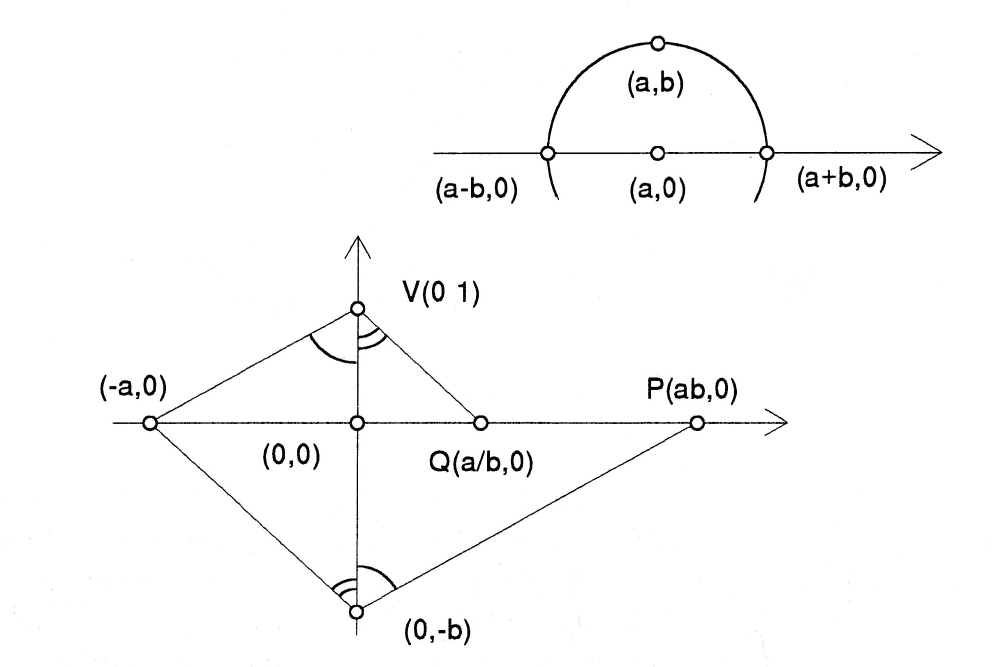}
\caption[\cite{martin}, Seite 34]{Konstruierbarkeit von Summen, Differenzen, Produkten und Quotienten}
\label{zul1}
\end{figure}

Wir erhalten somit, dass die ZuL-Zahlen einen Körper im algebraischen Sinne bilden.

Der Höhensatz liefert nun weiterhin eine Möglichkeit, Quadratwurzeln zu konstruieren.\footnote {ebd}\\

\begin{figure}[h]
\includegraphics[width=8cm]{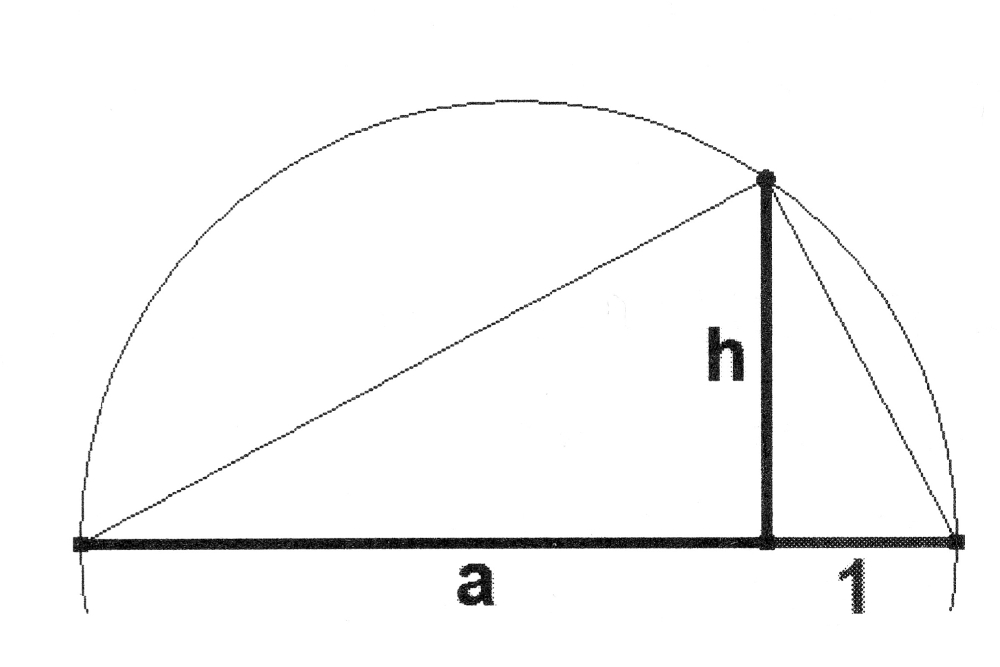}
\caption[\cite{henn}, Seite 49]{Konstruierbarkeit von Quadratwurzeln}
\label{zul2}
\end{figure}

Damit werden die ZuL-Zahlen zu einem euklidischen Körper:

\begin{defn}{Euklidischer Körper\\}
Ein Körper heißt euklidisch, wenn mit jedem Körperelement $x$ auch $\sqrt{x}$ im Körper enthalten ist.
\end{defn}

Nun benötigen wir noch die Aussage, dass, wenn $d$ eine positive Zahl ist, die in einem Körper $F$ enhalten ist, $\sqrt{d}$ aber kein Körperelement ist, die Menge $\{p+q\sqrt{d}, p,q \in F\}$ ein Körper ist. Man bezeichnet diesen Körper $F(\sqrt{d})$ als einen quadratischen Erweiterungskörper von $F$. \footnote{\cite{martin}, Seite 35f}. Dies zusammen führt zum Hauptsatz über konstruierbare Zahlen.

\begin{thm}{Hauptsatz über konstruierbare Zahlen\footnote{\cite{henn}, Seite 56}\\}
Mit Zirkel und Lineal sind genau diejenigen reellen Zahlen konstruierbar, die in einem reell-algebraischen Erweiterungskörper $K$ von $\Q$ mit Körpergrad $(K:Q)=2^n$ liegen, wobei $\Q$ und $K$ durch eine Kette quadratischer Körpererweiterungen verbunden sind.\\
Die ZuL-Zahlen lassen sich also als (eventuell mehrfach geschachtelte) Quadratwurzeln ausdrücken. Diese Menge der ZuL-Zahlen sei mit $\E$ bezeichnet.
\end{thm}

Damit haben wir, in unserer modernen Sprache, festgehalten, welche Zahlen mit Zirkel und Lineal konstruierbar sind. Hiermit kann man leicht zeigen, dass viele klassische Konstruktionsprobleme aus dem alten Griechenland nicht mit Zirkel und Lineal lösbar sind. \footnote{\cite{henn}, S. 50 ff}\\\

\bi
\item Eines davon ist die Quadratur des Kreises, also die Aufgabe, zu einem gegebenen Kreis ein flächengleiches Quadrat zu konstruieren. Nähme man an, der Kreis habe Radius 1, dann müsste man ein Quadrat mit Seitenlänge $\sqrt{\pi}$ konstruieren, dies ist nicht möglich, da $\pi$ nach dem Satz von Lindemann transzendent über $\Q$ ist.

\item Weiterhin ist die Verdoppelung eines Würfels, die als das Deli'sche Problem bekannt ist, nicht mit Zirkel und Lineal lösbar. 

\item Auch das Problem der Winkeldrittelung führt im Allgemeinen auf eine Gleichung vom Grad 3, und ist daher nicht mit Zirkel und Lineal lösbar.

\item Die Konstruktion regelmäßiger n-Ecke ist schon nicht mehr so einfach, tatsächlich weiß man heute, dass dieses Problem nicht für alle $n\in \N$ lösbar ist.

\ei

Neben dieser streng mathematischen Sichtweise auf Konstruktionsprobleme gab es aber auch schon immer eine pragmatische Sicht der Welt. Die Glaser der Gotik mussten für den Bau von Kirchen fähig sein, alle möglichen $n$-Ecke zu "`konstruieren"', für dieses Problem hatten schon die alten Griechen Näherungskonstruktionen gefunden.\footnote{ebd, S. 50 ff}
Weiterhin war es bereits Archimedes bekannt, dass die Winkeldrittelung exakt lösbar ist, wenn man ein Lineal mit 2 Markierungen verwenden darf. \footnote{ebd, S. 67f}\\

Man kann also sagen, dass die Frage nach der Konstruierbarkeit einer Zahl bzw. einer Streckenlänge wesentlich davon abhängt, welche Werkzeuge man für die Konstruktion zur Verfügung hat. An dieser Stelle möchte diese Arbeit ansetzen, die sich näher mit der Konstruktion durch Methoden des Origami beschäftigt. Es wird sich zeigen, dass die ersten drei dieser klassischen Konstruktionsprobleme mit Origami lösbar sind, jedoch ist auch mit Origami nicht jedes beliebige n-Eck konstruierbar.\\

Zunächst soll, ähnlich wie bei der Beschreibung der euklidischen Konstruktionen, klargestellt werden, welche Grundkonstruktionen im Origami zur Verfügung stehen, dazu soll ein mögliches System von Origami-Axiomen vorgestellt werden. Im Anschluss wird die Frage diskutiert werden, welche Zahlen mit den Methoden des Origami konstruierbar sind, und es wird gezeigt, dass diese Menge die ZuL-Zahlen enthält und identisch mit der Menge der mit einem markierten Lineal konstruierbaren Zahlen ist.
Im Anschluss sollen einige konkrete Origami-Konstruktionen vorgestellt werden. Hierzu gehören eine Methode der Winkeldrittelung sowie die allgemeine Konstruktion von dritten Wurzeln. Ein kleiner historischer Exkurs führt zur Konstruierbarkeit der Zahl des Goldenen Schnittes, bevor die Leistung des jungen Gauss, das 17-Eck mit Zirkel und Lineal zu konstruieren, mit den Methoden des Origami gezeigt wird. Dieses Kapitel schließt die konkrete Konstruktion des 7-Ecks mit Origami-Methoden ab, ein Beispiel für ein regelmäßiges n-Eck, welches nicht mit Zirkel und Lineal, wohl aber mit Origami konstruierbar ist.\\
Auch der Werkzeugbereich des Origami ist für Erweiterungen offen, so ist es zum Beispiel möglich, durch die Zulassung von kreisförmigen Faltungen die Zahl Pi zu konstruieren. Abschließen soll die Arbeit mit einem faszinierenden Ergebnis, das wunderbar für die Demonstration der Macht von Symmetrien geeignet ist, und daher auch und gerade für Lehrer nicht uninteressant ist, dem Cut-Fold-Theorem, welches besagt, dass jeder Polygonzug so gefaltet werden kann, dass man ihn anschließend mit nur einem Schnitt ausschneiden kann.

\newpage

\section{Ein Origami-Axiomensystem}

Origami ist, wenn man den Namen wörtlich übersetzt, die Kunst des Faltens von Papier. Im traditionellen Origami versteht man darunter im engeren Sinne, dass Schneiden oder Reißen verboten ist, sondern nur gerade Faltungen zugelassen werden.  \footnote{\cite{henn}, Kapitel 3.4} Weiterhin sollen zunächst nur solche Faltungen betrachtet werden, bei denen das Papier nach Durchführung einer Faltung stets umgehend wieder auf gefaltet wird, bevor man die nächste Faltung durchführt. \\

Das Falten von Papier als Methode für geometrische Konstruktionen wurde erstmals 1893 von T. Sundara Row aus Indien vorgestellt. Die wohl erste strenge Behandlung von Origami als Konstruktionsmethode waren die "`Geometric Tool"' von R.C. Yates 1949, er stellte auch ein erstes Axiomensystem für Origami-Grundkonstruktionen auf. Sein Axiomensystem bestand aus den folgenden drei Grundkonstruktionen.

\be
\item Platziere einen Punkt des Blattes auf einem anderen und falte das Blatt.
\item Falte das Blatt so, dass die Faltlinie durch zwei gegebene Punkte verläuft.
\item Falte einen gegebenen Punkt so auf eine gegebene Gerade, sodass die dabei entstehende Faltlinie durch einen weiteren, gegebenen Punkt verläuft, wenn die Punkte denn so liegen, dass das möglich ist.
\ee

Yates bewies in seiner Arbeit, dass auf Grundlage dieser 3 Konstruktionen alle Konstruktionen der Euklidischen Geometrie durch Falten möglich sind. Der durch dieses Axiomensystem konstruierbare Körper ist also der Körper $\E$ aus der Einleitung. \footnote{\cite{martin}, 145f}

Ein weiteres, mathematisch präzises Axiomensystem für das Origami wurde 2000 von Roger Alperin entworfen. \footnote{Zitiert nach \cite{henn}, Seite 60f}. Hierbei werden 6 Origami-Axiome vorgestellt. Zu beachten ist hierbei, dass die Konstruktionen, die seine ersten 5 Axiome erlauben, exakt denen des Yates-Systems entsprechen, das 6. Origami-Axiom aber eines ist, das zu einem echt größeren Körper konstruierbarer Zahlen führt.\\

Interessant ist hierbei, dass in der euklidischen Geometrie anschaulich klar ist, wie ein Anschauungsobjekt für einen Punkt zu erzeugen ist, nämlich durch das Markieren eines Punktes in der Zeichenebene. \footnote{\cite{ger95}, Seite 359}\\

 Die einfachste "`Prozedur"', die man mit einem Blatt Papier durchführen kann, ist das Falten entlang einer Geraden, am naheliegendsten dürfte die Halbierungsfaltugn parallel zu einer Kante sein. Prinzipiell kann man dabei Berg- und Talfalten unterscheiden. \footnote{\cite{flachs}, Seite 11} 
\begin{figure}[h]
\includegraphics[width=8cm]{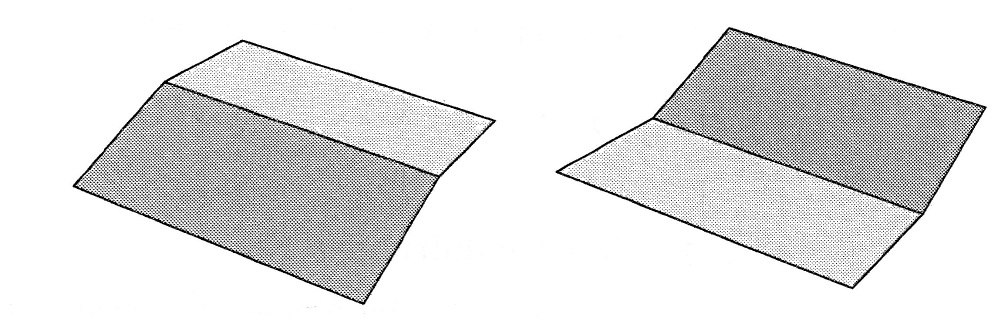}
\caption[\cite{flachs}, Seite 11]{Berg- und Talfaltung}
\label{zul2}
\end{figure}

Der im Papier auch nach dem Auffalten sichtbare Knick soll im Folgenden, unabhängig davon, ob er durch eine Berg- oder Talfaltung entstanden ist, als Faltlinie bezeichnet werden.
 
Es erscheint natürlicher, eine Gerade zu falten, als zum Beispiel eine krumme Linie, gleichwohl dieses möglich, wenn auch in der Praxis schwer exakt durchführbar ist. Hierzu sei auf das vorletzte Kapitel verwiesen. Natürlich entstehen durch iterierte Faltungen hierbei Objekte, die mehr oder weniger dreidimensional sind, gerade die künstlerische Richtung des Origami vollbringt hier ästhetische Höchstleistungen. Außer im letzten Kapitel der Arbeit soll es aber um flach aufgefaltetes Papier gehen, das man nach den entsprechenden Faltungen immer wieder aufklappt. Für die meisten Überlegungen in dieser Arbeit ist es unerheblich, ob ein rechteckiges oder quadratisches Papier verwendet wird, solange nichts anderes gesagt wird, darf von einem quadratischen Papier ausgegangen werden.\footnote{ebd}\\ 
 Das grundlegende Objekt unserer Betrachtungen ist also eine gerade Faltung, die wir mit einer Geraden identifizieren können. Nun liegt es nahe, einen Punkt als den Schnittpunkt zweier solcher Faltungen aufzufassen.\footnote{\cite{ger95}, Seite 359}
 
In den Skizzen zu den Axiomen werden bereits vorhandene Geraden fett, zu faltende Linien gestrichelt und neu entstehende Punkte als Kreuz und bereits vorhandene Punkte als Kreis gezeichnet, gestrichelte Pfeile deuten die Richtungen von Faltungen an.

\begin{oax}
Zu zwei nicht parallelen Faltlinien $g$ und $h$ kann der (eindeutige) Schnittpunkt $A$ konstruiert werden.
\end{oax}

\begin{figure}[h]
\includegraphics[width=5cm]{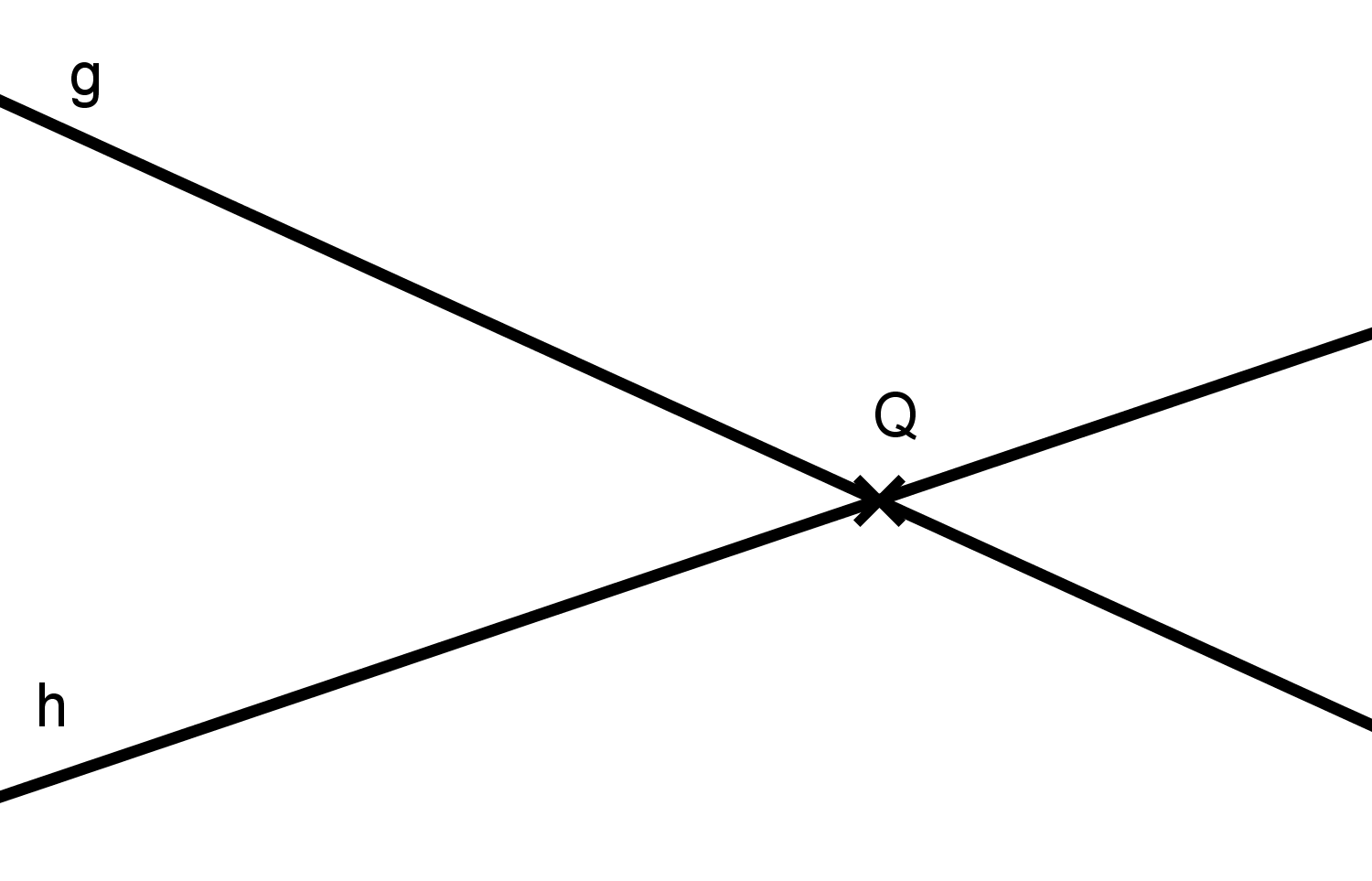}
\caption[Eigene Abbildung]{Origami-Axiom 1}
\label{zul2}
\end{figure}

Natürlich handelt es sich hierbei exakt um die Aussage der 3. Euklid-Grundkonstruktion. Wie schon angedeutet handelt es sich aber nun um eine Erklärung, wie man Punkte in der Ebene veranschaulichen kann. Man kann also sagen, dass die Faltung diejenige zentrale Position in der Origami-Geometrie einnimmt, die die Punkte in der Euklidischen Geometrie einnehmen. Zwei nicht-parallele Geraden definieren also auch hier einen eindeutigen Schnittpunkt.\\
Nachdem nun Punkte erklärt sind, können wir die Faltlinie durch zwei Punkte definieren.\footnote{\cite{ger95}, Seite 359ff}

\begin{oax}
Durch zwei nicht identische Punkte $P$ und $Q$ kann die Verbindungslinie gefaltet werden.
\end{oax}
\begin{figure}[h]
\includegraphics[width=5cm]{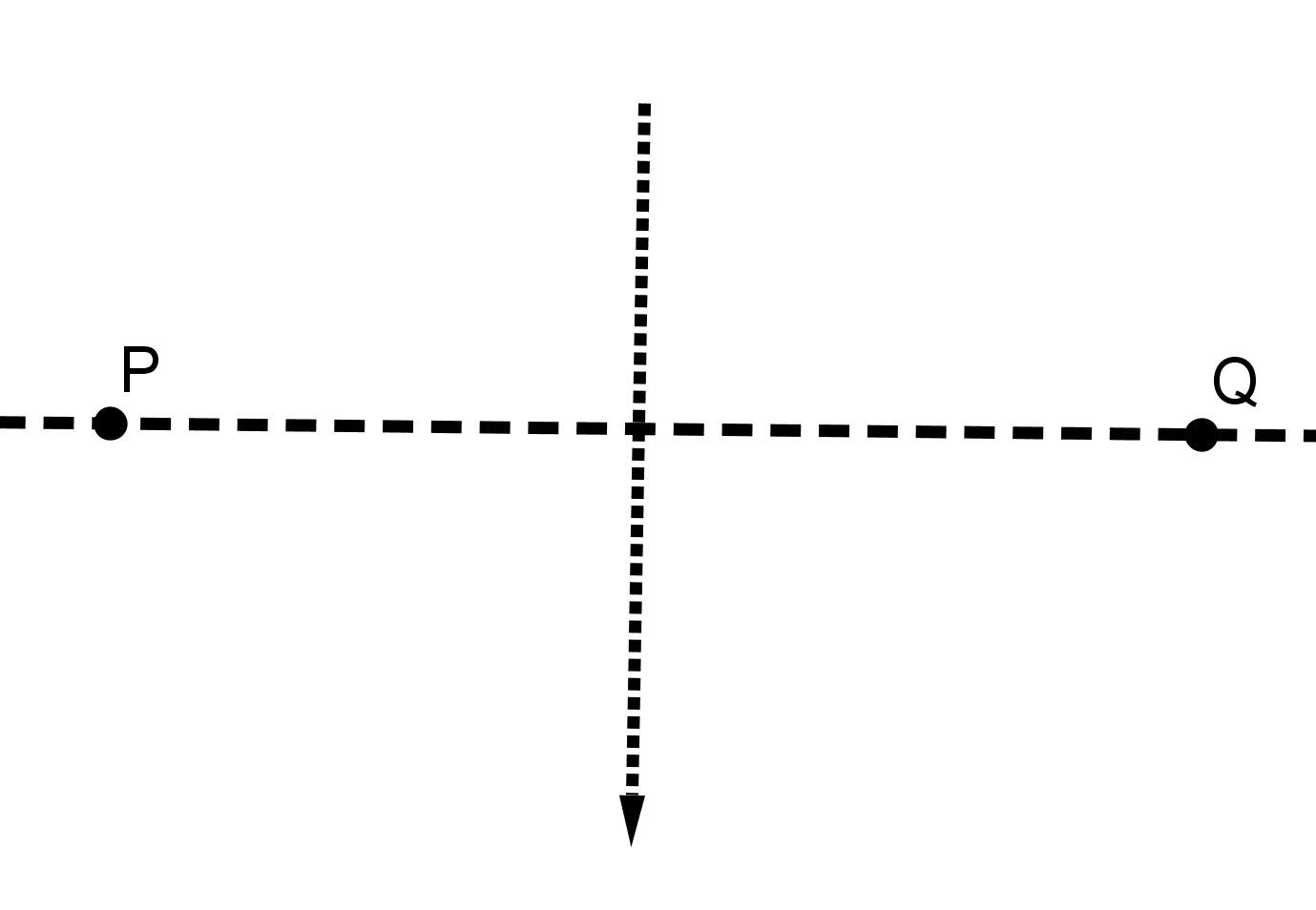}
\caption[Eigene Abbildung]{Origami-Axiom 2}
\label{zul2}
\end{figure}

Hierbei handelt es sich offensichtlich um das Äquivalent zu Euklid 1.\footnote{ebd}\\

Wir können die beiden Punkte aber nicht nur verbinden, wir können sie auch aufeinander falten.

\begin{oax}
Seien $P$ und $Q$ zwei verschiedene Punkte. Dann kann so gefaltet werden, dass $Q$ auf $P$ fällt.
\end{oax}

\begin{figure}[h]
\includegraphics[width=5cm]{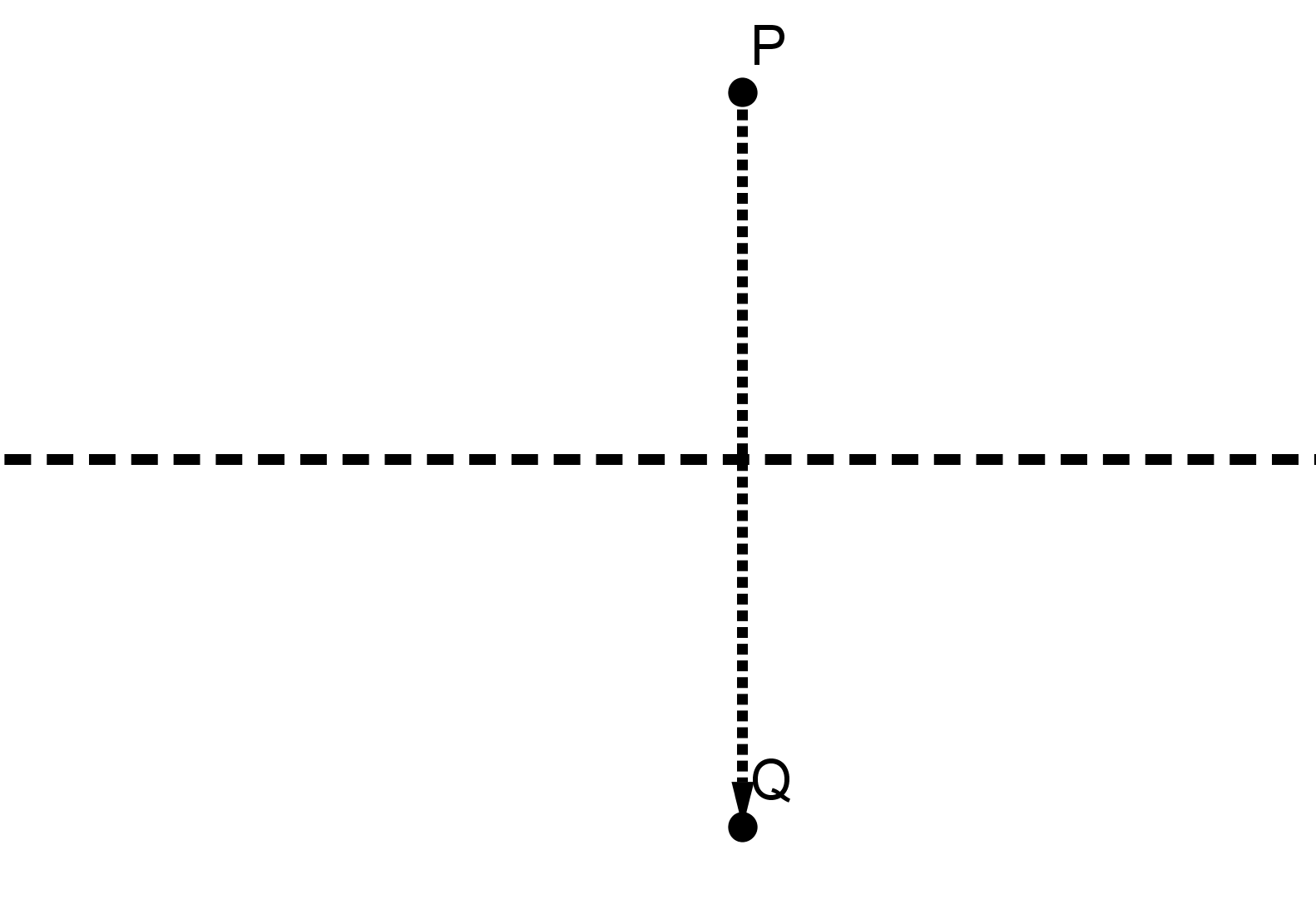}
\caption[Eigene Abbildung]{Origami-Axiom 3}
\label{zul2}
\end{figure}

Man sieht recht schnell ein, dass diese Konstruktion die Mittelsenkrechte der Strecke $\overline{PQ}$ konstruiert. 

\begin{oax}
Seien $g$ und $h$ zwei verschiedene Faltlinien. Dann kann $g$ auf $h$ gefaltet werden.
\end{oax}

Dieses Axiom wird zum Beispiel in der Publikation [Ger95] in zwei Axiome aufgeteilt. Das ist insofern gerechtfertigt, dass sich hierbei im Falle paralleler Faltlinien die Grundkonstruktion der Mittelparallelen ergibt.\footnote{\cite{ger95},Seite 360}

\begin{figure}[h]
\includegraphics[width=5cm]{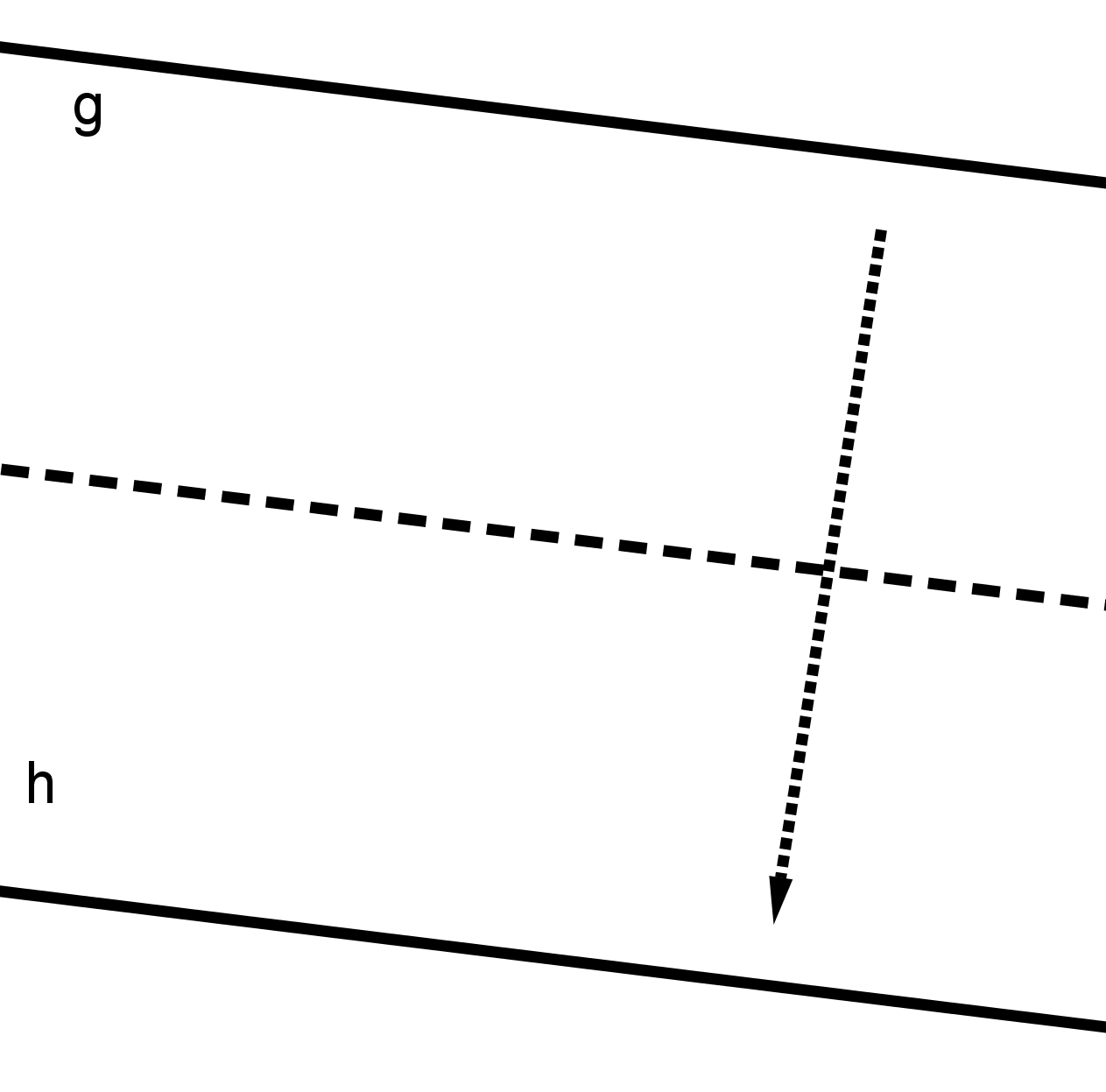}
\caption[Eigene Abbildung]{Origami-Axiom 4a}
\label{zul2}
\end{figure}

Betrachtet man jedoch dieses Axiom für nicht parallele Faltlinien, so ergibt sich die Konstruktion der Winkelhalbierenden des eingeschlossenen Winkels.\footnote{\cite{henn}, Seite 60f}

\begin{figure}[h]
\includegraphics[width=5cm]{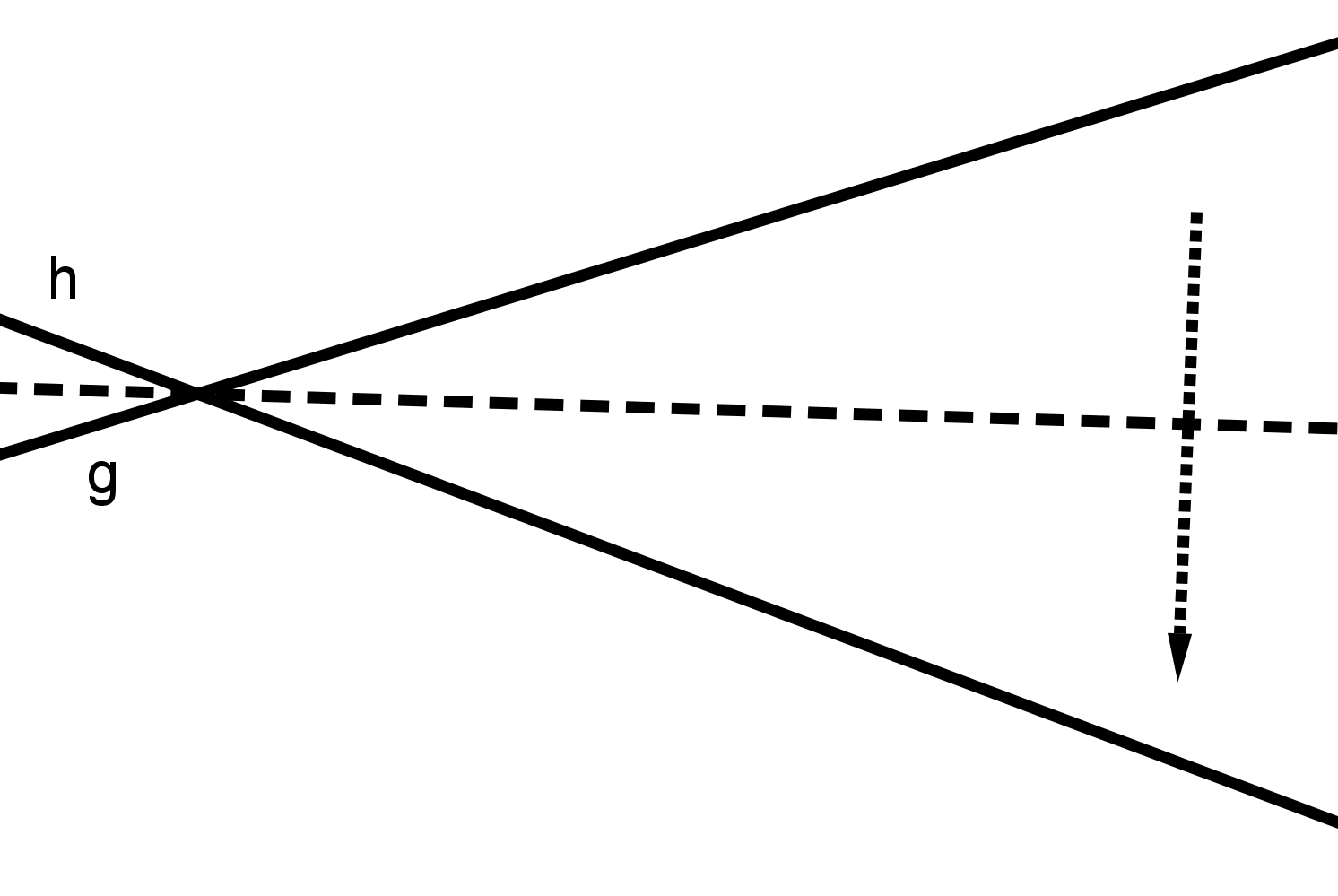}
\caption[Eigene Abbildung]{Origami-Axiom 4b}
\label{zul2}
\end{figure}

Die nächste Grundkonstruktion ist eine, der man die Verwandtschaft mit einer Konstruktion der Euklidischen Geometrie nicht sofort ansieht.

\begin{oax}
Gegeben seien zwei verschiedene Punkte $P$, $Q$ und und eine Faltlinie $g$. Dann kann man $P$ so auf $g$ falten, dass die dabei entstehende Faltlinie durch $Q$ geht.
\end{oax}

\begin{figure}[h]
\includegraphics[width=5cm]{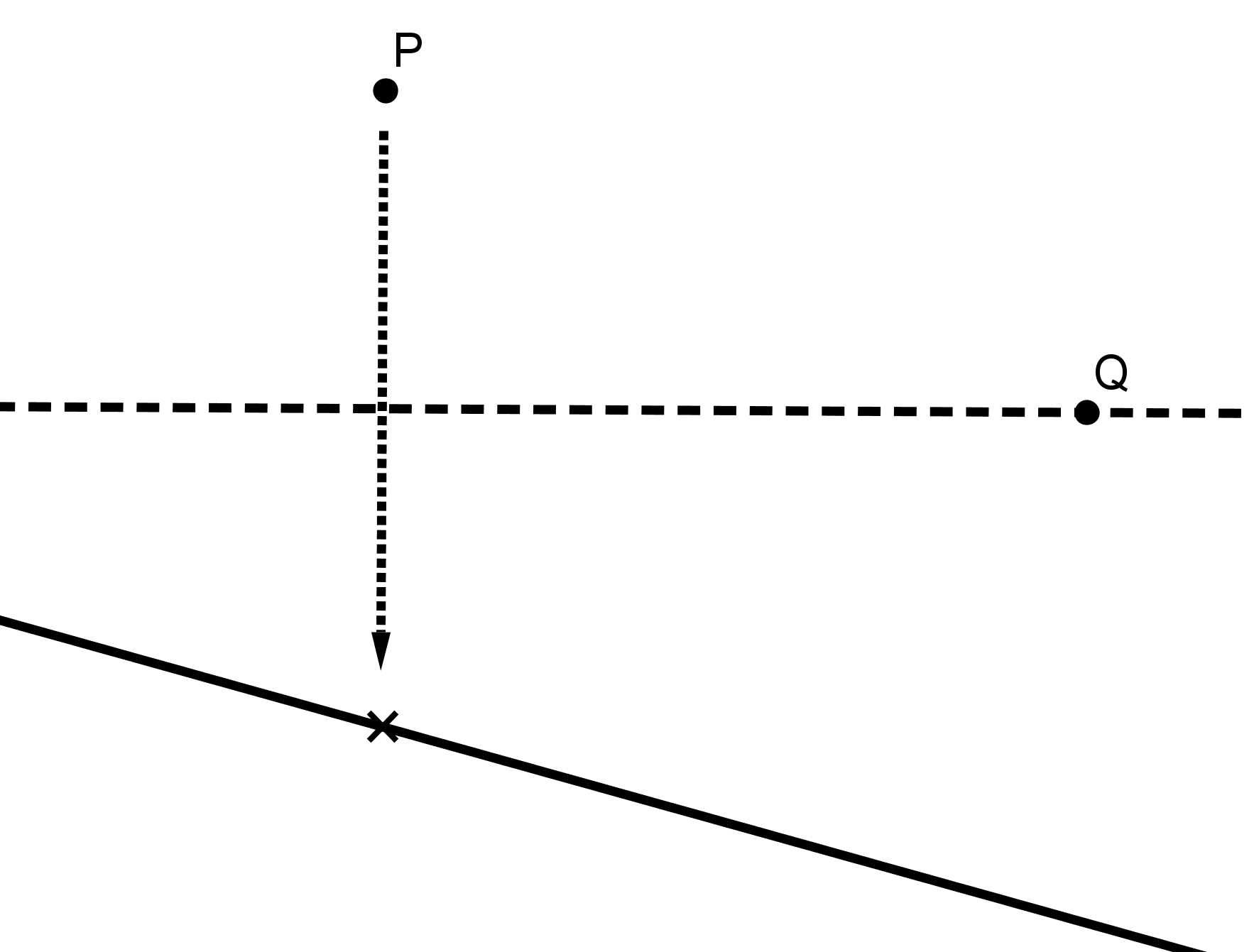}
\caption[Eigene Abbildung]{Origami-Axiom 5}
\label{zul2}
\end{figure}

Diese Konstruktion entspreicht der Konstruktion des Schnittes des Kreises $K(Q,QP)$ mit der Geraden $g$.\footnote{\cite{henn}, Seite 61}\\
Diese 5 Axiome beinhalten genau die gleichen möglichen Konstruktionen wie das Yates-System, bestehen sie doch aus (teilweise iterierten) Yates-Grundkonstruktionen.\\
Ein direktes Korollar aus dem Axiom 5 ist die Konstruktion eines Lotes durch einen Punkt:
\begin{cor}{Lot durch einen Punkt}\\
Es ist genauso möglich, eine Gerade $g$ so auf sich selbst zu falten, dass die Faltlinie durch einen bestimmten Punkt $Q$ geht. Die Faltlinie ist dann das Lot zu $g$ durch den Punkt $P$. 
\end{cor}
\begin{proof}
Es handelt sich hierbei um Axiom 5 für den Fall, dass die Gerade $g$ durch $P$ geht.
\end{proof}

Das sechste Origami-Axiom von Alperin ermöglicht nun eine Konstruktion, die so mit Zirkel und Lineal nicht möglich ist:\footnote{ebd}
\begin{oax}
Gegeben seien jeweils verschiedene Punkte $P$ und $Q$ und Geraden $p$ und $q$, so dass nicht gleichzeitig $P\in p, Q \in q$ und $p \parallel q$ gilt. Dann kann man so falten, dass $P$ auf $p$ und $Q$ auf $q$ zu liegen kommt.
\end{oax}

\begin{figure}[h]
\includegraphics[width=5cm]{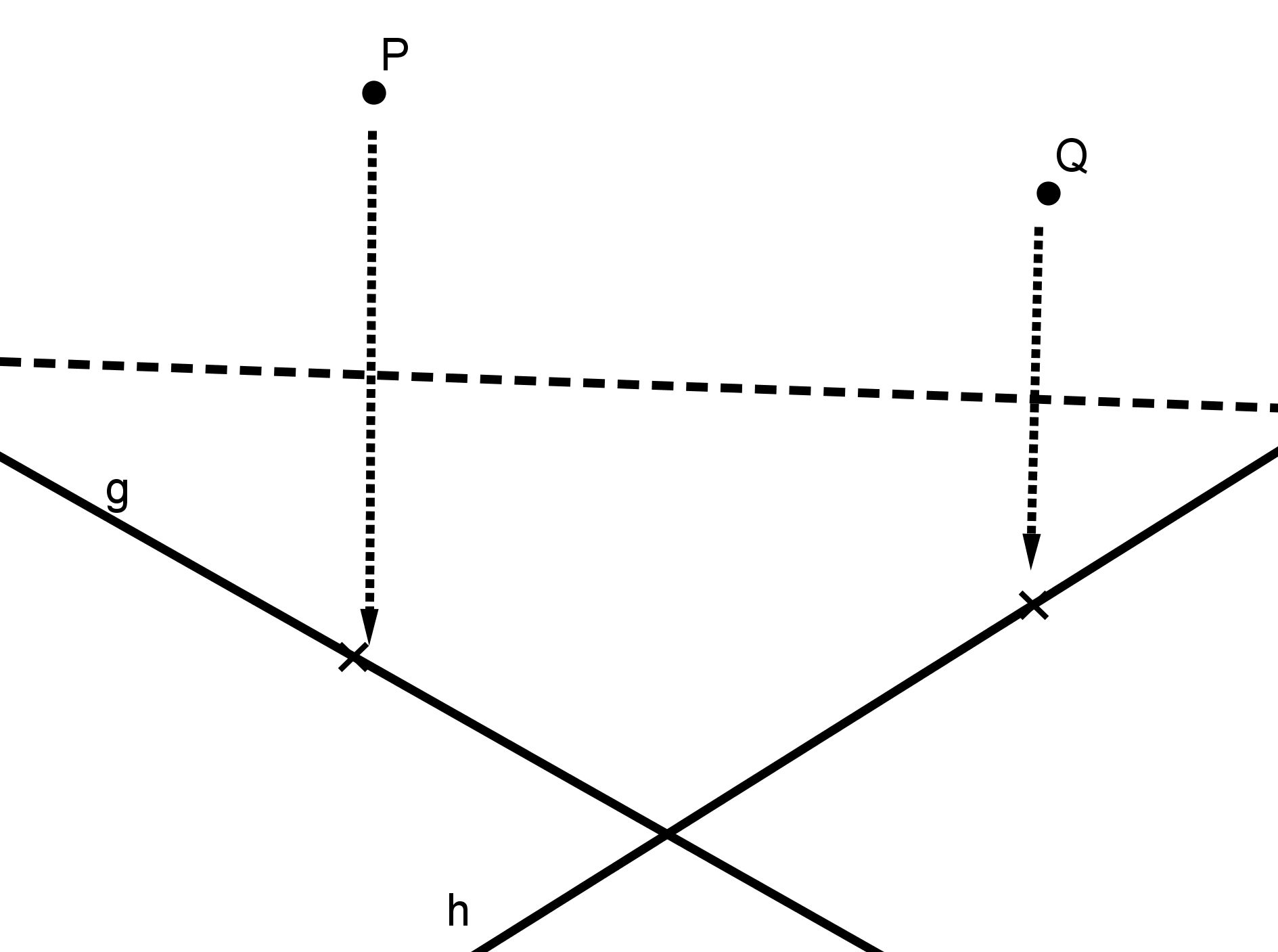}
\caption[Eigene Abbildung]{Origami-Axiom 6}
\label{zul2}
\end{figure}

Dieses Axiom macht den großen Unterschied zwischen der Euklidischen Geometrie und der Origami-Geometrie, denn dieser Origami-Konstruktion entspricht einer Konstruktion mit sogenanntem Einschiebelineal, welches in einem späteren Kapitel noch ausführlicher behandelt wird. Die Möglichkeit zu dieser Konstruktion gibt es im Axiomensystem von Yates nicht. Unter einem Einschiebelineal versteht man ein Lineal, auf dem zwei Punkte $A,B$ markiert sind, somit ist also eine bestimmte Streckenlänge auf diesem Lineal gegeben. Man kann dann das Lineal so zu verschieben, dass diese Punkte auf bestimmten Geraden (oder Kreisen) liegen.\footnote{\cite{henn}, Seite 61} Im Kapitel über die Trisektion des Winkels wird eine exakte Konstruktion beschrieben, die die Technik Einschiebelineal verwendet.\\

Als ein einfaches Beispiel für einen Origami-Beweis, und als erstes Herantasten an die Beweismethoden der Origami-Geometrie, möchte ich nun den Satz von Haga beweisen, der eine Origami-Konstruktion für die Dreiteilung einer Strecke liefert.

\begin{thm}{Satz von Haga}\\
Gegeben sei ein quadratisches Papierstück mit den Ecken ABCD. Man konstruiert zunächst die Mitte $M$ der oberen Seite $\overline{CD}$, dann wird die Ecke $B$ auf $M$ gefaltet. Der Schnittpunkt $G$ der Kante $\overleftrightarrow{AB}$ mit der Blattkante $\overleftrightarrow{AD}$ drittelt dann $\overline{AD}$.\footnote{\cite{henn}, Seite 61f}
\end{thm}

\begin{figure}[h]
\includegraphics[height=8cm]{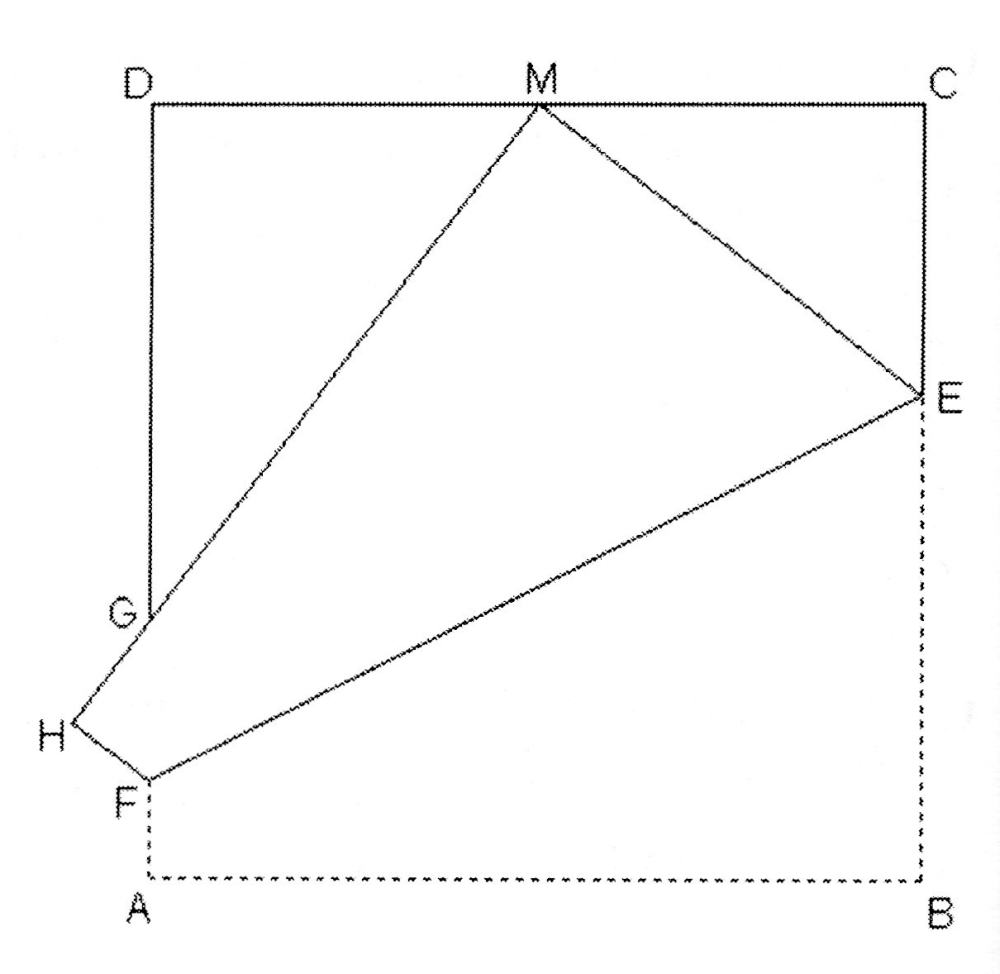}
\caption[\cite{henn}, Seite 62]{Satz von Haga}
\label{zul2}
\end{figure}

\begin{proof}
Die Konstruktion des Mittelpunktes von $\overline{AD}$ ist nach dem dritten Origami-Axiom möglich, die Faltung von $B$ auf $M$ ebenso.\\
Das Ausgangsquadrat habe o.B.d.A. die Seitenlänge 8 Längeneinheiten, da $M$ der Mittelpunkt der Seite $\overline{AD}$ ist, ist dann $MC=4$. Man betrachtet nun das Dreieck $\Delta ECM$, es bezeichne $x=EC$ und $y=EM=EB$. Der Satz des Pythagoras liefert dann im (offensichtlich bei $C$ rechtwinkligen) Dreieck: $y^2=4^2+x^2$. Weiterhin ist nach Konstruktion klar, dass $x+y$ die Länge der gesamten Quadratseite ergibt, also $x+y=8$. Ineinander einsetzen liefert sofort:

\begin{align*}
(8-x)^2 &= 16+x^2\\
64 - 16x + x^2 &= 16+x^2\\
64-16x &= 64\\
x &= 3
\end{align*}

und damit $y=5$.

Das Dreieck hat also ein pythagoräisches Tripel als Seitenlängen! Weiter ist $\measuredangle{EMC}=\measuredangle{DMG}$ und damit, auf Grund der Rechten Winkel des Ursprungsquadrates, $\measuredangle{CEM}=\measuredangle{MGD}=\measuredangle{HGF}$ (Scheitelwinkel). Damit muss $\measuredangle{GFH}=\measuredangle{CEM}$ sein, und damit haben die Dreiecke $\Delta MEC$, $\Delta GMD$ und $\Delta FGH$ die gleichen Winkel und sind somit ähnlich. Angewandt auf die Dreiecke $\Delta GMD$ und $\Delta ECM$ folgt damit $DG:4=4:3$, und damit:
	\begin{align*}
	DG &= \frac{16}{3}=\frac{2}{3}\cdot 8=\frac{2}{3} AD\\
		AG &= 8-\frac{2}{3}\cdot {8} = \frac{1}{3} AD
\end{align*}

\end{proof}

\newpage

\section{Die Origami-Zahlen}

\subsection{Die Origami-Konstruktionen und die euklidischen Konstruktionen}

In diesem Abschnitt der Arbeit soll untersucht werden, welcher Zahlenbereich mittels Origami-Konstruktionen gemäß der Axiome des vorhergehenden Kapitels konstruiert werden können. Ein geeigneter Anfang für Origami-Konstruktionen ist eine Menge von 2 Geraden und 2 Punkten, es braucht aber 3 Punkte, um 2 Geraden festzulegen. Praktischerweise beginnt man also hier stets mit zwei orthogonalen Geraden, die wir mit x- und y-Achse des Koordinatensystems identifizieren, und dem Schnittpunkt $(0,0)$ sowie den Einheitspunkten $(1,0)$ und $(0,1)$. Diese Startmenge enthält natürlich dann die Geraden $x=0, y=0$ und $x+y=1$Als Nächstes muss nun gesagt werden, was eine Origami-Zahl ist.

\begin{defn}{\em Origami-Punkt\\}
Eine Gerade ist mit Hilfe der Origami-Axiome konstruierbar (kurz: Origami-Gerade), wenn sie die letzte einer endlichen Abfolge von Geraden $t_1, \dots, t_n$ ist, so dass jede der Geraden eine aus der Startmenge ist, oder aber eine Gerade $t$ sein, bei der die Spiegelung von zwei Punkten $P$ und $Q$ an dieser Geraden auf Geraden $p$ bzw $q$ zu liegen kommen, und dass gilt:
\bi
\item $p$ und $q$ sind nicht parallele Geraden, die schon vorher in der Abfolge aufgetaucht sind\\
\item $P$ und $Q$ sind verschiedene Punkte, von denen jeder der Schnitt zweier vorher in der Abfolge auftauchenden Geraden ist\\
\item Entweder ist $q=\overleftrightarrow{PQ}$ oder $p \bot q$
\ei

Ein Origami-Punkt ist ein Punkt, der sich als Schnitt zweier Origami-Geraden ergibt.\\
Ein Origami-Kreis ist ein Kreis durch einen Origami-Punkt mit einem Origami-Punkt als Mittelpunkt.\\
Eine Zahl $x$ ist eine Origami-Zahl genau dann, wenn $(x,0$ ein Origami-Punkt ist.\\
Ein Punkt ist genau dann ein Origami-Punkt, wenn seine beiden Koordinaten Origami-Zahlen sind.\\
Die Menge der Origami-Zahlen soll mit $\Or$ bezeichnet werden.\footnote{nach \cite{martin}, Seite 151}
\end{defn}

Zunächst soll gezeigt werden dass mit den Origami-Axiomen alle euklidischen Grundkonstruktionen durchgeführt werden können.\footnote{Dieser Teil nach \cite{ger95}, Abschnitt 4}\\ 

Offensichtlich ist die euklidische Konstruktion 1, das Ziehen einer Geraden durch 2 Punkte, identisch zum Origami-Axiom 2.\\
Das Zeichnen eines Kreises, Grundkonstruktion 2, ist schon nicht mehr so einfach nachvollziehbar, ein Kreis kann nicht mit Origami gezogen werden. Er ist aber exakt bestimmt, wenn man nur Mittelpunkt und Radius kennt, denn damit lassen sich alle Punkte auf der Kreislinie und auch Tangenten konstruieren. 
\begin{figure}[h]
\includegraphics[width=5cm]{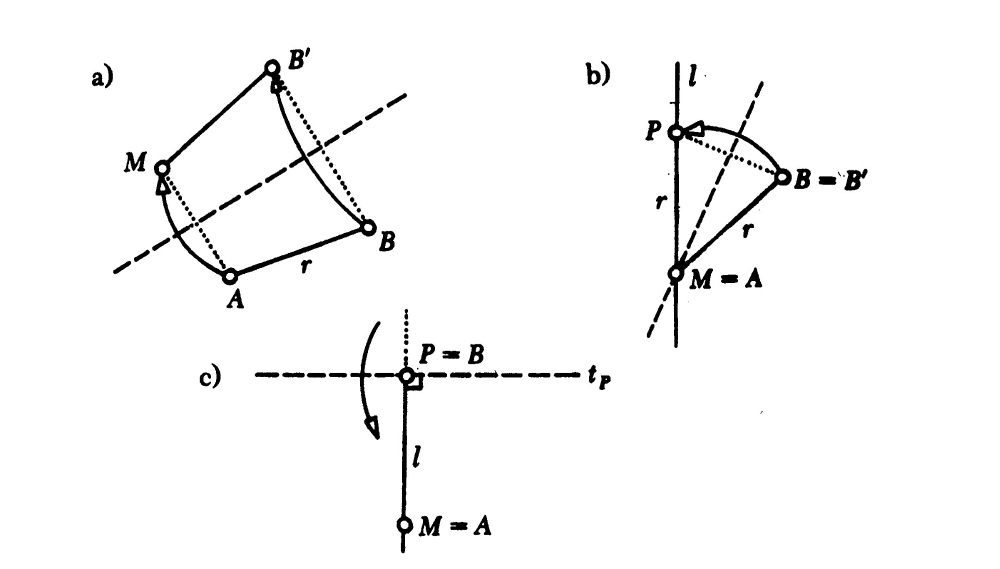}
\caption[\cite{ger95}, Abschnitt 4]{Kreise mittels Origami}
\label{zul2}
\end{figure}
Abbildung a) soll verdeutlichen, dass, wenn Mittelpunkt $M$ und der Radius $AB$ eines Kreises bekannt ist, es möglich ist, $A$ auf $M$ zu falten (Origami-Axiom 3). Dies bildet $B$ auf einen Punkt $B'$ ab, welcher von $M$ den Abstand $r$ hat, also ein weiterer Punkt auf dem Kreis ist. In Abbildung b) ist eine Gerade $l$ durch $M$ sowie der Radius $\overline{MB'}$ gegeben, dann kann dieser auf $l$ gefaltet werden, und somit der Punkt, der auf $l$ und $K(M,MB')$ liegt konstruiert werden. Wenn man schließlich, wie in Abbildung c), eine Gerade $l$  durch den Mittelpunkt $M$ und einen Punkt $P$ auf dem Kreis gegeben hat, kann man nun die Gerade $l$ so auf sich selbst falten, dass die Faltlinie durch $P$ verläuft. Dies ergibt die Tangente an $K(M,AB)$ in $P$.\\
Die euklidische Grundkonstruktion 3, der Schnitt zweier Geraden, ist identisch zum Origami-Axiom 1.\\
Der Schnitt eines Kreises mit einer Geraden, die euklidische Grundkonstruktion 4, entspricht, wie bereits dargestellt, dem Origami-Axiom 5.\\

Die euklidische Grundkonstruktion 5, der Schnitt zweier Kreise, macht die größten Probleme, da Kreise, wie bereits gesagt, nur durch die Kenntnis spezieller Punkte und Tangenten zugänglich sind. Daher verwundert es auch nicht, dass diese Konstruktion nicht direkt möglich ist, sie lässt sich jedoch auf das Problem der euklidischen Grundkonstruktion 4 zurückführen. Man nehme hierzu an, dass die beiden Kreise, deren Schnittpunkte man bestimmen will, gegeben sind. Die Entfernung ihrer Mittelpunkte sei $a$, die Radii seien $b$ und $c$. OBdA kann angenommen werden, dass sich der eine Mittelpunkt im Ursprung des karthesischen Koordinatensystems befindet und sich der Mittelpunkt des zweiten Kreises bei $(a,0)$ befindet. Die Kreisgleichungen sind dann $K_1: x^2+y^2=c^2$ bzw. $K_2: (x-a)^2+y^2=b^2$. 
\begin{figure}[h]
\includegraphics[width=5cm]{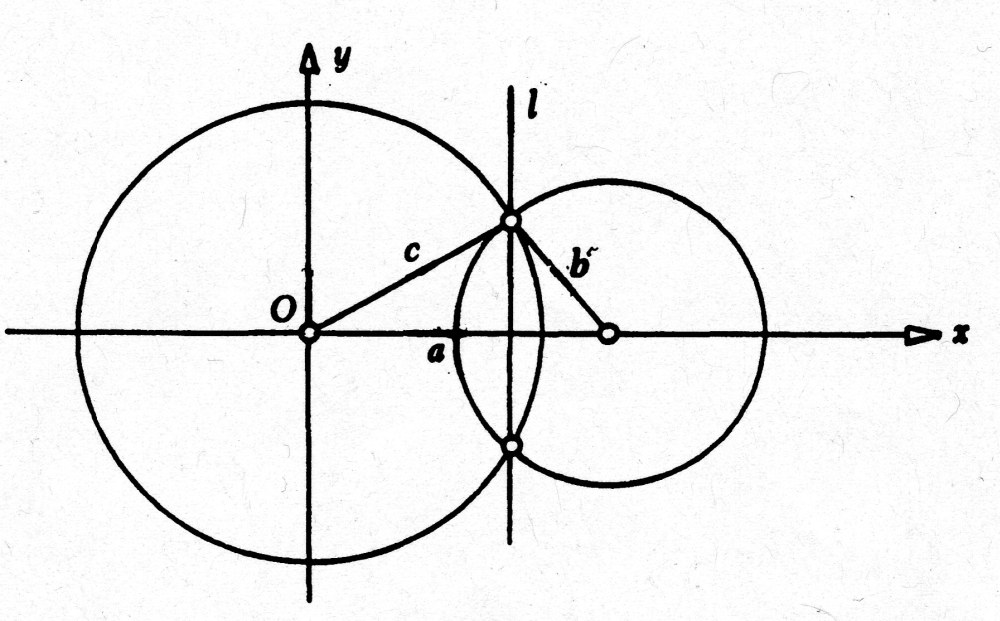}
\caption[\cite{ger95}, Abschnitt 4]{Schnitt zweier Kreise 1}
\label{zul2}
\end{figure}

Gleichsetzen der beiden Kreisgleichungen liefert:
\begin{align*}
x^2+y^2-c^2 &= x^2-2xa+a^2+y^2-b^2\\
2xa &= a^2-b^2+c^2\\
x &= \frac{a^2-b^2+c^2}{2a}
\end{align*}

Die gemeinsamen Punkte der beiden Kreise liegen also schon einmal auf der Geraden, die durch diese Gleichung gegeben ist, es handelt sich um die Gerade parallel zur y-Achse mit Abstand $\frac{a^2-b^2+c^2}{2a}$. Der Schnitt dieser Gerade mit einem der beiden Kreise ergibt die gesuchten Schnittpunkte, und diese Gerade lässt sich mittels Origami-Konstruktionen finden. Eine Möglichkeit hierzu sei nun beschrieben.
\begin{figure}[h]
\includegraphics[width=10cm]{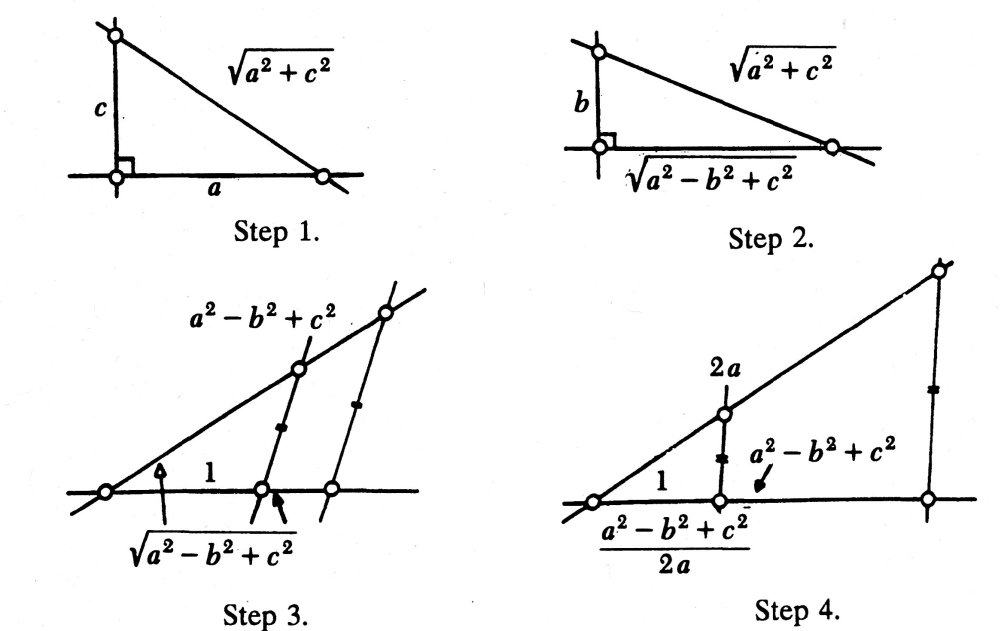}
\caption[\cite{ger95}, Abschnitt 4]{Schnitt zweier Kreise 2}
\label{zul2}
\end{figure}
\be
\item Die Entfernungen $a,c$ sind bekannt, also lässt sich ein rechtwinkliges Dreieck mit den Katheten $a$ und $c$ falten, hierzu benötigt man die Axiome 3, 4 und Korollar 1. Die Länge der Hypotenuse ist dann $\sqrt{a^2+b^2}$\\
\item Die Entfernungen $b$ und $\sqrt{a^2+b^2}$ sind nun bekannt, sodass es möglich ist, ein rechtwinkliges Dreieck mit diesen beiden Katheten zu falten (Axiome 3,4, 5 und Korollar 1). Seine Hypotenuse hat dann die Länge $\sqrt{a^2-b^2+c^2}$\\
\item Damit lässt sich nun ein Dreieck falten, bei dem eine Seite die Länge 1 und eine Seite die Länge $\sqrt{a^2-b^2+c^2}$ hat. Dann kann ein dazu ähnliches Dreieck gefaltet werden (Axiom 4), bei dem eine Seite mit Länge $\sqrt{a^2-b^2+c^2}$ der Seite mit Länge 1 im Ursprungsdreieck. Die Seite, die der Seite mit Länge $\sqrt{a^2-b^2+c^2}$ im ursprünglichen Dreieck entspricht, hat nun nach dem Strahlensatz die Länge $a^2-b^2+c^2$\\
\item  Nun kann man, mit Origami-Axiom 4, ein Dreieck falten, bei dem eine Seite die Länge $2a$ und eine Seite die Länge $a^2-b^2+c^2$ hat. Nun kann man ein dazu ähnliches Dreieck falten, bei dem eine Seite mit Länge 1 der Seite mit Länge $2a$ im ursprünglichen Dreieck entspricht (Axiome 3, 4). Dann hat die Seite, die im ursprünglichen Dreieck der Seite mit Länge $a^2-b^2+c^2$ entspricht, nach dem Strahlensatz die Länge $\frac{a^2-b^2+c^2}{2a}$, und das ist genau die Entfernung unserer Geraden zum Ursprung. \\
\ee

Nun lässt sich ein Punkt auf der x-Achse konstruieren, der diesen Abstand von 0 hat, und leicht das Lot durch diesen Punkt, indem man die x-Achse auf sich selbst faltet, sodass die Faltlinie durch diesen Punkt geht. Damit haben wir die Gerade, auf der die Schnittpunkte liegen müssen, erhalten, durch Schnitt dieser Gerade mit einem der beiden Kreise entstehen die Schnittpunkte der beiden Kreise.\\

Wir haben damit bewiesen:

\begin{thm}
Jede euklidische Grundkonstruktion kann durch elementare Origami-Konstruktionen reproduziert werden. Die euklidischen Grundkonstruktionen aus der Einleitung lassen sich alle durch Kombinationen der Origami-Axiome ersetzen.\\
\end{thm}

\begin{cor}
Das bedeutet insbesondere, dass die Menge der mit Origami konstruierbaren Zahlen den Körper $\E$ enthält.
\end{cor}

Wir wissen nun schon, dass die Origami-Zahlen mindestens die Menge der ZuL-Zahlen sind. Tatsächlich aber ist dem aufmerksamen Leser sicher aufgefallen, dass das letzte Origami-Axiom bisher noch nicht verwendet wurde, es ist also anzunehmen, dass die Menge $\Or$ noch größer wird. Dies ist tatsächlich der Fall. Im Laufe des Abschnittes wird noch gezeigt, dass die Menge der mit Origami konstruierbaren Zahlen der Menge gleich ist, die man unter Verwendung eines markierten Lineals, auch als Einschiebe-Lineal bezeichnet, erreicht.

\subsection{Das markierte Lineal}\footnote{nach \cite{martin}, Kapitel 9}

Es sei noch einmal daran erinnert, dass ein Lineal im Sinne Euklids eines ist, das keinerlei Markierungen besitzt, sondern nur die Möglichkeit bietet, zwei gegebene Punkte zu verbinden. Das markierte Lineal ist nun eines, auf dem 2 Markierungen angebracht sind, die wir als die Länge einer Einheitsstrecke auffassen können. Im Laufe dieses Unterabschnittes werden diese beiden Punkte $R$ und $S$ genannt, und damit $RS=1$. Natürlich bietet es sich nun als Grundkonstruktion an, auf Geraden Einheitsstrecken abzutragen. Weiterhin ist es möglich, eine Markierung des Lineals auf einem bestimmten, gegebenen Punkt zu halten und dann das Lineal um diesen Punkt zu drehen, bis die zweite Markierung eine bestimmte Gerade berührt, falls denn der Einheitskreis um den Punkt $R$ diese Gerade schneidet. Damit haben wir also neben dem Lineal im euklidischen Sinne auch noch den Einheitskreis gegeben. Die Menge der hiermit konstruierbaren Zahlen werden als LuK-Punkte bezeichnet:

\begin{defn}{LuK-Punkt}\\
Man startet mit den Punkten $(0,0),(2,0)$ und $(0,2)$. Ein Punkt der karthesischen Ebene ist ein LuK-Punkt, wenn der Punkt der letzte einer endlichen Abfolge $P_1, \dots, P_n$ von Punkten ist, so, dass er entweder jeder Punkt aus der Startmenge ist, oder aber in einer der folgenden Arten entsteht:
\bi
\item Schnittpunkt zweier Geraden, die durch bereits vorher auftretende Punkte definiert sind.
\item Als Schnittpunkt einer solchen Geraden mit dem Kreis um $(0,0)$ durch $(0,1)$
\ei
Eine LuK-Gerade ist eine, die durch zwei LuK-Punkte geht.\footnote{\cite{martin}, Seite 98f}\\
Eine Zahl $x$ ist genau dann eine LuK-Zahl, wenn $(x,0)$ ein LuK-Punkt ist.
\end{defn} 

Eine Aussage über diese Zahlenmenge macht der Satz von Poncelet-Steiner:

\begin{thm}{Poncelet-Steiner}\\
Ein Punkt ist ein ZuL-Punkt genau dann, wenn er ein LuK-Punkt ist.\footnote{ebd}
\end{thm}

\begin{proof}
Es ist klar, dass jeder LuK-Punkt ein ZuL-Punkt ist. Weiterhin lässt sich leicht zeigen, dass die Menge der nur mit markiertem Lineal konstruierbaren Zahlen der Körper $\Q$ ist. Noch zu zeigen bleibt, dass mit jeder Zahl aus $\Q^+$ auch die Quadratwurzel eine LuK-Zahl ist.
\begin{figure}[h]
\includegraphics[width=5cm]{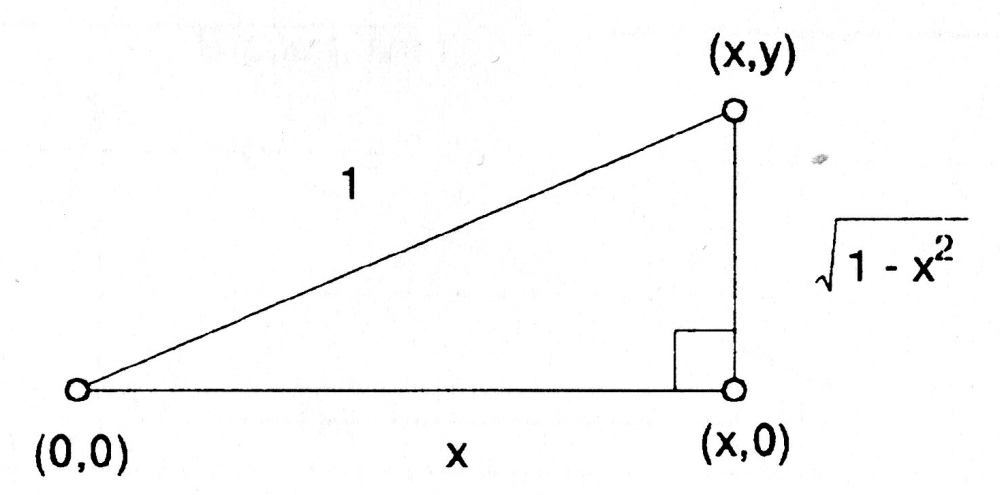}
\caption[\cite{martin}, Seite 99]{LuK-Zahlen}
\label{zul2}
\end{figure}
Die Zeichnung verdeutlicht, dass mit einer LuK-Zahl $x$ mit $-1<x<1$ auch die Zahl $\sqrt{1-x^2}$ eine LuK-Zahl ist. Es gilt aber folgende Identität, falls $z$ positiv ist:

\[
\sqrt{z} = \sqrt{\left( \frac{z+1}{2}\right)^2-\left( \frac{z-1}{2}\right)^2}=\frac{z+1}{2} \cdot \sqrt{1-\left( \frac{z-1}{z+1}\right)^2}
\]

Da aber der Bruch $\frac{z-1}{z+1}$ für positive $z$ zwischen $-1$ und $1$ liegt, beweist diese Identität, dass, für positive LuK-Zahlen $z$ auch die Wurzel in $z$ ist! Also ist auch die Gegenrichtung bewiesen.\footnote{ebd}
\end{proof}
Nun folgt leicht:
\begin{cor}
Mit einem markierten Lineal sind alle euklidischen Konstruktionen möglich.
\end{cor}

Man kann also, salopp gesprochen, seinen Zirkel wegwerfen, wenn man nur ein markiertes Lineal hat. Die charakteristische Konstruktion mit einem markierten Lineal ist aber das "`Einschieben"', worunter man folgendes versteht:

\begin{defn}{Einschiebe-Konstruktion}\\
Gegeben sei ein Punkt $V$ sowie 2 verschiedene Geraden $r$ und $s$. Dann kann man damit eine Gerade $v$ so zeichnen, dass diese $V$ enthält, $R$ auf $r$ und $S$ auf $s$ liegen. Diese Gerade $v$ schneidet die Geraden $r$ und $s$ also so, dass die Schnittpunkte Abstand 1 haben. Diese so entstehende Gerade wird als $e(r,s;V)$ bezeichnet
\end{defn}

Diese Konstruktion ist es, die das Arbeiten mit einem markierten Lineal so fundamental von den euklidischen Konstruktionen unterscheidet, die Menge der damit konstruierbaren Zahlen ist echt größer $\E$ in dem Sinne, dass sie Zahlen enthält, die nicht in $\E$ liegen.

\begin{defn}{ML-Punkt}\\
Ein Punkt der karthesischen Ebene ist mit dem markierten Lineal konstruierbar (Kurz: ML-Punkt) genau dann, wenn er der letzte einer endlichen Abfolge von Punkten $P_1, \dots, P_n$ ist, so dass jeder dieser Punkte aus $\{(0,0),(0,1),(1,0)\}$ ist, oder aber erhalten wird:
\bi
\item Als Schnittpunkt zweier Geraden, die durch Punkte definiert sind, die vorher auftragen, oder
\item Als einer von zwei Punkten, die eine Einheit voneinander entfernt sind und kollinear mit einem vorher aufgetretenen Punkt liegen, und die auf zwei verschiedenen Geraden durch vorher auftretende Punkte liegen.
\ei
Eine ML-Gerade ist eine durch zwei ML-Punkte.\\
Ein ML-Kreis ist ein Kreis durch einen ML-Punkt mit einem ML-Punkt als Mittelpunkt.\\
Eine Zahl $x$ ist eine ML-Zahl genau dann, wenn $(x,0)$ ein ML-Punkt ist.\\
Die Menge der ML-Zahlen sei mit $\Lin$ abgekürzt.
\end{defn}

Diese Definition impliziert folgende Aussagen:
\bi
\item Der Schnittpunkt zweier ML-Geraden ist ein ML-Punkt.
\item Wenn $R$ und $S$ auf verschiedenen ML-Geraden liegen und Abstand $1$ haben und mit einem ML-Punkt kollinear liegen, dann sind $R$ und $S$ ML-Punkte.
\ei

Nach dem Satz von Poncelet-Steiner ist schon klar, dass die Menge $\Lin$ den Körper $\E$ enthält. Das heißt also insbesondere, dass aus $x\in \Lin$ folgt, dass auch$\sqrt{x}\in \Lin$. Nun wir das Einschieben etwas algebraischer betrachtet:

\begin{figure}[h]
\includegraphics[width=8cm]{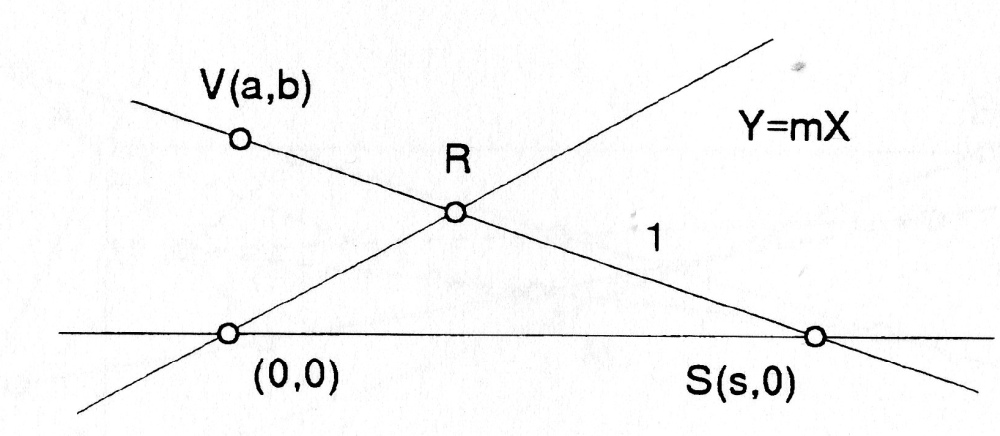}
\caption[\cite{martin}, Seite 125]{Einschieben 1}
\label{zul2}
\end{figure}

Wir schieben hier also durch die gegebenen Geraden $y=0$ und $y=mx$ und den Punkt $V=(a,b)$ ein. Ziel ist, den Punkt $S=(s,0)$ in Abhängigkeit der gegebenen Größen $a,b,m$ darzustellen. Seien die Koordinaten von $R=(x_0,y_0)$. Dann ist offenbar $y_0=m\cdot x_0$. $x_0$ ergibt sich durch das Gleichsetzen der beiden Geraden $y=m\cdot x$ und $y=\frac{b}{a-s}\cdot x - \frac{bs}{a-s}$ zu $x_0=\frac{bs}{ms-ma+b}$. $R$ muss aber auch auf dem Kreis um $S$ mit Radius $RS=1$ liegen, also $(x_0-s)^2+y_0^2=1$. Ineinander einsetzten gibt, nach einigen Umformungen, eine Gleichung in $s$ vom Grad 4:
\[
m^2s^4-2ams^3+(a^2m^2+b^2m^2-m^2)s^2+(2am^2+2bm)s-a^2m+2abm-b^2=0
\]

Es ist also offensichtlich möglich, mit der Einschiebekonstruktion bestimmte Gleichungen 4. Grades zu lösen! Das ist anschaulich betrachtet auch keine Überraschung, denn in der allgemeinen Situation gibt es 4 Möglichkeiten für $e(r,s;V)$:

\begin{figure}[h]
\includegraphics[width=8cm]{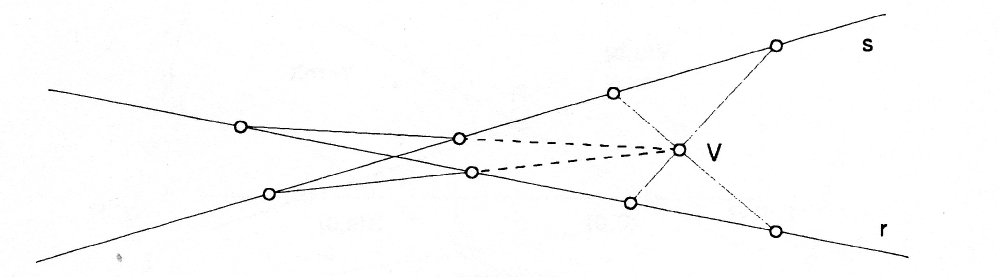}
\caption[\cite{martin}, Seite 126]{Einschieben 2}
\label{zul2}
\end{figure}

Hieraus folgt insbesondere, dass ML-Punkte solche sind, die Lösungen von Polynomen vom Grad höchstens 4 und mit Koeffizienten in $\Lin$ sind. Aber welche Gleichungen vom Grad 4 kann man nun mit der Einschiebekonstruktion lösen? Es soll gezeigt werden, dass mit Hilfe des markierten Lineals alle Aufgaben gelöst werden können, die, in die Sprache der Algebra übersetzt, das Lösen eines Polynoms 3. oder 4. Grades mit Koeffizienten in $\Lin$ bedingen. Hierzu sind einige Vorüberlegungen notwendig. Vieta hat bereits gezeigt, dass man, wenn man Winkel dritteln und dritte Wurzeln konstruieren kann, alle Gleichungen vom Grad bis 4 lösen kann. Der Beweis in dieser Arbeit geht genau so vor. Zunächst wird eine das Einschiebelineal verwendende Methode zur Winkeldrittelung vorgestellt, die auf Archimedes zurück geht\footnote{zitiert nach \cite{henn}, Seite 66f}.

\begin{thm}{Winkeldrittelung mit markiertem Lineal}\\
Man betrachte folgende Abbildung:
\begin{figure}[h]
\includegraphics[width=8cm]{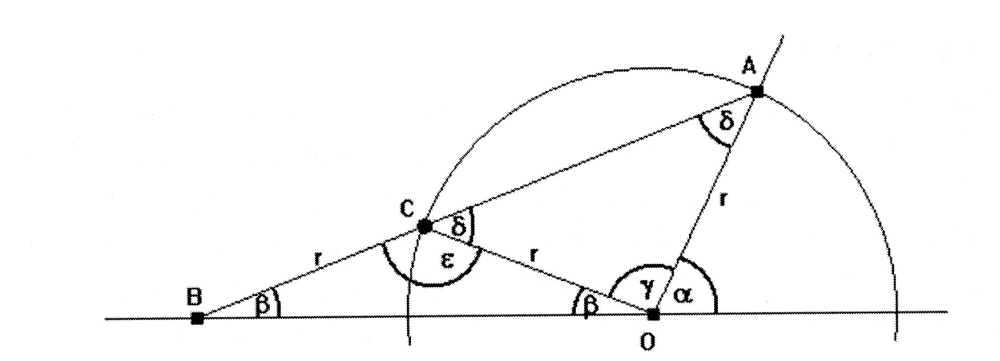}
\caption[\cite{henn}, Seite 66]{Winkeldritteln 1}
\label{zul2}
\end{figure}
Der Winkel $\alpha$ sei gegeben. Die Gerade $g$ verlängert den einen Schenkel. Man zeichnet nun den Einheitskreis um $O$. Dann lege man $S$ auf den Halbkreis und verschiebe das Lineal so, dass $R$ auf der Geraden $g$ liegt und das Lineal durch $A$ verläuft. dann misst der Winkel $\beta$ genau $\beta=\frac{1}{3}\alpha$.
\end{thm}

\begin{proof}
Zum Beweis betrachtet man eine erweiterte Skizze:

\begin{figure}[h]
\includegraphics[width=8cm]{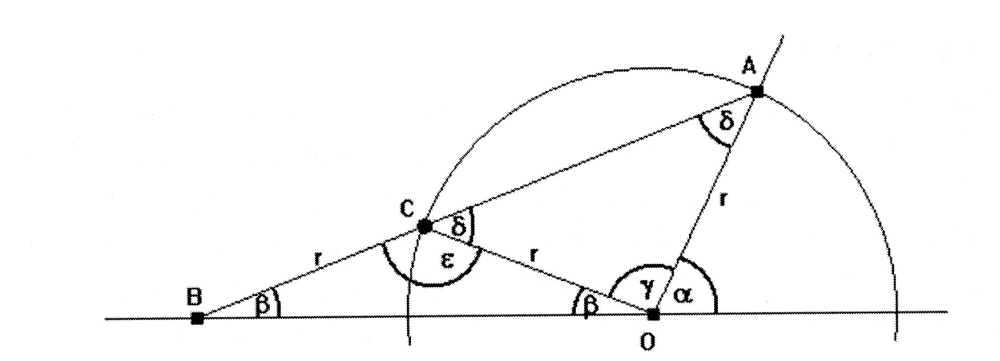}
\caption[\cite{henn}, Seite 67]{Winkeldritteln 2}
\label{zul2}
\end{figure}

In $\Delta(BOC)$ gilt $\delta=2\beta$ nach dem Außenwinkelsatz, da das Dreieck gleichschenklig nach Konstruktion ist. Genauso ist im Dreieck $\Delta(BOA)$ nach dem Außenwinkelsatz $\alpha=\beta+\delta$, also $\alpha=3\beta$ oder $\beta=\frac{1}{3}\alpha$, was zu zeigen war.
\end{proof}

\begin{cor}{Drittelungssatz}\footnote{\cite{martin}, Kapitel 9}\\
Wenn $cos(x)\in\Lin$, dann auch $cos(x/3)$.
\end{cor}

\begin{proof}
Wenn der zu drittelnde Winkel spitz ist, also zwischen 0 und 90 Grad liegt, dann folgt das direkt aus Satz 5. Wenn der Winkel ein rechter Winkel ist, ist die Aussage klar. Wenn der Winkel zwischen 90 und 180 Grad liegt, dann folgt das Resultat aus Satz 5 und der Identität $\frac{\measuredangle{\alpha}}{3}=30+\frac{\measuredangle{\alpha}-90}{3}$.\footnote{\cite{martin}, Seite 182}
\end{proof}

Nun muss man zeigen, dass man auch dritte Wurzeln konstruieren kann. Hierzu betrachtet diese Arbeit eine Konstruktion mit dem Einschiebe-Lineal, die von Nicomedes aus dem 3. Jahrhundert vor Christi Geburt stammt\footnote{\cite{martin}, Kapitel 9}.

\begin{thm}{Nicomedes}\\
Man betrachte folgende Zeichnung:\\
\begin{figure}[h]
\includegraphics[width=8cm]{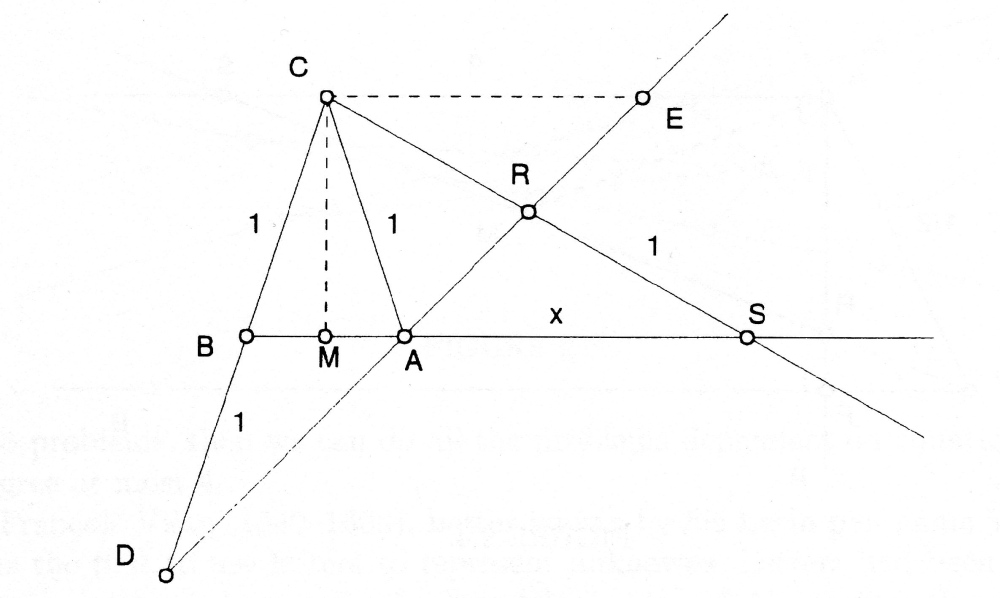}
\caption[\cite{martin}, Seite 128]{Nicomedes}
\label{zul2}
\end{figure}

Das Dreieck $\Delta(ABC)$ habe die Seitenlängen $1,1$ und $k/4$, wobei $AB=k/4$ mit $0<k<8$. Sei $B$ der Mittelpunkt von $\overline{CD}$ und schneide $\overrightarrow{CR}$ $\overrightarrow{DA}$ bei $R$ und $\overrightarrow{BA}$ bei $S$ mit $RS=1$, aber $R\ne B$, es handelt sich also mit anderen Worten um $e(\overleftrightarrow{AB},\overleftrightarrow{AD},C)$. Dann ist $AS=\sqrt[3]{k}$.
\end{thm}

\begin{proof}
Schneide die Parallele zu $\overleftrightarrow{AB}$, welche durch C geht, $\overleftrightarrow{DA}$ bei $E$, wie oben skizziert. Dann ist offenbar das Dreieck $\Delta(ABD)$ ähnlich zum Dreieck $\Delta(ECD)$. Nach Konstruktion ist $B$ der Mittelpunkt von $\overline{CD}$, und daher ist 
\[
CE=2BA=k/2
\]
 Weiterhin ist das Dreieck $\Delta(ECR)$ ähnlich zum Dreieck $\Delta(ASR)$, und daher gilt:
\[\frac{k/2}{CR}=\frac{AS}{1}\]

Sei $x=AS$, wie skizziert, dann ist
\[CR=\frac{k}{2x}\]
Weiterhin ist, da $\Delta(ACB)$ gleichschenklig ist, $M$ der Mittelpunkt von $\overline{AB}$. Dann wendet man zwei Mal den Satz den Pythagoras an, und erhält:
\begin{align*}
CS^2 &= CM^2+MS^2\\
CS^2 &= (CB^2-BM^2)+MS^2
\left( 1+\frac{k}{2x} \right)^2 &= \left[ 1^2-\left( \frac{k}{8} \right) \right] + \left( x+\frac{k}{8}\right)^2
\end{align*}

Diese Gleichung lässt sich vereinfachen zu

\[ 4x^4+kx^3-4xk-k^2=0, \]

Was sich wiederum zu 
\[(4x+k)(x^3-k)=0\]

faktorisieren lässt. Da aber $k$ und $x$ nicht negativ sind, muss $x^3-k=0$ gelten, also ist $x$ die reelle dritte Wurzel aus $k$, was zu beweisen war.

\end{proof}

\begin{cor}{Dritte-Wurzel-Satz}\\
Wenn $x\in\Lin$, dann auch $\sqrt[3]{x}$.
\end{cor}

\begin{proof}
Wenn $|x|<8$ ist, dann folgt die Aussage direkt aus Satz 6. Wenn aber $|x|\ge 8$, dann ist $x=8^n k$ für eine natürliche Zahl $n$ und eine reelle Zahl $k$ mit $|k|\le 1$ und die Aussage folgt aus dem Satz 6 und der Identität $\sqrt[3]{x}=\sqrt[3]{8^n k}= 2^n \sqrt[3]{k}$, und damit ist die Aussage für alle positiven $x$ bewiesen \footnote{\cite{martin}, Seite 182}.
\end{proof}

Zusammenfassend wissen wir nun\footnote{\cite{martin}, Kapitel 9}:

\bi
\item $\Lin$ ist ein euklidischer Körper
\item Wenn $x\in\Lin$, dann auch $\sqrt[3]{x}$.
\item Wenn $cos(x)\in\Lin$, dann auch $cos(x/3)$.
\ei

\begin{defn}{Vieta'scher Körper}\\
Ein Körper heißt abgeschlossen unter dritten Wurzeln, wenn der Dritte-Wurzel-Satz gilt. Ein Körper heißt abgeschlossen unter Dreiteilung, wenn der Drittelungssatz gilt. Ein euklidischer Körper heißt \em{vieta'scher Körper}, wenn er unter dritten Wurzeln und unter Dreiteilung abgeschlossen ist.
\end{defn}

Insbesondere ist damit $\Lin$ ein vieta'scher Körper. Nun kann man den zentralen Satz beweisen, dass $\Lin$ aus den Lösungen aller Polynome von Grad höchstens 4 mit Koeffizienten in $\Lin$ besteht. Dieser Satz liefert auch, dass $\Lin$ der kleinste vieta'sche Körper ist.

\begin{thm}{Polynome über vieta'schen Körpern}\\
Sei $V$ ein vieta'scher Körper. Wenn $x$ eine reelle Lösung eines Polynoms
\[ax^4+bx^3+cx^2+dx+e=0\]
mit Koeffizienten $a,b,c,d,e$ aus $V$ ist, dann ist $x\in V$.
\end{thm}

\begin{proof}
Wenn $a$ und $b$ gleich Null ist, folgt aus der Lösungsformel für quadratische Gleichungen, dass die Lösung aus $V$ ist, da $V$ insbesondere euklidisch ist.\\

Der nächste zu betrachtende Fall ist der, dass $a=0$, aber $B\ne 0$ ist, oBdA kann man dann $b=1$ wählen. Mittels Tschirnhaustransformation ($x=y-c/3$) lässt sich die Gleichung weiter vereinfachen zu 
\begin{align}
y^3+py+q=0
\end{align}
wobei $p=d-c^2/3$ und $q=2c^3/27 -dc/3+e$ ist, also insbesondere algebraische Ausdrücke in den Koeffizienten der ursprünglichen Gleichung $x^3+cx^2+dx+e$. $x$ ist genau dann in $V$, wenn $y$ in $V$ liegt, und man kann die ursprüngliche Gleichung genau dann lösen, wenn man diese reduzierte Gleichung lösen kann, und man kann die Lösungen für $x$ aus den Lösungen für $y$ berechnen.\\
Die erste Lösung folgt aus den Formeln von Cardano, man nutzt die Identität
\[(A+B)^3-3AB(A+B)-(A^3+B^3)=0,\]
die direkt aus dem Binomischen Lehrsatz folgt. Man setzt dann $-3AB=p$ und $-(A^3+B^3)=q$. Wenn man $A$ und $B$ so finden kann, dass sie diese beiden Gleichungen simultan lösen, dann löst $y=A+B$ die reduzierte Gleichung (3.1), denn es gilt dann:
\[0=y^3-3ABy-(A^3+B^3)=y^3+py+q.\]
Die beiden Gleichungen für $p$ und $q$ führen zu folgender Identität:
\[(z-A^3)(z-B^3)=z^2-(A^3+B^3)z+(AB)^3=z^2+qz-p^3 /27\]
Die beiden Gleichungen
\begin{align}
(z-A^3)(z-B^3) &= 0\\
z^2+qz- p^3 /27 &= 0
\end{align}
haben also die gleichen Nullstellen. (3.2) hat offensichtlich die Nullstellen $A^3$ und $B^3$. Die Gleichung (3.3) hat die Nullstellen:

\begin{align}
z_{1,2}=-\frac{q}{2} \pm \sqrt{\left( \frac{q}{2}\right)^2+ \left( \frac{p}{3}\right)^3}
\end{align}

Man kann also $A=\sqrt[3]{z_1}, B=\sqrt[3]{z_2}$ wählen und muss noch zeigen, dass nun $y_0=A+B$ eine Lösung der Gleichung $y^3+py+q=0$ ist, nach einiger Rechnung ist das der Fall. Man findet somit eine reelle Lösung für die reduzierte Gleichung (3.1). Cardano gibt weiterhin die Lösungen $\omega A+ \omega ^2 B$ und $\omega ^2 A + \omega B$ mit der dritten Einheitswurzel $\omega=-1+i \sqrt{3}$ an, welche im Allgemeinen komplex sind. Wenn aber $y_0$ eine reelle Zahl ist, dann ist sie sicher auch in $V$, weil $V$ vietaisch ist. Falls es weitere reelle Lösungen gibt, lassen diese sich durch Polynomdivision von (3.1) durch $y-y_0$ mit der quadratischen Lösungsformel finden.\\

Interessanterweise ergeben sich für den Fall, dass der Ausdruck $\left( \frac{q}{2}\right)^2+ \left( \frac{p}{3}\right)^3$ negativ ist, gerade drei reelle Lösungen. Für Cardano war das historisch ein Problem, da er die komplexen Zahlen zunächst nicht akzeptierte. In der heutigen Mathematik ist das kein Problem mehr, und tatsächlich lassen sich durch konsequente Verwendung der komplexen Zahlen die drei reellen Lösungen finden! $\left( \frac{q}{2}\right)^2+ \left( \frac{p}{3}\right)^3$ ist genau dann negativ, wenn $27q^2<-4p^3$, also in dem Fall, dass $p<0$ und $|(-q/2)\sqrt{-27/p^3}|\le 1$ ist. Also gibt es dann, bei $p<0$, ein $r$ so, dass $cos(3r)=(-q/2)\sqrt{-27/p^3}$ ist. Die von Vieta stammende Substitution $y=t\sqrt{-p/3}$ in der Gleichung (3.1) führt dann zu der Gleichung

\[t^3-3t-2 cos(3r)=0\], welche die Nullstellen $2 cos r$, $2 cos (r+120^\circ)$ und $2 cos (r+240^\circ)$ hat.\\
Wenn man also $\left( \frac{q}{2}\right)^2+ \left( \frac{p}{3}\right)^3 <0$ und $cos(3r)=(-q/2)\sqrt{-27/p^3}$ setzt, dann hat die Gleichung (3.1) die Lösungen

\begin{align}
\sqrt{-4p/3} \cdot cos (r+ k\cdot 120^\circ), ~ k=0,1,2
\end{align}

Die Quadratwurzeln sind alle reell, und, da $V$ vieta'sch ist, sind die Nullstellen in $V$. Zusammenfassend lässt sich sagen:  (3.1) hat entweder nur eine reelle Wurzel, welche die Cardano'sche Formel als $A+B$ angibt, oder drei reelle Wurzeln,  die in (3.5) dargestellt sind. Wenn $p$ und $q$ aus $V$ sind, das sind sie aber nach Konstruktion, so sind die reellen Wurzeln alle in $V$, da $V$ abgeschlossen unter Drittelung und dritten Wurzeln ist.\\

Es bleibt noch, den Fall der allgemeinen Gleichung von Grad 4, also $ax^4+bx^3+cx^2+dx+e$ für den Fall $a\ne 0$, zu untersuchen. Auch hier kann man wieder oBdA $a=1$ annehmen und mit Tschirnhaustransformation ($x=y-b/4$) zu einer reduzierten Gleichung übergehen:

\begin{align}
y^4+py^2+qy+r=0
\end{align}

Hierbei sind die Koeffizienten $p,q,r$ alle algebraisch aus den Koeffizienten $a,b,c,d,e$ bestimmbar als: $p=\frac{8c-3b^2}{8}$, $q=\frac{b^3-4bc+8d}{8}$ und $r=-\frac{3b^4-16b^2c+64bd-256e}{256}$. Für den Fall $q=0$ liegt eine quadratische Gleichung in $y^2$ vor, und die Lösungen ergeben sich mit Hilfe der quadratischen Lösungsformel sehr schnell. Andernfalls wendet man folgende Umformung an:

\begin{align}
y^4 &= -py^2 - qy -r \notag\\
y^4+zy^2+z^2/4 &= -py^2-qy-r+zy^2+z^2/4 \notag\\
(y^2+z/2)^2 &= (z-p)y^2-qy+(z^2/4-r),
\end{align}
wobei $z$ noch zu bestimmen ist. Die linke Seite der Gleichung ist ein Quadrat, und wenn man nun $z$ so wählen kann, dass die rechte Seite das Quadrat eines linearen Ausdrucks $gy+h$ in $y$ wird (binomische Formel), dann folgen die Lösungen für $y$ leicht aus der Gleichung $y^2+z/2=\pm (gy+h)$. Damit ist das Ziel klar. Man bestimmt $z$ so, dass die rechte Seite ein Quadrat wird, und das ist genau dann der Fall, wenn die Nullstellen der quadratischen Gleichung in $y$
\[(z-p)y^2 + (-q)y +(z^2/4-r) = 0\]
zusammenfallen. Nach der quadratischen Lösungsformel passiert das genau dann, wenn $(-q)^2-4(z-p)(z^2/4-r)=0$ gilt, was zu einer Gleichung vom Grad 3 in $z$ führt:
\begin{align}
z^3-pz^2-4rz-(q^2-4pr) =0
\end{align}
Damit ist die Lösungsmethode auf das Problem des Lösens einer Gleichung vom Grad 3 zurückgeführt. Man bestimmt also zunächst eine Lösung $z\ge p$ von (3.8) mit den Methoden, die für die Gleichung vom Grad 3 oben vorgestellt wurden, und löst anschließend die biquadratische Gleichung (3.7). Das Zurückführen der Lösung der Gleichung von Grad 4 auf die vom Grad 3 ergibt insbesondere, dass dann alle Lösungen der Gleichung von Grad 4 in $V$ liegen.
\end{proof}

Nun folgt ganz leicht:

\begin{cor}{Hauptsatz über ML-Zahlen}\\
Sei $x$ reelle Nullstelle eines Polynoms vom Grad höchstens 4, und die Koeffizienten des Polynoms sind in $\Lin$, dann ist $x\in\Lin$.
\end{cor}

Eine Formulierung dieser Tatsache in der Sprache der Körpertheorie könnte so lauten:

\begin{thm}{Ml-Zahlen, Körpertheorie}\\
Eine Zahl $z$ ist eine ML-Zahl genau dann, wenn sie in einer endlichen Körperkette 
\[\Q = F_0 \subset F_1 \subset \dots \subset F_n\]
über $\Q$ liegt, so dass $z\in F_n$ und $[F_i:F_{i-1}]= 2$ oder $3$ für alle $1\le i\le n$.\\
Das heißt, $F_n$ hat dann Körpergrad $[F_n:Q]=2^a3^b$ für natürliche Zahlen $a,b$.\footnote{\cite{cox}, Seite 277f}
\end{thm}

Diese Arbeit ist nun einen gewissen Umweg über das markierte Lineal gegangen und hat dort, in einer recht etablierten Theorie, bewiesen, welche Zahlen mit dem markierten Lineal konstruierbar sind. Nun kommt sie wieder zum Origami zurück. Es wird bewiesen, dass der Körper $\Or$ gleich dem Körper $\Lin$ ist. Dies ist zumindest schon zu vermuten, da, wie bereits gesagt, das letzte Origami-Axiom gerade die Konstruktion mit dem Einschiebe-Lineal simuliert. Damit wäre dann die Frage, welche Zahlen mit Origami konstruierbar sind, beantwortet.

\subsection{Die Origami-Zahlen und das markierte Lineal}\footnote{dieses Kapitel folgt\cite{martin}, Kapitel 10}

Origami ist, kann man sagen, auf einer einzigen Idee aufgebaut, nämlich auf der Idee von Symmetrien und Spiegelungen. Zunächst sollte daher etwas vertiefter auf die Geometrie und Algebra von Spiegelungen eingegangen werden, alle diese Resultate sind für die weiteren Überlegungen unabdingbar.

\begin{defn}{Spiegelung}\\
Gegeben eine Gerade $g$. Der Bildpunkt $r_g(P)$ eines Punktes $P$ bei Spiegelung an der Geraden $g$ ist $P$, falls $P$ auf $g$ liegt, oder, falls $P$ nicht auf $g$ liegt, ein Punkt $Q$, der nicht auf $g$ liegt, und zwar so, dass $g$ die Mittelsenkrechte von $\overline{PQ}$ ist.\\
Abkürzend wird auch $P^g=r_g(P)$ bzw., für eine Gerade $h$, $h^g=r_g(h)=\{r_g(P):P\in h\}$.
\end{defn}

\begin{lem}
In der kartesischen Ebene ist der Spiegelpunkt von $P=(x,y)$ an der Geraden $g:ax+by+c=0$ gegeben durch:

\[P'=(x',y')=\left( x-\frac {2a(ax+by+c)}{a^2+b^2},y-\frac{2b(ax+by+c}{a^2+b^2}\right)\]
\end{lem}
\begin{proof}
Wenn $P$ auf $g$ liegt, so stimmt die Gleichung, da dann die jeweils abgezogenen Brüche 0 werden. Sei also $P$ außerhalb von $g$ gelegen. Dann ist $g$ per Definition die Mittelsenkrechte von $P$ und $P'$. Also liegt der Mittelpunkt von $P$ und $P'$ auf $g$, erfüllt also die Geradengleichung: $a[(x+x')/2]+b[(y+y')/2]+c=0$. Weiterhin ist die Gerade $g$ senkrecht zu $\overleftrightarrow{PP'}$: $a(y'-y)=b(x'-x)$. Das lässt sich in ein Gleichungssystem aus 2 Gleichungen mit den beiden Unbekannten $x'$ und $y'$ schreiben.

\begin{align*}
ax' + by' &= -2c-ax-by\\
bx' - ay' &= bx-ay
\end{align*}

Das Lösen dieses Gleichungssystems ergibt die behauptete Darstellung.
\end{proof}

Was sind denn nun die Geraden $t$ so, dass das Bild eines gegebenen Punktes $P$ unter der Spiegelung an $t$ auf einer bestimmten Geraden $p$ liegen? Der Leser möge sich daran erinnern, dass diese Frage hinter Origami-Axiom 5 steckt! Dazu erinnere man sich kurz daran, wie eine Parabel definiert war:

\begin{defn}{Parabel}\\
Eine Parabel ist die Menge aller Punkte, die von einem gegebenen Punkt $P$, ihrem Brennpunkt, und einer Geraden $p$, genannt ihrer Leitgeraden, den gleichen Abstand haben.
\end{defn}

Das führt nun zu folgender Charakterisierung dieser Geraden:

\begin{thm}{Parabel und Spiegelung}\\
\be
\item Sei $P$ auf der Geraden $p$, dann ist $P^t$ genau dann auf $p$, wenn $t$ durch $P$ geht oder $t$ senkrecht zu $p$ steht.\\
\item Sei $P$ nicht auf der Geraden $p$, dann ist $P^t$ genau dann auf $p$, wenn $t$ eine Tangente der Parabel mit Brennpunkt $P$ und Leitgeraden $p$ ist.
\ee
\end{thm}

\begin{proof}
Der erste Teil des Satzes folgt direkt aus der Definition einer Spiegelung.\\
Angenommen, $P$ liegt nicht auf $p$ und $P^t=Q$ mit $Q\in p$. Nach Definition einer Spiegelung an $t$ ist $t$ Mittelsenkrechte von $\overline{PQ}$. Nachdem $P$ nicht auf $p$ liegt, ist $p \perp t$ nicht möglich. Das Lot zu $p$ durch $Q$ schneidet $t$ in einem Punkt $T$ mit $TP=TQ$. $T$ hat also den gleichen Abstand zu $P$ und zu $p$, liegt also auf der Parabel mit Brennpunkt $P$ und Leitgeraden $p$. $t$ enthält aber auch nur einen Punkt dieser Parabel, denn:\\
Angenommen, $S$ sei ein weiterer Punkt auf dieser Parabel, und sei $F$ der Lotfußpunkt des Lotes $p$ durch $S$. Dann ist $F\ne Q$, aber, da $S$ auf $t$, dem Ort aller Punkte, die von $P$ und $Q$ gleich weit weg sind, liegt, ist $SQ=SP$. Weiterhin ist aber auch $SP=SF$, da $S$ auf der Parabel liegt, als Ort aller Punkte mit gleichem Abstand zu $P$ und $p$. Also muss auch $SQ=SF$ sein, und das führt dazu, dass das Dreieck $\Delta(FSQ)$ gleichschenklig sein muss, mit einem rechten Basiswinkel. Das geht aber nicht. Widerspruch. $t$ hat also genau einen Punkt mit der Parabel gemeinsam und ist somit Tangente der Parabel.\\
Sei umgekehrt $t$ Tangente der Parabel mit Brennpunkt $P$ und Leitgeraden $p$, und schneide $t$ die Gerade bei $T$. Sei weiter $Q$ der Fußpunkt des Lotes zu $p$ durch $T$, und sei $n$ die Mittelsenkrechte zu $\overline{PQ}$. Also ist $Q=P^n$. $T$ liegt auf der Parabel, also ist $TP=TQ$ und $T$ liegt somit auf $n$. Wie oben ist somit $n$ Tangente an die Parabel, diese kann aber nur eine Tangente in einem Punkt haben. Man muss also schließen, dass $n=t$ war, und somit $Q=P^t$ ist.
\end{proof}
Man betrachte nun nochmals das Axiom 6: Gegeben seien jeweils verschiedene Punkte $P$ und $Q$ und Geraden $p$ und $q$, so dass nicht gleichzeitig $P\in p, Q \in q$ und $p \parallel q$ gilt. Dann kann man (entlang einer Geraden $t$) so falten, dass $P$ auf $p$ und $Q$ auf $q$ zu liegen kommt. Es gibt nun verschiedenen Möglichkeiten der Lage dieser 4 Objekte:\\
Angenommen, dass $P$ auf $p$ liegt, $Q$ auf $q$, aber $p\not \parallel q$. Die Lösung für Geraden $t$ sind das Lot zu $q$ durch $P$, das Lot zu $p$ durch $Q$, und die Verbindungsgerade von $P$ und $Q$. Diese sind im schlimmsten Falle alle verschieden und jede davon ist algebraisch bestimmt durch eine höchstens quadratische Gleichung.\\
Nehmen wir oBdA $P\notin Q$ und $p \not\parallel q$ an. Nachdem alle Parabeln ähnlich sind, kann man weiter annehmen, dass $P=(0,0)$ und $p:y=2$, $q$ ist dann allgemein definiert durch eine Gleichung $q: x+ry+s=0$, da die Geraden nicht parallel sein dürfen. Sei $Q=(u,v)$. Da $t$ nicht senkrecht zu $p$ sein kann, wenn $P^t$ auf $p$ ein soll, muss $t$ eine Gleichung der Form $t: y=mx+b$ erfüllen. Die Forderung $P^t\in p$ beinhaltet dann, dass $y'=y-2(-1)(mx-y+b)/(m^2+(-1)^2)$ mit $y'=2$ und $x=y=0$ sein muss. Die Gleichung reduziert sich dann zu $b=m^2+$. Da $Q^t=(u',v')\in q$ sein soll, muss $u'+rv'+s=0$ sein. Nach Lemma 1 ergibt sich dann die Gleichung:
\[\left[u-2m\frac{mu-v+(m^2+1)}{m^2+1}\right] + r \left[v+2\frac{mu-v+(m^2+1)}{m^2+1}\right]+s=0\]
Es handelt sich um eine Gleichung von Grad $3$ in $m$ mit Koeffizienten, die rationale Ausdrücke aus $r,s,u,v$ sind. Es gibt sicher höchstens 3 Lösungen für $m$, und eine Lösung für $m$ bestimmt eindeutig eine Lösung für $b$ und somit eindeutig die Gerade $t$, die dann eine der Tangenten der Parabel mit Brennpunkt $P$ und Leitgerade $p$ ist.\\
Wenn zusätzlich noch $Q\notin q$ gilt, dann ist $t$ auch Tangente an die Parabel mit Brennpunkt $Q$ und Leitgerade $q$! Man hat dann also eine {\em gemeinsame Tangente zweier Parabeln} gefunden!\\
Sei nun noch kurz der Fall betrachtet, dass $P \notin p$ und $p\parallel q$ ist. Die Geraden $t$ bestimmen sich dann aus höchstens quadratischen Gleichungen mit höchstens zwei Lösungen.\\
Fasst man zusammen, ergibt sich: In den erlauben Fällen gibt es höchstens drei verschiedene Geraden $t$, wie im Axiom gefordert. Jede dieser Geraden ist durch ein Polynom vom Grad höchstens 3 bestimmt, mit Koeffizienten in dem kleinsten Körper, der die Koordinaten von $P$ und $Q$ sowie die Schnittpunkte von $p$ und $q$ mit den Koordinatenachsen enthält. Da das markierte Lineal, wie im vorherigen Unterabschnitt gezeigt, solche Probleme lösen kann (Hauptsatz über ML-Zahlen), folgt nun leicht:

\begin{thm}{Hauptsatz über Origami-Zahlen, erste Version}\\
Jede Origami-Zahl ist eine Ml-Zahl. Also ist $\Or\subset\Lin$
\end{thm}

Überraschenderweise gilt auch die Umkehrung dieses Satzes, dieses zu zeigen ist nun die zu lösende Aufgabe. Dies ergibt sich aus folgenden Überlegungen:

Nach den Überlegungen des Unterkapitels 3.1 ist klar, dass der Körper $\E$ mit Origami konstruierbar ist. Nun muss man noch beweisen, dass die Winkeldrittelung und das Ziehen von Quadratwurzeln mit Origami-Methoden lösbar ist, um, unter Anwendung der Erkenntnisse des letzten Kapitels, sagen zu können, dass die Mengen $\Or$ und $\Lin$ identisch sind.

\begin{lem}{Winkeldrittelung durch Origami}\\
Die winkeldrittelnden Geraden eines durch Origami-Geraden gegebenen Winkels sind Origami-Geraden.
\end{lem}

\begin{proof}

\begin{figure}[h]
\includegraphics[width=8cm]{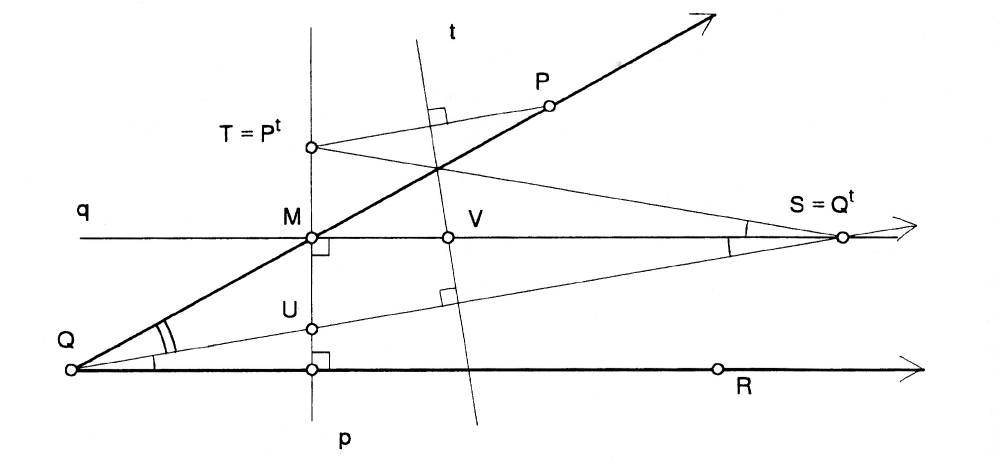}
\caption[\cite{martin}, Seite 155]{Winkeldritteln mit Origami}
\label{zul2}
\end{figure}

Seien $P,Q,R$ Origami-Punkte, die den spitzen Winkel $\angle(PQR)$ bilden, wie in der Skizze. Sei $M$ der Mittelpunkt von $\overline{PQ}$, $p$ das Lot auf $\overleftrightarrow{QR}$ durch $M$, und sei $q$ das Lot auf $p$ durch $M$. Sei $t$ diejenige Gerade $t$ mit $T=P^t\in p$ und $S=Q^t \in q$, die $\overline{PM}$ schneidet. Schneide $\overrightarrow{QS}$ die Gerade $p$ im Punkt $U$. Dann ist das Dreieck $\Delta(PMT)$ kongruent dem Dreieck $\Delta(QMU)$ nach dem WSW-Satz. Also ist $TM=MU$ und $\Delta(TMS)$ damit kongruent zu $\Delta(UMS)$ nach SWS. Da $\angle(TSM)$ kongruent zu $\angle(MSQ)$ ist, welcher selbst wieder kongruent zu $\angle(SQR)$ ist, folgt, dass der Winkel $\angle(TSQ)$ dem Winkel $\angle(PQS)$ kongruent ist (gleichschenkliges Dreieck), drittelt $\overrightarrow{QS}$ den Winkel $\angle(PQR)$.Bezeichnet man den Schnittpunkt von $t$ und $q$ mit $V$, so ist $\overrightarrow{QV}$ die zweite winkeldrittelnde Gerade.\\ Das Ergebnis folgt auch für stumpfe Winkel, da die Drittelung des rechten Winkels (Konstruktion eines $30^\circ$-Winkels) bereits mit ZuL möglich ist.
\end{proof}

Im späteren Verlauf der Arbeit wird noch eine einfach durchzuführende Konstruktion für die Winkeldrittelung explizit besprochen werden.\\

\begin{lem}{Dritte Wurzeln mit Origami}\\
Seien $O,S$ Origami-Punkte mit $OS=k$. Dann gibt es einen Origami-Punkt $R\in\overline{OS}$ so, dass $OR=\sqrt[3]{k}$.
\end{lem}

\begin{proof}
\begin{figure}[h]
\includegraphics[width=8cm]{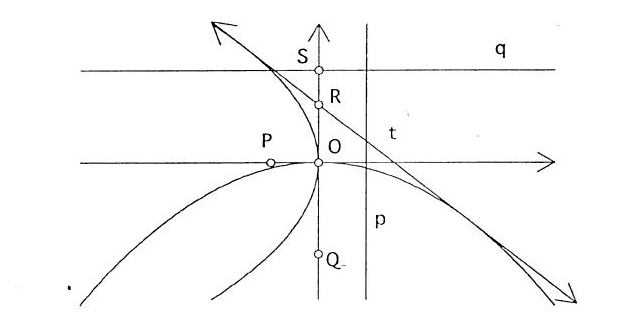}
\caption[\cite{martin}, Seite 156]{Dritte Wurzel mit Origami}
\label{zul2}
\end{figure}

OBdA sei $O=(0,0), S=(0,k)$. Seien $P=(-1,0), Q=(0,-k)$, dies sind offensichtlich Origami-Punkte. Sei $p:X=1$ und $q:Y=k$, das sind Origami-Geraden. Man betrachte nun die obige Skizze und erinnere sich an Satz 9. Da zwei Parabeln mit gemeinsamen Scheitel und Hauptachsen eine eindeutige gemeinsame Tangente haben, gibt es eine Gerade $t$ mit der Eigenschaft $P^t\in p$ und $Q^t\in q$. $t$ ist dann Origami-Gerade und schneidet die $y-Achse$ beim Origami-Punkt $R$. $t$ habe eine Gleichung $t:mx+b$, oder $t:mx-y+b=0$. Da $P^t$ auf $p$ liegt, ist seine x-Koordinate $1$.\\
 Also ist, nach Lemma 1,
\[1=-1-2m\frac{m(-1)-0+b}{m^2+1}=-mb,\]
also $mb=-1$. Weiterhin hat $Q^t$ die y-Koordinate $k$. Wieder wendet man Lemma 1 an, und erhält
\[k=-k-2(-1)\frac{m(0)-(-k)+b}{m^2+1}=\frac{b}{m^2}\]
also $km^2=b$.\\ 
Also müssen $m$ und $b$ die Bedingung $km^3=mb=-1$ erfüllen. Dann ergibt sich $m=\frac{-1}{\sqrt[3]{k}}$ und $b=\sqrt[3]{k}$. Die Gerade $t:y=\frac{-1}{\sqrt[3]{k}}\cdot x +\sqrt[3]{k}$ existiert also eindeutig, ist eine Origami-Gerade, und $R$ ist ihr Schnittpunkt mit der $y$-Achse, und dieser hat die Koordinaten $(0,\sqrt[3]{k})$.
\end{proof}

Da jede ML-Konstruktion auf eine Abfolge von quadratischen Problemen, Winkeldrittelung und das Ziehen dritter Wurzeln reduziert werden kann, ist damit gezeigt, dass jeder ML-Punkt auch ein Origami-Punkt ist:

\begin{cor}{Isomorphiesatz}\\
Ein Punkt ist ein Origami-Punkt genau dann, wenn er ein ML-Punkt ist. Die Körper $\Or$ und $\Lin$ sind isomorph.
\end{cor}

Insbesondere gelten also Korollar 6 und Satz 8 auch für die Origami-Zahlen. Die vollständige Charakterisierung aller aus dem Axiomensystem von Kapitel 2 konstruierbaren Zahlen lautet also, wie in Satz 8 fürs markierte Lineal:

\begin{thm}{Hauptsatz über Origami-Zahlen}\\
Eine Zahl $z$ ist eine Origami-Zahl genau dann, wenn sie in einer endlichen Körperkette 
\[\Q = F_0 \subset F_1 \subset \dots \subset F_n\]
über $\Q$ liegt, so dass $z\in F_n$ und $[F_i:F_{i-1}]= 2$ oder $3$ für alle $1\le i\le n$.\\
Das heißt, $F_n$ hat dann Körpergrad $[F_n:Q]=2^a3^b$ für natürliche Zahlen $a,b$.\footnote{\cite{cox}, Seite 277f}
\end{thm}

\newpage

\section{Konkrete Origami-Konstruktionen}

Nachdem nun klar ist, welche Zahlen/Punkte der Ebene mit Origami erreichbar sind, wendet sich diese Arbeit nun einigen konkreten Konstruktionen zu, denn schließlich ist es das eine, zu wissen, dass eine Konstruktion theoretisch möglich ist, aber doch nochmal etwas anderes, sie konkret durchführen zu können. Es werden nun im Folgenden einige konkrete Konstruktionen betrachtet, teilweise solche, die auch mit Zirkel und Lineal möglich sind, aber auch solche, die die volle Kraft der Origami-Axiome ausnutzen und mit Zirkel und Lineal nicht lösbar sind. 

\subsection{Der Goldene Schnitt}

Das Verhältnis des Goldenen Schnittes zieht sich nicht nur durch die Mathematik-, sondern auch durch die Kunstgeschichte und sogar die Natur, denn viele Pflanzen haben fünfeckige Blüten, und die Zahl des goldenen Schnittes ist das Verhältnis, in dem sich die Diagonalen im Fünfeck schneiden. Man kann zeigen, dass das Fünfeck genau dann mit Zirkel und Lineal konstruierbar ist, wenn das für die Teilung im Goldenen Schnitt gilt. Historisch spielte das Pentagramm als Ordenssymbol der Pythagoräer eine Rolle, und hat so wohl zur Entdeckung der Inkomensurabilität beigetragen. Eine gewisse Ironie steckt dennoch darin, dass im Ordenssymbol eines Ordens, der die Annahme, dass es zu zwei Strecken stets ein gemeinsames Maß gebe, stets verteidigt hat, ein Streckenverhältnis auftaucht, für das diese Aussage eben gerade nicht gilt. \footnote{\cite{henn}, Seiten 78ff} 

\begin{defn}{Goldener Schnitt}\footnote{ebd}\\
Der Punkt $C$ teilt die Strecke $\overline{AB}$ genau dann im Verhältnis $\Phi$ des Goldenen Schnittes, wenn sich die größere Teilstrecke zur kleineren genau so verhält wie die Gesamtstrecke zur größeren Teilstrecke. Für die Strecken $a=\overline{AC}, b=\overline{CB}$ gilt dann:
\[\frac{a}{b}=\frac{a+b}{a}\]
\begin{figure}[h]
\includegraphics[width=8cm]{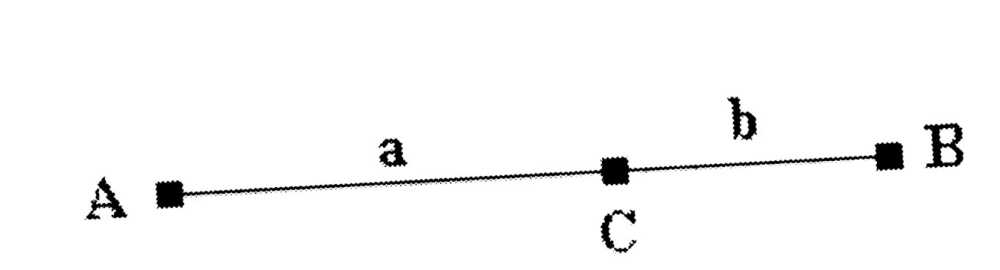}
\caption[\cite{henn}, Seite 80]{Goldener Schnitt}
\label{zul2}
\end{figure}
\end{defn}

Aus dieser Definition folgt mit $\Phi=\frac{a+b}{a}=1+\frac{b}{a}=1+\frac{1}{\Phi}$ sofort die Goldene quadratische Gleichung
\[\Phi^2-\Phi-1=0\]
mit den Nullstellen
\[\Phi=\frac{1}{2}\left(1\pm\sqrt{5}\right).\]
Da $\Phi$ eine Streckenlänge sein soll, ist 
\[\Phi=\frac{1}{2}\left(1+\sqrt{5}\right)\approx 1,618.\]
Vergleicht man hingegen $a$ mit der Gesamtstrecke $c:=a+b$, so ergibt sich 
\[\frac{a}{c}=T=\frac{1}{2}\left(\sqrt{5}-1\right)\footnote{ebd}\]

Es gelten dabei die charakteristischen Gleichungen $\Phi\cdot T=1$ und $\Phi=1+T$.\\

Da das Verhältnis des Goldenen Schnittes Lösung einer Gleichung vom Grad $2$ mit Koeffizienten aus $\Q$ ist, ist $\Phi\in\E$ und damit auch in $\Lin$ und $\Or$ enthalten, gleiches gilt natürlich für $T$.

Eine Origami-Konstruktion zur Teilung einer Quadratseite nach dem Goldenen Schnitt und damit zur Konstruktion eines Goldenen Rechteckes soll nun vorgestellt werden\footnote{\cite{flachs}, Seite 103}.
\begin{thm}{Konstruktion des Goldenen Schnittes}\\
\begin{figure}[h]
\includegraphics[width=8cm]{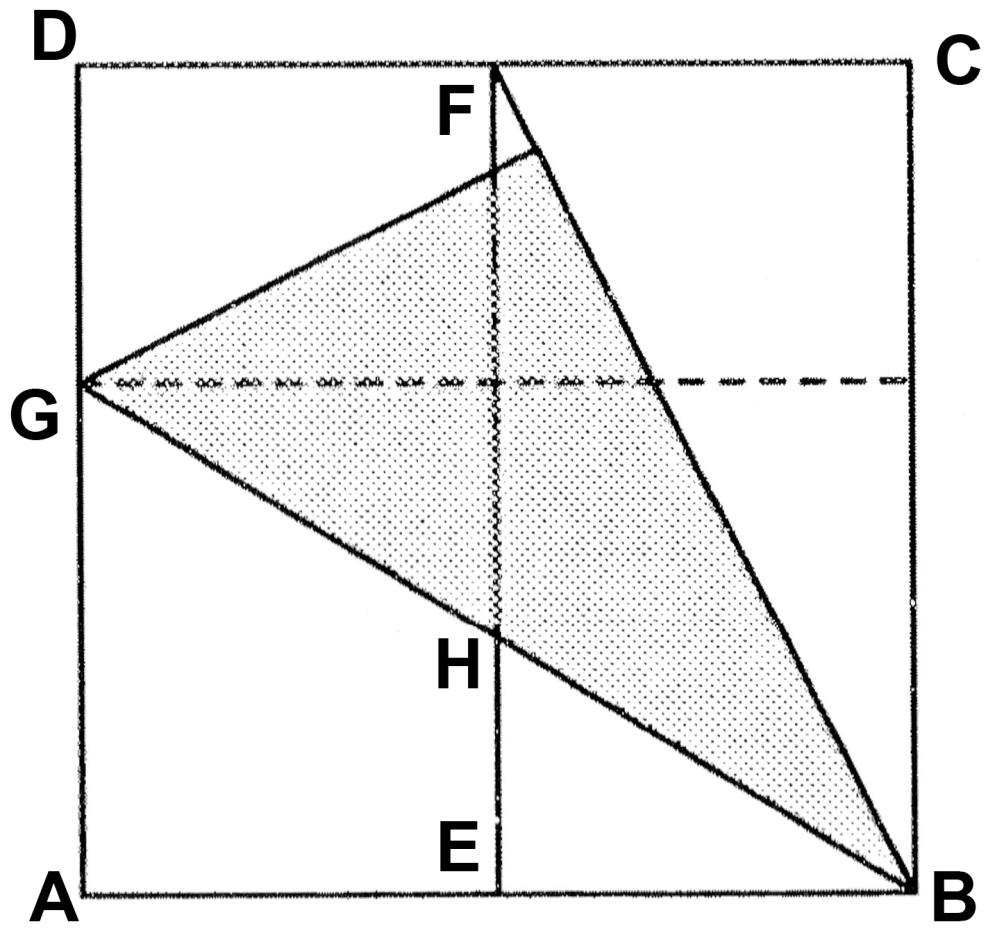}
\caption[Bearbeitung von \cite{flachs}, Seite 103]{Konstruktion Goldener Schnitt}
\label{zul2}
\end{figure}
Man faltet hierzu in ein quadratisches Blatt mit Ecken $A,B,C,D$ eine kantenparallele Mittellinie $\overline{EF}$ ein. In eines der beiden entstehenden Rechtecke faltet man eine Diagonale $\overline{BF}$ ein.Den an der Quadratseite entstehenden größeren Winkel $\beta=\angle(FBA)$ halbiert man. Die Winkelhalbierende$\overrightarrow{BG}$ teilt die Quadratkante $\overline{AD}$ im Punkt $G$ nach dem Goldenen Schnitt.
\end{thm}

\begin{proof}
Das ursprüngliche Quadrat habe oBdA die Seitenlänge 2. Man betrachte nun das Dreieck $\Delta{BEF}$. Die innere Winkelhalbierende $\overrightarrow{BH}$ des Winkels $\beta$ teilt die Seite $\overline{FE}$ im Verhältnis der an $\beta$ anliegenden Seiten $\overline{BF}, \overline{BE}$ teilt, und nach Pythagoras $BF=\sqrt{5}$ ist, gilt: 
\[\frac{EH}{FH}=\frac{1}{\sqrt{5}}\]
Nach Voraussetzung ist $FH+EH=2$, damit gilt:\footnote{\cite{walser}, Seite 31} 
\begin{align*}
\frac{EH}{2-EH} &= \frac{1}{\sqrt{5}}\\
\sqrt{5} EH &= 2-EH\\
EH &= \frac{2}{1+\sqrt{5}}\\
EH &= \frac{1}{\Phi}\\
EH &= T
\end{align*}

Nach dem Strahlensatz hat dann $AG$ die Länge $2T$, und damit $DG=2-2T$. Daher:
\[\frac{AD}{AG}=\frac{2}{2T}=\frac{1}{T}=\Phi\]
 und
\[\frac{AG}{GD}=\frac{2T}{2-2T}=\frac{1}{\frac{1}{T}-1}=\frac{1}{\Phi-1}=\frac{1}{T}=\Phi\]

und damit teilt $G$ die Quadratkante $AD$ im Goldenen Schnitt.

\end{proof}

\subsection{Das Delische Problem } \footnote{\cite{henn}, Seiten 62ff}\\

Es kann ohne Beschränkung der Allgemeinheit angenommen werden, dass der zu verdoppelnde Würfel die Seitenlänge $1$ und somit das Volumen $V=1^3=1$ hat. Die Kantenlänge $z$ des Würfels mit dem doppelten Volumen müsste also der Gleichung
\[z^3-2=0\]
genügen. Dieses Polynom ist aber irreduzibel über $\Q$, denn, wäre es reduzibel, so gäbe es eine Zerlegung der Art:
\[z^3-2=(z-a)(z^2+bz+c)=z^3+(b-a)z^2+(c-ab)z-ac \] mit $a,b,c,d\in\Q$. Koeffizientenvergleich liefert dann $b=a, c=ab, ac=2$ und damit $2=a^3$. $a$ muss dann als vollständig gekürzter Bruch eine Darstellung $a^3=\frac{p^3}{q^3}$ haben, mit $p,q$ teilerfremd, also gilt $2q^3=p^3$. Weil $p,q$ teilerfremd angenommen waren, muss nun aber gelten: $2|p$, also $p^3=2^3P^3$ für ein passendes $P$ und damit $2q^3=2^3P^3$ bzw. $q^3=2^2P^3$. Dann muss aber $2$ auch $q$ teilen, was nicht geht, weil $p,q$ teilerfremd angenommen waren. \footnote{http://www.geometrie-und-algebra.de; Lösungshinweise} \\
Der Körper $\Q(z)$ ist also eine Erweiterung vom Grad 3 über $\Q$, und damit ist die Konstruktion zwar nicht mit Zirkel und Lineal, nach dem Hauptsatz über Origami-Zahlen aber wohl mit Origami möglich. Eine konkrete Konstruktion hierfür basiert auf dem bereits bewiesenen Satz von Haga und ist folgender Skizze zu entnehmen:

\begin{thm}{Lösung des Delischen Problems mit Origami}\\
\begin{figure}[h]
\includegraphics[width=8cm]{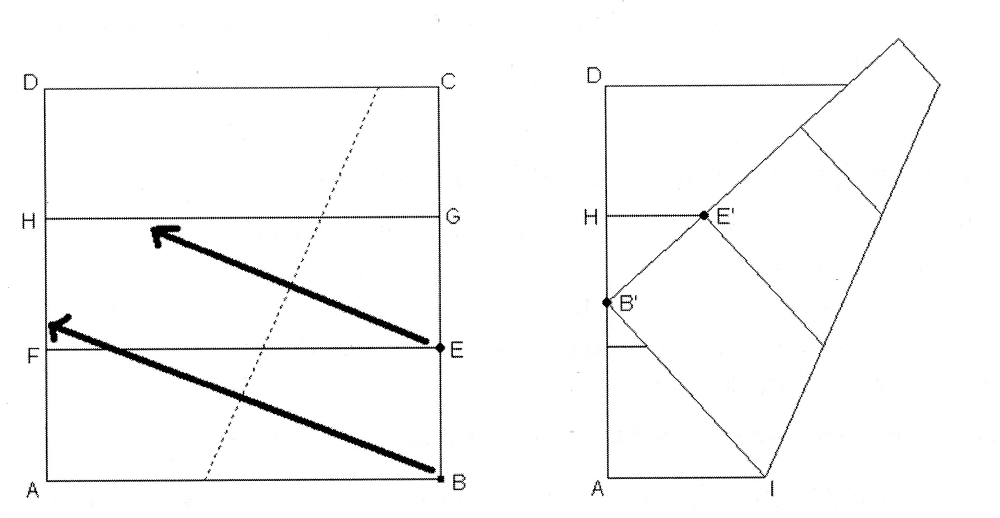}
\caption[\cite{henn}, Seite 63]{Delisches Problem}
\label{zul2}
\end{figure}
Man beginnt hierzu damit, das quadratische Papier mit Bezeichnungen wie in Abbildung 4.3, nach dem Satz von Haga zu dritteln. Danach faltet man $B$ auf $B'\in\overline{AD}$ und gleichzeitig $E$ auf $E'\in\overline{HG}$. Der rechte Teil der Skizze zeigt das Faltergebnis, der Punkt $B'$ teilt nun die Seite $\overline{AD}$ im gewünschten Verhältnis $\sqrt[3]{2}$.
\end{thm}

\begin{proof}
Die Seitenlänge des Ausgangsquadrates kann oBdA als $3$ angenommen werden, dies vereinfacht die nun folgende Rechnung etwas. Dann ist $B'E'=1$ nach dem Satz von Haga. Setze weiterhin $x:=\overline{DB'}, y:=\overline{AB'}, a:=\overline{AI}$ und $b:=\overline{IB'}$, wie in der Zeichnung, dann gilt: $x+y=a+b=3$. Die Dreiecke $\Delta(AIB')$ und $\Delta(B'E'H)$ sind dann ähnlich, und es gilt:
\[\frac{a}{b}=\frac{x-1}{1},\]
woraus $a=\frac{3x-3}{x},~ b=\frac{3}{x}$ folgt. Weiterhin nach der ersten Gleichung, $y=3-x$. Man wendet nun den Satz des Pythagoras im Dreieck $\Delta(AIB')$ an und erhält:

\begin{align*}
y^2+a^2 &= b^2\\
(3-x)^2 + \left( \frac{3x-3}{x}\right)^2 &= \left(\frac{3}{x}\right)^2\\
x^2-6x+9+\frac{9x^2-18x+9}{x^2} &= \frac{9}{x^2}\\
x^2-6x+9+9-\frac{18}{x} + \frac{9}{x^2} &= \frac{9}{x^2}\\
x^2-6x+18 &=\frac{18}{x}\\
x^3-6x^2+18x-18 &= 0\\
3x^3 -18x^2+54x-54 &= 0\\
\end{align*}

Dieses Resultat wird nun geschickt in das behauptete Ergebnis umgeformt:
\begin{align*}
0 &= x^3-6x^2+18x-18 \\
0 &= 3x^3 -18x^2+54x-54 \\
x^3 &= -2x^3+18x^2-54x+54\\
x^3 &= 2(27-27x+9x^2-x^3)\\
x^3 &= 2(3-x)^3\\
x^3 &= 2y^3\\
\end{align*}

Es verhält sich also, wie behauptet, $x:y$ wie $\sqrt[3]{2}$.

\end{proof}

\subsection{Winkeldrittelung}\footnote{\cite{henn}, Seite 64ff}\\

Wie bereits in der Einleitung festgestellt, ist die Winkeldrittelung mit Zirkel und Lineal nicht möglich. Es ist bekannt\footnote{zum Beispiel \cite{henn}, Seite 64f}, dass sich das Problem der Winkeldrittelung des Winkels $\alpha$ auf das Lösen eines Polynoms in $\cos(\alpha)$ vom Grad 3, genauer auf $4\cos^3(\alpha/3)-3\cos(\alpha/3)-\cos(\alpha)=0$ zurückführen lässt. Dies lässt sich durch Substitution auf das Polynom $4x^3-3x-\cos(\alpha)=0$ zurückführen. Die Winkeldrittelung ist dann genau dann, mit Zirkel und Lineal möglich, wenn dieses Polynom reduzibel über $\Q$ ist, was zwar für zum Beispiel $90^\circ$ der Fall ist, nicht aber für $60^\circ$, denn das Polynom $4x^3-3x-1/2$, welches sich durch Substitution von $y=2x$ im mit zwei Mal genommenen Ausdruck in $y^3-3y-1=0$ umformen lässt, ist irreduzibel über $\Q$. Wäre es nämlich reduzibel, so gäbe es einen vollständig gekürzten Bruch $\frac{a}{b}$, der das Polynom löst, und es wäre:
\[\frac{a^3}{b^3}-3\frac{a}{b}-1=0 \Leftrightarrow a^3=b^3+3ab^2=b^3(b+3a)\]
Dies würde bedeuten, dass jeder Primteiler von $a$ auch einer von $b$ ist, Widerspruch zu "`teilerfremd"'. Es könnte höchstens noch $b=1$ sein, dann müsste aber jeder Primteiler von $a$ auch $1$ teilen, was zu $a=\pm 1$ führt, und ebenfalls ein Widerspruch ist. \footnote{\cite{henn}, Seite 65}\\
Die Tatsache, dass das Winkeldritteln mit Origami möglich ist, wurde im letzten Kapitel nicht nur bewiesen, es wurde auch bereits eine mögliche Konstruktion angegeben. Nun soll eine weitere, einfacher zu faltende Konstruktion vorgestellt werden:

\begin{thm}{Winkeldrittelung, konkrete Konstruktion}\\
Gegeben sei ein rechteckiges Blatt mit Ecken $A,B,C,D$, zum Beispiel DIN A4, in das der zu drittelnde Winkel $90^\circ>\alpha=\angle(CBP)$  eingefaltet ist:

\begin{figure}[h]
\includegraphics[width=8cm]{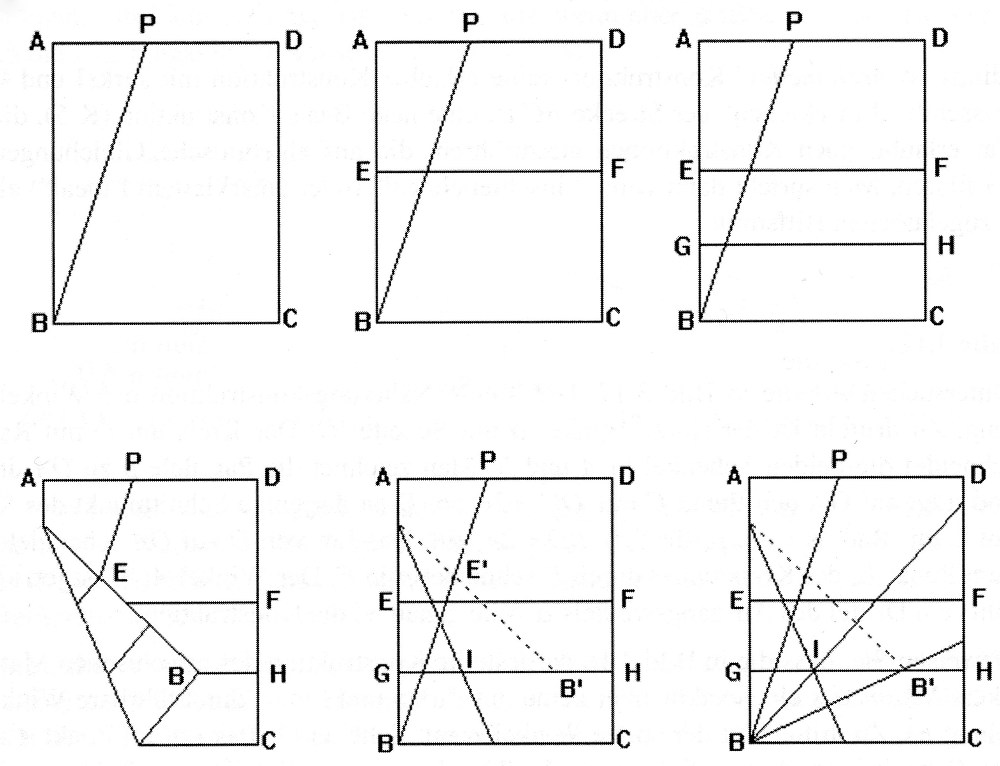}
\caption[\cite{henn}, Seite 68]{Winkeldritteln}
\label{zul2}
\end{figure}

Man faltet dann eine beliebige Parallele $\overleftrightarrow{EF}$ zu $\overleftrightarrow{BC}$ und die Mittelparallele $\overleftrightarrow{GH}$ dieser beiden Geraden. Dann faltet man so, dass $B\in\overline{GH}$ und $E\in\overline{BP}$ ist, die Bildpunkte seien $B'$ bzw. $E'$, die dabei entstehende Faltkante schneidet $\overleftrightarrow{GH}$ im Punkt $I$. Dann dritteln die Halbgeraden $\overrightarrow{BI};\overrightarrow{BB'}$ den Winkel $\alpha$.
\end{thm}

\begin{proof}\footnote{http://www.geometrie-und-algebra.de, Lösungshinweise}
Man betrachtet zunächst noch einmal die letzte durchgeführte Faltung nach Origami-Axiom 6.
\begin{figure}[h]
\includegraphics[width=8cm]{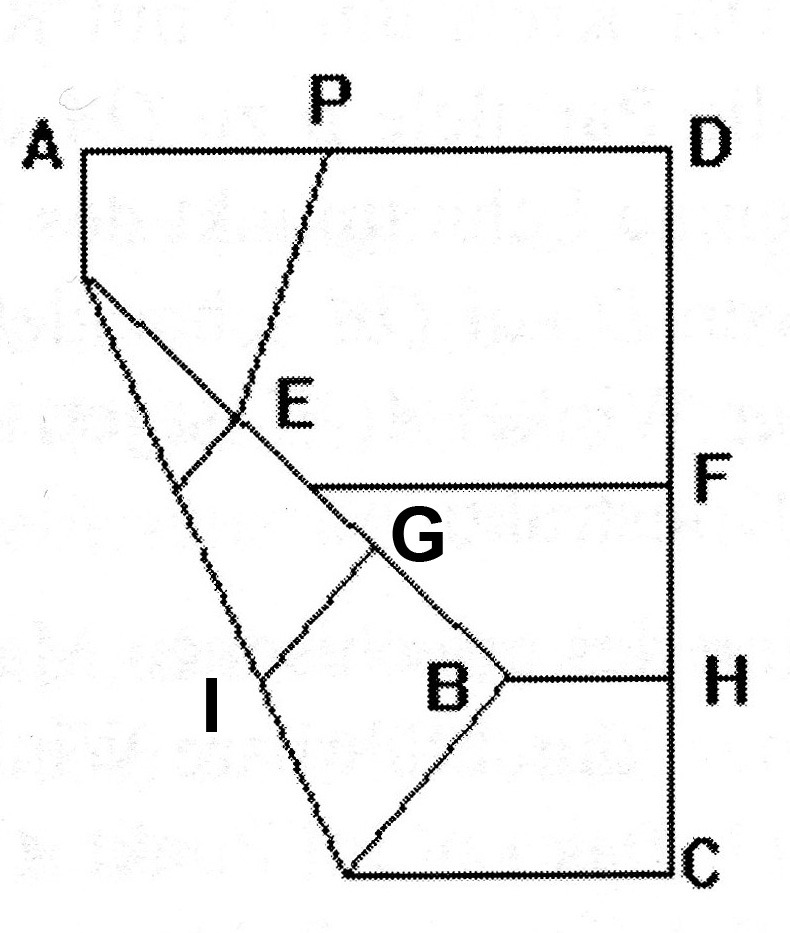}
\caption[eigene Bearbeitung von\cite{henn}, Seite 68]{Winkeldritteln 2}
\label{zul2}
\end{figure}

Nach Konstruktion ist nun der Fall, dass die Gerade $\overleftrightarrow{IG}$ senkrecht auf der Blattkante steht, die Blattkante legt hierbei nun die Verbindungsgerade $g$ von $E'$ und $B'$ fest. Da $\overleftrightarrow{GH}$ und $\overleftrightarrow{BC}$ parallel waren, muss das nun auch für $\overleftrightarrow{GI}$ und die Gerade $h$ durch $B$ und den Schnittpunkt der Faltlinie mit der Seite $\overline{BC}$ gelten. $h$ steht aber offenbar senkrecht auf $g$ damit muss aber nun auch die Gerade $\overleftrightarrow{BI}$ senkrecht auf $g$ stehen. Nach Konstruktion ist aber auch  der Schnittpunkt von $\overleftrightarrow{BI}$ mit $g$, bezeichnet $G'$, der Mittelpunkt der Strecke $\overline{E'B'}$, weil die Geraden $\overleftrightarrow{BC}, \overleftrightarrow{GH}, \overleftrightarrow{EF}$ parallel und waren und $\overleftrightarrow{GH}$ die Mittelparallele der beiden anderen Geraden ist. Damit sind die Dreiecke $\Delta{BB'G'}$ und $\Delta(BG'E')$ kongruent nach dem Satz SWS.\\
Man bezeichne den Fußpunkt des Lotes von $\overleftrightarrow{BC}$ durch $B'$ mit $R$. Das Dreieck $\Delta(BRB')$ ist dann, da nach Konstruktion $B'R=B'G'$ gilt, ebenfalls kongruent zu den beiden vorherigen Dreiecken, nun nach SsW.\\
 Daraus folgt insbesondere, dass die drei Winkel bei $B$ gleich sind, der zu drittelnde Winkel also tatsächlich in drei gleiche Teile zerlegt wurde.
\end{proof}

Da es nach Axiom 4 möglich ist, Winkel mittels Origami zu halbieren, lassen sich auch Winkel zwischen $90^\circ$ und $180^\circ$ dritteln. Diese Konstruktion ist in der Praxis wesentlich einfacher durchzuführen als die im vorhergehenden Kapitel genannte, wenn auch der Beweis etwas trickreicher argumentiert und, zumindest für den Autor dieser Arbeit, schwerer nachzuvollziehen war. Das aktive Nachfalten ist fürs Verständnis der Kongruenzen extrem hilfreich!\\

Die bisher angegebenen Konstruktionen lösten oft eine spezielle Gleichung vom Grad 3. Nun soll das Lösen von allgemeinen Gleichungen von Grad 3 durch konkrete Faltkonstruktion genauer untersucht werden.\\

\subsection{Lösung von Gleichungen vom Grad 3}\footnote{\cite{ger95}, Kapitel 6}\\

Bereits in Kapitel 3.3. wurde diskutiert, dass das Falten der gemeinsamen Tangente zweier Parabeln ein analytisches Problem vom Grad 3 ist. Das Origami-Axiom 6 erlaubt es, diese Tangente zu falten, daher ist damit zu rechnen, dass es so möglich sein könnte, Wurzeln von Gleichungen vom Grad 3 explizit falten zu können, wenn man nur die beiden Parabeln entsprechend geschickt wählt. Zunächst soll die Konstruktion der dritten Wurzel aus einem Bruch von Origami-Zahlen dargestellt werden, um diese Methode dann auf allgemeine Gleichungen vom Grad 3 anzuwenden. Die Lösungsmethode für die dritte Wurzel, die hierbei verwendet wird, erinnert naturgemäß stark an die bereits im Beweis von Lemma 3 verwendete Methode. Eine ausführliche Behandlung dieser Faltmethode wird dennoch für notwendig gehalten, da eine exakte Parametrisierung der beteiligten Parabeln bisher unterblieb, diese war auch zum Beweis des Lemmas nicht notwendig, wird aber wohl benötigt, um die Lösungen einer Gleichung $ax^3+bx^2+cx+d=$ aus den Koeffizienten dieser Gleichung konstruieren zu können. Dass das Problem des Findens der gemeinsamen Tangente zweier Parabeln zu mit dritten Wurzeln zu tun hat, war schon in der Antike bekannt.

\subsubsection{Das Falten dritter Wurzeln}\footnote{ebd}\\

Man betrachte die Parabeln

\[p_1:y^2=2ax ~ \text{und}~ p_2:x^2=2by.\]

\begin{figure}[h]
\includegraphics[width=8cm]{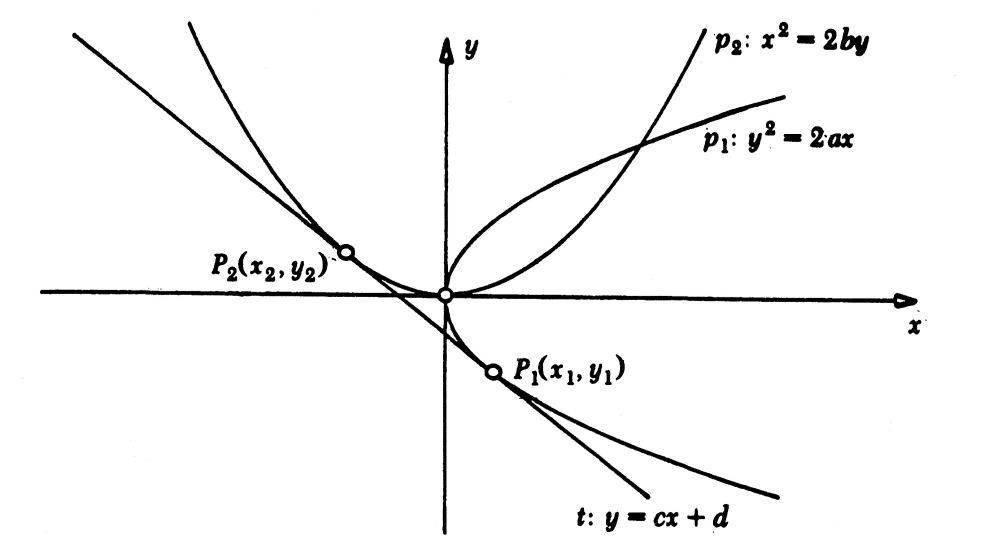}
\caption[\cite{ger95}, Kapitel 6]{Dritte Wurzeln 1}
\label{zul2}
\end{figure}

Diese Parabeln schneiden sich bereits in zwei Punkten, haben also nur noch eine gemeinsame Tangente, diese habe die Gleichung

\[t:y=cx+d.\]

Sei $P_1=(x_1,y_1)$ der Berührpunkt von $t$ mit $p_1$, dann hat $t$ aber auch die Gleichung:

\[yy_1=ax+ax_1 \Leftrightarrow y=\frac{a}{y_1}x+\frac{ax_1}{y_1}\]

Koeffizientenvergleich liefert dann $c=a/y_1$ und $d=ax_1/y_1$. Also ist 
\[y_1=\frac{a}{c} ~\text{und}~x_1=\frac{d}{x}\]
Wenn man das in die Gleichung von $p_1$ einsetzt, ergibt sich
\[\frac{a^2}{c^2}=2a\frac{d}{c} \Leftrightarrow a=2cd\]

Sei nun $P_2=(x_2, y_2)$ der Punkt, in dem sich $t$ und $p_2$ berühren. Dann hat $t$ auch die Gleichung
\[xx_2=by+by_2 \Leftrightarrow y=\frac{x_2}{b}x-y_2.\]
Diesmal liefert der Koeffizientenvergleich $c=x_2/b$ und $d=-y_2$. Dann ist:
\[x_2=bc \text{und} y_2=-d\]
Einsetzen in $p_2$ liefert:
\[b^2c^2=-2bd \Leftrightarrow d=-\frac{bc^2}{2}\]
Zusammenbasteln der Resultate für $a$ und $d$ liefert:\\

\begin{align*}
a &= -bc^3\\
a &= -\sqrt[3]{\frac{a}{b}}
\end{align*}

Also ist die Steigung der Tangente die negative dritte Wurzel aus dem Quotienten der beiden Koeffizienten der Parabeln. Daher kann eine dritte Wurzel auch folgendermaßen gefaltet werden:

\begin{figure}[h]
\includegraphics[width=8cm]{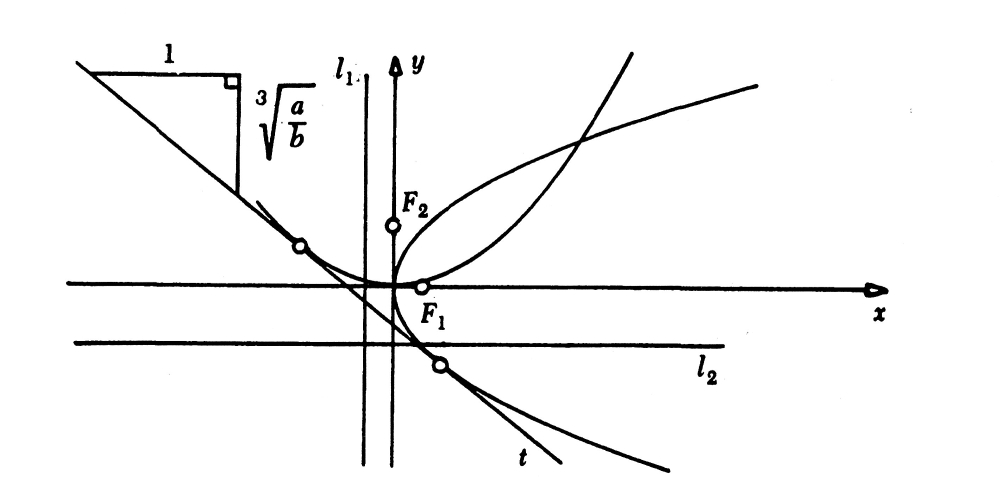}
\caption[\cite{ger95}, Kapitel 6]{Dritte Wurzeln 2}
\label{zul2}
\end{figure}

Es seien $a$ und $b$ gegeben. Man faltet nun zwei orthogonale Geraden, die die Hauptachsen der Parabeln repräsentieren sollen. Man faltet nun $a/2$ nach links und rechts des Ursprungs, um links die Leitgerade $l_1$ der Parabel $p_1$, die parallel zur y-Achse durch diesen Punkt verläuft, und rechts den Brennpunkt $F_1$ von $p_1$ zu erhalten. Genauso faltet man nach oben und unten $b/2$, um unten die Leitgerade $l_2$ der Parabel $p_2$, die  parallel zur x-Achse durch diesen Punkt verläuft, und oben den Brennpunkt $F_2$ von $p_2$ zu erhalten. (Es wäre genauso möglich, $a$ und $b$ nach rechts und links zu falten, oder jedes andere Vielfache, solange man nur das gleiche Vielfache von $a$ und $b$ wählt.) Axiom 6 erlaubt nun die Konstruktion der gemeinsamen Tangente durch das Falten von $F_1$ auf $l_1$ und $F_2$ auf $l_2$, deren Steigung, wie bereits gesagt, $-\sqrt[3]{\frac{a}{b}}$ ist. Faltet man nun die Einheitsstrecke parallel zur x-Achse an irgendeinen Punkt der Tangente an, und vervollständigt dann das Steigungsdreieck, so ergibt sich die Strecke der Länge $\sqrt[3]{\frac{a}{b}}$ wie eingezeichnet als die zur y-Achse parallele Kathete des Steigungsdreiecks.\\
Diese Konstruktionsmethode soll nun zu einer Konstruktionsmethode für die Lösungen einer allgemeinen Gleichung 3. Grades verallgemeinert werden.

\subsubsection{Lösen einer Gleichung vom Grad 3}\footnote{ebd}\\

Man betrachte nun die Parabeln

\[p_1:(y-n)^2=2a(x-m) ~ \text{und}~ p_2:x^2=2y.\]

\begin{figure}[h]
\includegraphics[width=8cm]{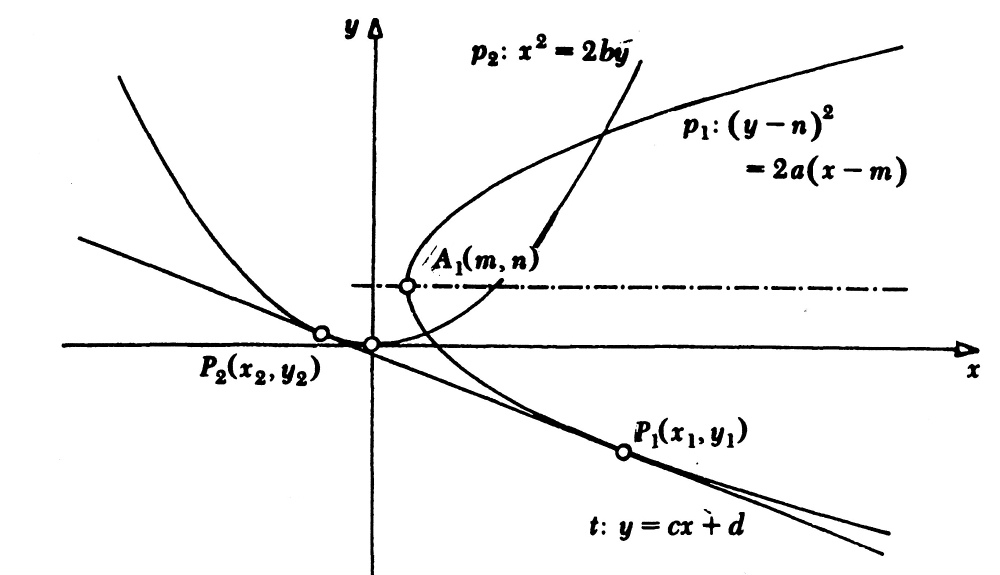}
\caption[\cite{ger95}, Kapitel 6]{Lösung einer Gleichung 3. Grades}
\label{zul2}
\end{figure}

In diesem Falle muss es nicht unbedingt nur eine Tangente geben, aber alle haben sicher wieder eine Gleichung der Form

\[t:y=cx+d\]

Sei nun wieder $P_1=(x_1,y_1)$ der Berührpunkt von $t$ und $p_1$, und daher hat $t$ auch die Gleichung:

\begin{align*}
(y-n)(y_1-n) &= a(x-m)+a(x_1-m)\\
y &= \frac{a}{y_1-n} x + \left( n+\frac{ax_1-2am}{y_1-n} \right)
\end{align*}

Daher, durch Koeffizientenvergleich, $c=\frac{a}{y_1-n}$ und $d=\left( n+\frac{ax_1-2am}{y_1-n} \right)$ und damit:

\begin{align*}
y_1 &= \frac{a+nc}{c}\\
x_1 &= \frac{d-n}{c}+2m
\end{align*}

Einsetzen in $p_1$ liefert nun:

\begin{align*}
(y_1-n)^2 &= 2a(x_1-m)\\
\frac{a^2}{c^2} &= 2a\left(\frac{d-n}{c} +m\right)\\
a &= 2c(d-n+cm)
\end{align*}

Wie im vorhergehenden Unterabschnitt führt die Betrachtung von $p_2$ und $t$ für den Parameter $b=1$ zu $d=-\frac{c^2}{2}$. Substituiert man nun $d$ in der vorigen Gleichung, so erhält man:
\begin{align}
2c \left( -\frac{bc^2}{2} - n +cm\right) &= a\notag\\
c^3 - 2m\cdot c^2+2n\cdot c +a &= 0
\end{align}

Eine Lösung der Gleichung (4.1) ist somit wieder die Steigung der gemeinsamen Tangente der beiden Parabeln. Die Gleichung (4.1) kann nun nur eine reelle und zwei komplexe Lösungen haben, oder aber drei reelle Lösungen, von denen evtl. zwei oder alle drei zusammenfallen. Natürlich gehören hierzu auch jeweils spezielle Konstellationen der beiden Parabeln. Zwei Lösungen fallen nämlich genau dann zusammen, wenn zwei Tangenten zusammenfallen, und das passiert genau dann, wenn sich die Parabeln berühren, und alle drei Lösungen fallen genau dann zusammen, wenn drei gemeinsame Punkte der Parabeln zusammenfallen (sie oskulieren).\\

Wie findet man nun eine Lösung der konkreten Gleichung dritten Grades, die oBdA durch $x^3+px^2+qx+r=0$ gegeben ist? Dann ergeben sich die Parameter in Gleichung (4.1) zu $m=-p/2, n=q/2$ und $a=r$.\\

Man muss nun also, was möglich ist, den Brennpunkt $F_1=\left(\frac{-p+r}{2},\frac{q}{2}\right)$ und die Leitgerade $l_1:x=\frac{-p-r}{2}$ der Parabel $p_1$ konstruieren. Der Brennpunkt $F_2=(0,1/2)$ und die Leitgerade $l_2:y=-1/2$ von $p_2$ sind ebenfalls leicht zu konstruieren. Wieder faltet man $F_1$ auf $l_1$ und $F_2$ auf $l_2$ und erhält somit eine Lösung der Gleichung (4.1) durch die Konstruktion des Steigungsdreiecks der hierbei entstehenden Tangente wie im vorhergehenden Unterabschnitt.\\

Natürlich folgen hieraus nun auch weitere mögliche Konstruktionen zur Lösung des Delischen Problems und des Problems der Winkeldrittelung als Spezialfälle, allerdings sind die bereits angegebenen expliziten Konstruktionen hierbei wesentlich eleganter.

\subsection{Die Konstruktion regelmäßiger n-Ecke}
Wenn in diesem Kapitel von einem n-Eck die Rede ist, so soll immer ein regelmäßiges n-Eck gemeint sein, also eines mit n gleich langen Seiten.
\subsubsection{Elementare Überlegungen und die Forschungen von Gauß}\footnote{\cite{henn}, Seiten 72ff}\\

Einfache Konstruktionen mit Zirkel und Lineal für das regelmäßige (gleichseitige) Dreieck und das regelmäßige Viereck, das Quadrat, kennt man aus der Schule und müssen hier nicht mehr erklärt werden. Weiterhin sind folgende Punkte leicht einsehbar:
\bi
\item Ist das n-Eck konstruierbar, so auch das 2n, 4n, 8n, ...-Eck, denn die Eckpunkte des 2n-Eckes sind die Schnittpunkte der Mittelsenkrechten der Seiten des n-Ecks mit seinem Umkreis.
\item Ist das regelmäßige n-Eck konstruierbar, und ist $n=ab, a\ge 3$, so ist auch das a-Eck konstruierbar, man nimmt dann einfach nur jede b-te Ecke.
\item Sind $a,b$ teilerfremd, und sind das a-Eck und das b-Eck konstruierbar, so ist auch das $ab-$Eck konstruierbar. Denn wenn $360^\circ /a$ und $360^\circ /b$ konstruierbar sind, dann auch $360^\circ /(ab)$.
\ei

Eine erste vollständige Theorie der n-Ecke und ihrer Konstruierbarkeit hat Gauß geschaffen, ihr zentraler Satz lautet:

\begin{thm}{Charakterisierung der mit ZuL konstruierbaren regelmäßigen n-Ecke}\\
Das regelmäßige n-Eck ist genau dann mit Zirkel und Lineal konstruierbar, wenn $n$ eine Zweierpotenz ist, oder aber wenn gilt:
\[n=2^s p_1\dots p_n\]
mit $s\in \N_0$ und paarweise verschiedenen fermatschen Primzahlen $p_1,\dots,p_n$.
\end{thm}
Da der Satz kein zentraler Bestandteil der Arbeit ist, soll auf einen Beweis verzichtet werden. Ein solcher findet sich zum Beispiel bei \cite{henn}. Mit Origami lassen sich deutlich mehr n-Ecke konstruieren, hier gilt:

\begin{thm}{Charakterisierung der mit Origami konstruierbaren n-Ecke}\\
Das regelmäßige n-Eck ist genau dann mit Origami-Methoden konstruierbar, wenn gilt:
\[n=2^r 3^s p_1 p_2\dots p_n\]
mit natürlichen Zahlen $r,s$ und Primzahlen $p_i$ vom Typ $p_i=2^u 3^v + 1;~ u,v\in \N$, sogenannte Pierpontsche Primzahlen.
\end{thm}
Das kleinste n-Eck, das nicht mit Zirkel und Lineal, wohl aber mit Origami konstruiert werden kann, ist das 7-Eck, und diese Konstruktion soll nun als nächstes besprochen werden.

\subsubsection{Konstruktion des 7-Ecks}\footnote{\cite{henn}, Seite 88ff}
\begin{figure}[h]
\includegraphics[width=8cm]{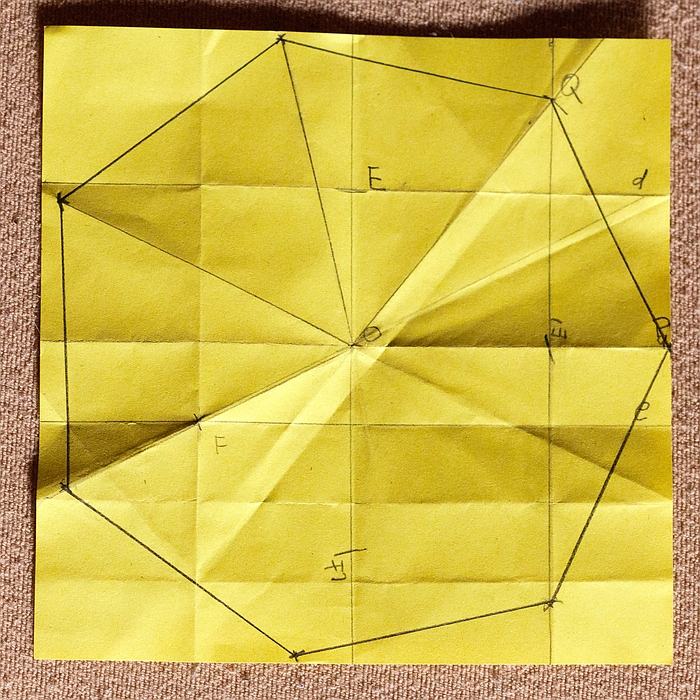}
\caption[Eigenes Foto]{Regelmäßiges 7-Eck}
\label{zul2}
\end{figure}

\begin{thm}{Konstruktion des 7-Ecks nach Alperin(2002)}\\
Man beginnt mit dem Quadrat mit Ecken $ABCD$. Zuerst werden die beiden Mittellinen $a$ und $b$, die "`Viertellinien"' $c$ und $d$ und die "`Achtellinie"' $e$ gemäß der Skizze in Abbildung 4.10 eingefaltet:
\begin{figure}[h]
\includegraphics[width=6cm]{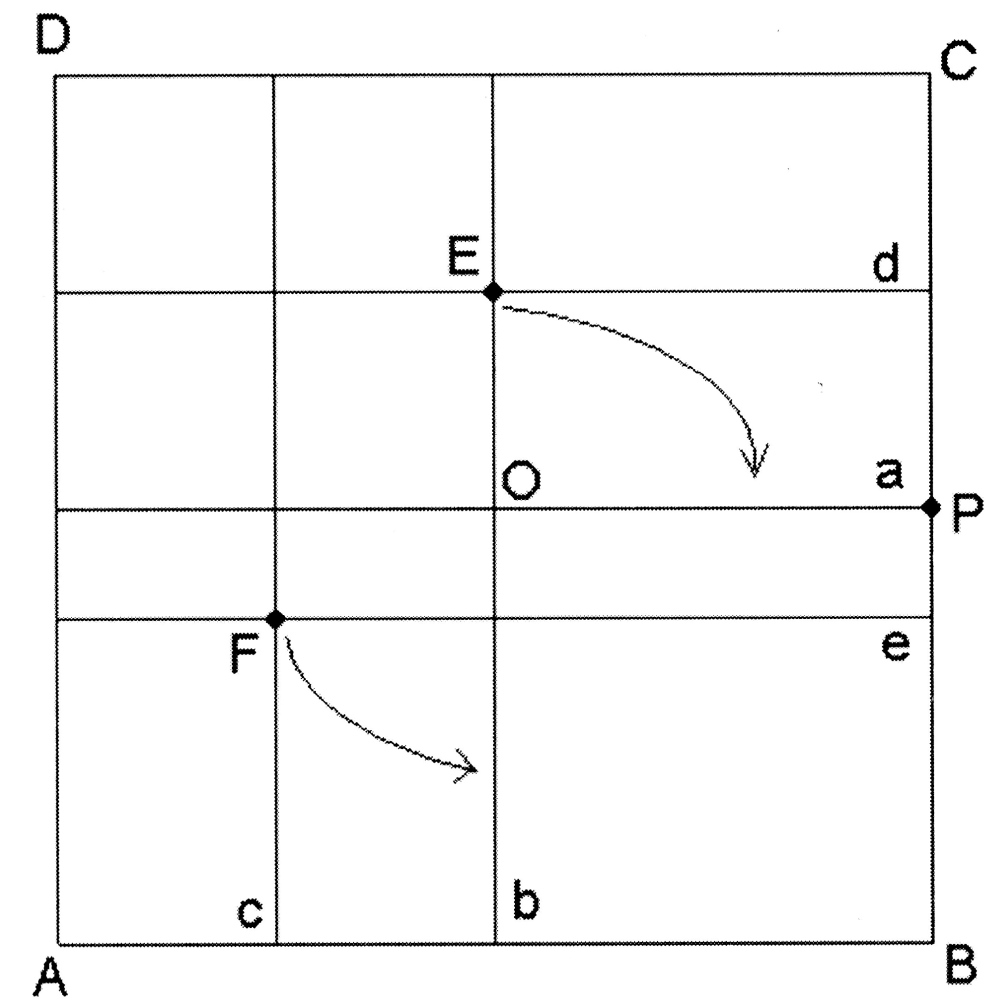}
\caption[\cite{henn}, Seite 88]{Regelmäßiges 7-Eck 2}
\label{zul2}
\end{figure}

Damit sind dann auch die Punkte $O,E,F$ gegeben. Ziel ist es, dass 7-Eck mit Mittelpunkt $O$ und Umkreisradius $OP$ zu falten, $P$ ist dann eine Ecke. Nun wird gemäß Axiom 6 so gefaltet, dass der Punkt $E\in a$ und $F\in b$ liegt, die Bildpunkte werden mit $E'$ und $F'$ bezeichnet. (Abbildung 4.11)
\begin{figure}[h]
\includegraphics[width=6cm]{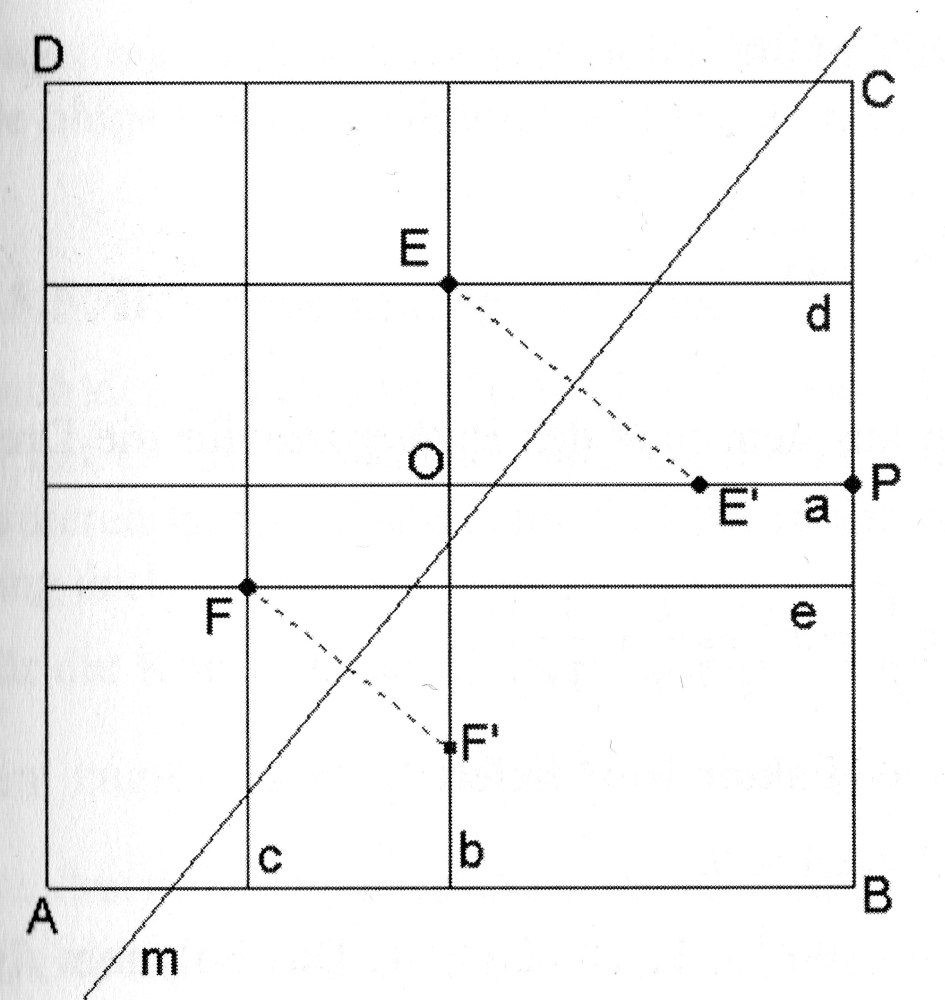}
\caption[\cite{henn}, Seite 89]{Regelmäßiges 7-Eck 3}
\label{zul2}
\end{figure}

Dann faltet man die Parallele $f$ zur Kante $\overline{BC}$ durch $E'$ und anschließend $P$ so auf $f$, dass die Faltlinie durch $O$ geht, der Bildpunkt von $P$ sei $Q$.(Abbildung 4. 12)
\begin{figure}[h]
\includegraphics[width=6cm]{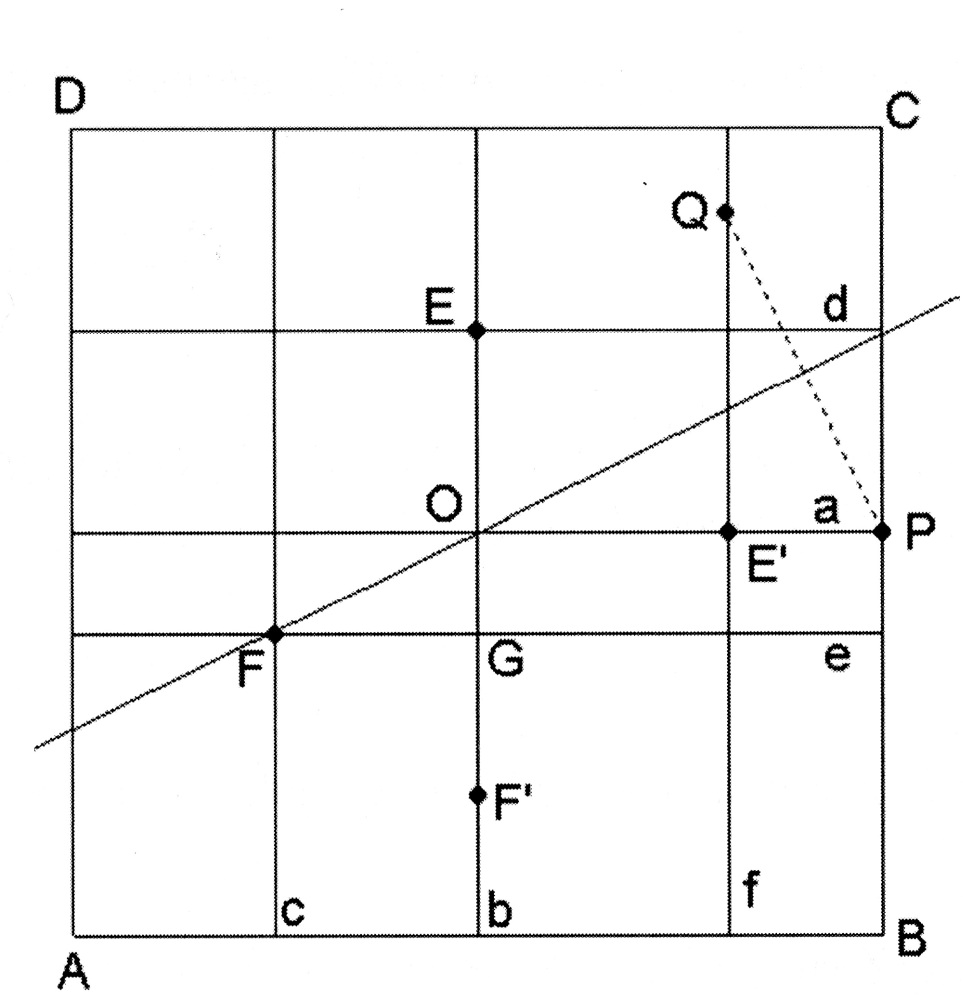}
\caption[\cite{henn}, Seite 89]{Regelmäßiges 7-Eck 4}
\label{zul2}
\end{figure}
Dann ist $Q$ ein weiterer Eckpunkt des 7-Ecks. Durch Spiegelung des Punktes $P$ an $\overrightarrow{OQ}$ erhält man den nächsten Punkt des Siebenecks, und so weiter.
\end{thm}

\begin{proof}
OBdA sei die Seitenlänge des Blattes $2$, sodass das entstehende Siebeneck den Einheitskreis als Umkreis hätte. Man verwendet nun das Koordinatensystem mit Ursprung $O$ und Achsen $a, b$. Zu zeigen ist nun, dass der Punkt $Q$ als komplexe Zahl die siebte Einheitswurzel $\xi=\cos(\alpha)+i\sin(alpha)$ mit $\alpha=\frac{2\pi}{7}$ ist, dafür wiederum genügt es zu zeigen, dass die x-Koordinate von $Q$, welche gleich der von $E'$ ist, $x=\cos(\alpha)$ ist.
\begin{figure}[h]
\includegraphics[width=6cm]{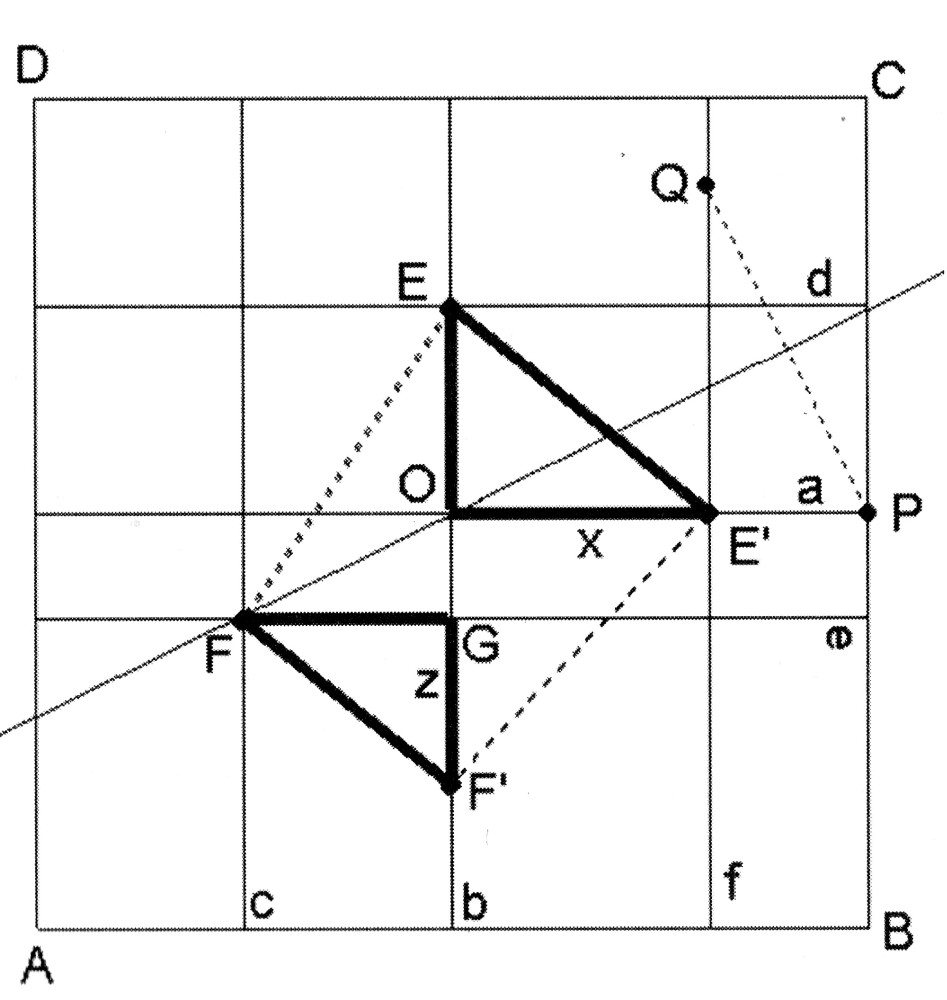}
\caption[\cite{henn}, Seite 89]{Regelmäßiges 7-Eck 5}
\label{zul2}
\end{figure}

Sei weiter $z=GF'$, dann gilt: $E=(0|0,5), E'=(x|0), F=(-0,5|-0,25), F'=(0|-0.25-z), G=(0|-0,25)$.\\
Die beiden fett markierten Dreiecke $\Delta(GFF')$ und $\Delta(OEE')$ sind ähnlich, daher gilt:
\[\frac{z}{0,5})=\frac{0,5}{x}~\Rightarrow z=\frac{1}{4x}\]
Nach Konstruktion (vgl. Abbildung 4.11) ist weiter $EF=E'F'$. Der Satz des Pythagoras im Dreieck $\Delta(FGE)$ liefert daher: $FG^2+GE^2=EF^2$, und der Pythagoras im Dreieck $\Delta(OE'F')$ liefert: $OE'^2+OF'^2=E'F'^2$. Man kann nun beide Gleichungen zusammensetzen und erhält:
\[(0.5)^2+(0,75)^2=x^2+(0,25+z)^2=x^2+\left(\frac{1}{4}+\frac{1}{4x}\right)^2\]
Zusammenfassen führt zum über $\Q$ reduziblen Polynom $16x^4-12x^2+2x+1=0$. $x$ ist also Nullstelle des Polynoms $f(y)=16y^4-12y^2+2y+1$, dieses ist aber nicht das Minimapolynom von $x$, da es, wie schon gesagt, reduzibel ist und den Faktor $(y-0,5)$ abspaltet. Polynomdivision führt zum Polynom $g(y)=8y^3+4y^2-4y-1$, und das ist irreduzibel über $\Q$. Dieses Polynom hat drei reelle Nullstellen, von denen eine echt positiv ist, daher kommt nur diese für $x$ in Frage. Zu zeigen bleibt, dass diese Nullstelle auch gerade $\cos(\alpha)$ ist.\\
Dazu wird gezeigt, dass $w=\cos(\alpha)$ ebenfalls eine positive Nullstelle dieses Polynoms ist. Da es nur eine solche gibt, muss dann $w=x$ gelten.\\
Die bereits ausgezeichnete 7-te Einheitswurzel $\xi$ hat Minimalpolynom $y^6+y^5+\dots+y+1$ und die Potenzen $\xi^2, \dots, \xi^6$ sind seine anderen Nullstellen. Es folgt nun:
\[w=cos(\alpha)=\frac{\xi+\xi^6}{2}=\frac{\xi+\xi^{-1}}{2},\]
da $\xi^i=\xi^{7-1}$ gilt. Damit kann man die nun in unserem Polynom benötigten Vielfachen ausrechnen:
\begin{align*}
4w &= 2(\xi+\xi^{-1})\\
4w^2 &= (\xi+\xi^{-1})^2=\xi^2+2+\xi^{-2}\\
8w^3 &= (\xi+\xi^{-1})^3=\xi^3+3\xi+3\xi^{-1}+\xi^{-3}
\end{align*}
Einsetzen in $g(y)$ schließlich liefert:

\begin{align*}
g(w) &= 8w^3+4w^2-4w-1\\
&= (\xi^3+3\xi+3\xi^{-1}+\xi^{-3})+(\xi^2+2+\xi^{-2})-2(\xi+\xi^{-1})-1\\
&= \xi^3+\xi^2+\xi+\xi^{3}+\xi{-2}+\xi^{-1}+1\\
&=\xi^6+\xi^5+\xi^4+\xi^3+\xi^2+\xi+1
\end{align*}
$\xi$ ist aber nach Definition Nullstelle dieses Polynoms. Damit hat die Faltkonstruktion tatsächlich ein regelmäßiges 7-Eck erzeugt.
\end{proof}
Ein weiteres Resultat, das untrennbar mit den Arbeiten von Gauß verbunden ist, und auch das nächste n-Eck, dessen Konstruktion mit Zirkel und Lineal möglich, und dabei nicht völlig trivial, ist, ist das 17-Eck. 

\subsubsection{Das regelmäßige 17-Eck}
Es ist bereits bekannt, dass zur prinzipiellen Möglichkeit der Konstruktion zu zeigen ist, dass der Cosinus des Zentralwinkels mit geschachtelten Quadrat - und für Origami eventuell dritte - Wurzeln darstellbar sein muss. Gauß gab hierzu folgende Darstellung an:

\begin{align*}
\cos\left(\frac{360^\circ}{17}\right) &= \frac{1}{16} \left( -1 + \sqrt{17}+\sqrt{2(17-\sqrt{17})}\right.\\
&+\left. 2\sqrt{17+3\sqrt{17}-\sqrt{2(17-\sqrt{17})}-2\sqrt{2(17+\sqrt{17})}}\right)
\end{align*}
Damit ist das 17-Eck mit Zirkel und Lineal und somit insbesondere auch mit Origami konstruierbar.\footnote{\cite{henn}, Seite 91}\\
Wie bereits nach der Darstellung dieses Wurzelausdruckes zu erwarten, ist die Konstruktion allerdings nicht ganz einfach. Es handelt sich um die "`Übersetzung"' einer ZuL-Konstruktion von H.W. Richmond in "`Origami-Sprache"'.

\newpage
\begin{thm}{Konstruktion des 17-Ecks}\footnote{\cite{ger02}}\\

\begin{figure}[h]
\includegraphics[width=10cm]{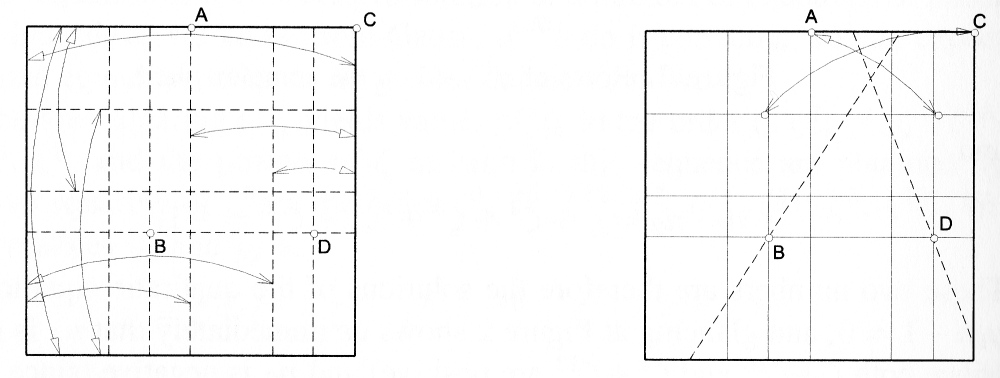}
\caption[\cite{ger02}]{Regelmäßiges 17-Eck 1}
\label{zul2}
\end{figure}
Man faltet zunächst die Mittellinien und dann in der rechten Hälfte des quadratischen Blattes die Viertel- bzw. Achtellinien. Die Linie, auf der $B$ liegt, ergibt sich durch Falten der linken Quadratseite auf die eben gefaltete Mittellinie, die zweite vertikale Linie in der linken Hälfte durch Faltung der Kante auf die rechte Viertellinie. Die Mittelline in der oberen Hälfte wird gefaltet, die horizontale Linie in der unteren Hälfte ergibt sich durch Faltung der unteren Blattkante auf diese Mittellinie.\\
Danach wird $A$ so auf die obere Mittellinie gefaltet, dass die Faltung durch $B$ geht. Nun faltet man $C$ so auf diese Mitellinie, dass die Faltung durch $D$ geht.

\begin{figure}[h]
\includegraphics[width=10cm]{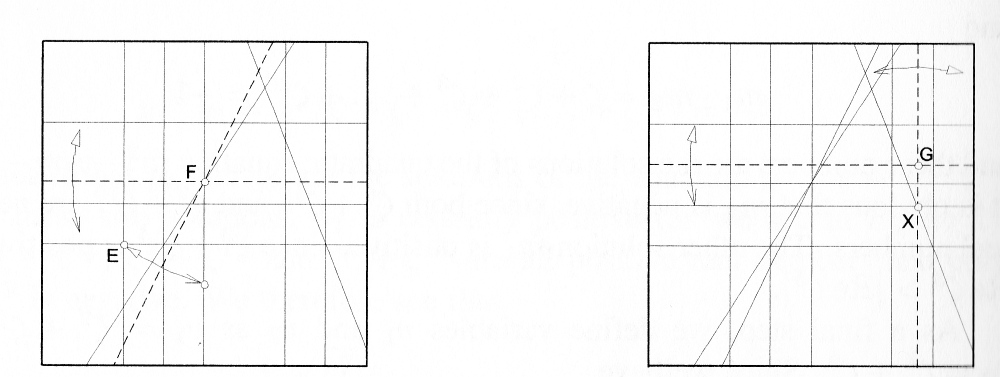}
\caption[\cite{ger02}]{Regelmäßiges 17-Eck 2}
\label{zul2}
\end{figure}
Man faltet nun die Linie, die den Punkt $B$ enthält, auf die obere Viertellinie. Der Schnittpunkt dieser Faltgeraden mit der vertikalen Mittellinie ist der Punkt $F$. Dann faltet man den Punkt $E$ so auf diese Mitellinie, dass die Faltung durch $F$ geht.
Der Schnittpunkt dieser Faltlinie mit der Faltlinie durch $D$, mit der man $C$ auf die Mittellinie faltete, wird nun so auf die rechte Quadratkante gefaltet, dass die Faltung parallel zu dieser Kante geht. Der Schnittpunkt dieser Faltung mit der horizontalen Mittellinie ist der Punkt $X$. Nun faltet $X$ auf die obere Viertellinie so, dass die Faltline parallel zur oberen Quadratkante ist, deren Schnittpunkt mit der gerade durchgeführten vertikalen Faltung ist der Punkt $G$.

\begin{figure}[h]
\includegraphics[width=10cm]{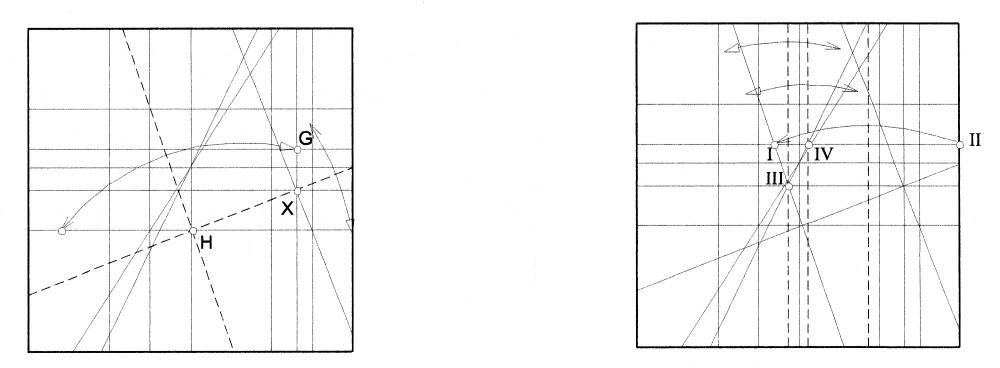}
\caption[\cite{ger02}]{Regelmäßiges 17-Eck 3}
\label{zul2}
\end{figure}
Man faltet nun die Faltung aus Schritt 2, die den Punkt $D$ enthält, so auf sich selbst, dass die Faltlinie durch $X$ geht. Diese Faltlinie schneidet dann die horizontale Faltlinie durch $E$ im Punkt $H$. Nun faltet man $G$ so auf $\overleftrightarrow{EH}$, dass die Faltlinie durch $H$ geht.\\
Anschließend faltet man vertikale Faltungen durch die in der Skizze markierten Punkte $III$ und $IV$ und anschließend $II$ auf $I$, ohne aufzufalten.

\begin{figure}[h]
\includegraphics[width=10cm]{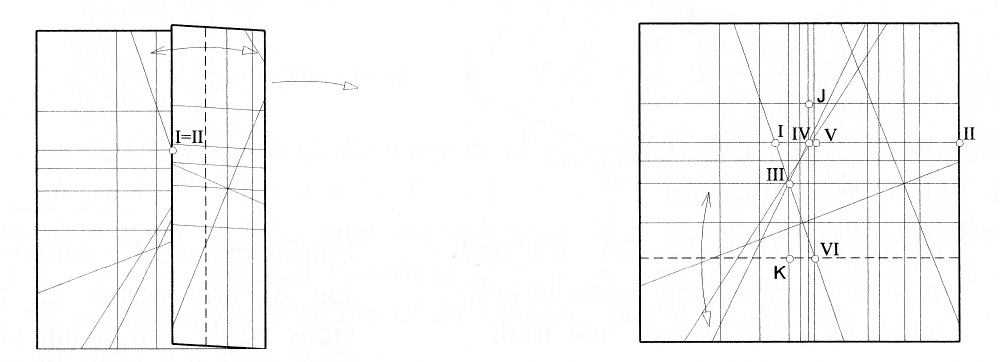}
\caption[\cite{ger02}]{Regelmäßiges 17-Eck 4}
\label{zul2}
\end{figure}
Nun faltet man die beiden Schichten entlang der vertikalen Faltung durch Punkt $D$ aus Schritt 1 und faltet alles auf.\\
Dann faltet man eine horizontale Linie durch $K$ (Lage entsprechend Skizze) ein.

\newpage
\begin{figure}[h]
\includegraphics[width=10cm]{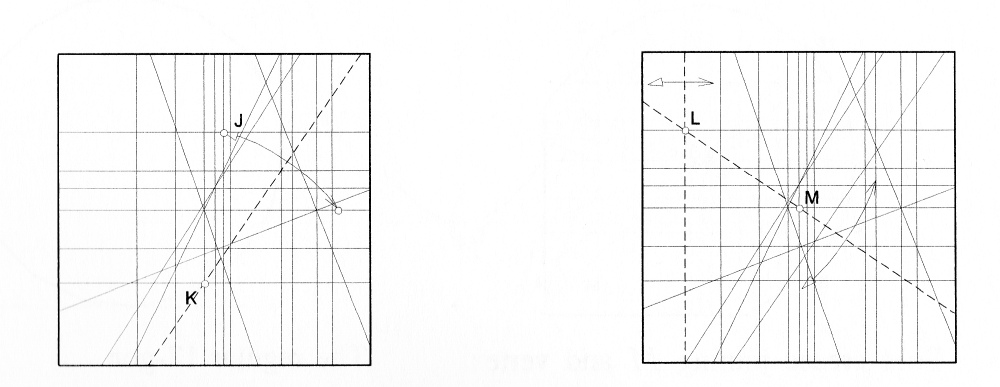}
\caption[\cite{ger02}]{Regelmäßiges 17-Eck 5}
\label{zul2}
\end{figure}
Nun faltet man $J$ (gelegen wie in der Skizze) so auf die horizontale Mittellinie, dass die Faltlinie durch $K$ verläuft.\\
Diese Faltlinie wird nun so auf sich selbst gefaltet, dass die neue Faltlinie durch $M$ verläuft. Diese Faltung schneidet die obere Viertellinie im Punkt $L$, und man faltet nun noch die Parallel zur linken Quadratkante durch $L$ ein.

\begin{figure}[h]
\includegraphics[width=10cm]{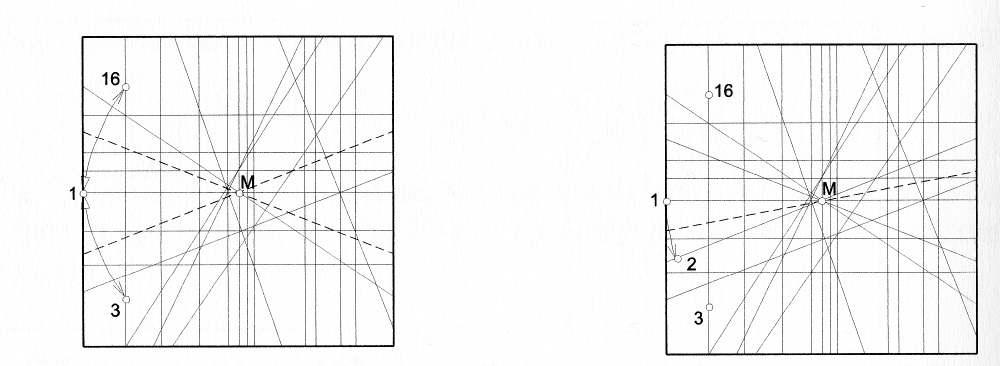}
\caption[\cite{ger02}]{Regelmäßiges 17-Eck 6}
\label{zul2}
\end{figure}
Der Punkt $1$ wird nun auf die gerade eingefaltete vertikale Falte gefaltet, dies geht nach unten oder nach oben und ergibt die Punkte $16$ und $3$.\\
Nun faltet man $1$ auf die gerade entstandene Falte, die $1$ auf $3$ abgebildet hatte. Dann ist der Bildpunkt von $1$ der Punkt $2$. $\Delta(M12)$ ist dann ein Bestimmungsdreieck des 17-Ecks. Ähnliche Symmetrieüberlegungen wie bei der vorhergehenden Konstruktion des 7-Ecks vervollständigen nun das 17-Eck.
\end{thm}

Bevor mit dem konkreten Beweis begonnen wird, sind noch einige theoretische Vorüberlegungen durchzuführen. Man nehme an, dass die Ecken des 17-Ecks als Lösung eines Polynoms in der komplexen Ebene gegeben sind. Dieses Polynom ist die Gleichung $z^17 -1=0$. Daher ist der Einheitskreis der Umkreis des 17-Ecks, und der Punkt $z_1=1$ auf der reellen Achse ist eine Ecke. Die anderen 16 Ecken sind daher die Lösung des Kreisteilungspolynoms

\[\frac{z^{17} -1}{z-1}=z^{16}+z^{15}+\dots+z+1=0\] 

und sie lassen sich schreiben als:
\[z_k=\cos\left(\frac{2\pi (k-1)}{17}\right)+i \sin\left(\frac{2\pi (k-1)}{17}\right), ~ k=1,2,\dots, 17\]

$z_2:=\xi$ ist dann eine primitive Einheitswurzel, und es gilt:
\[z_k=\xi^{k-1}, ~k=2,3,\dots,17\]
Geschicktes Umsortieren der Potenzen von $\xi$ erlaubt nun, diese Gleichung vom Grad 16 auf das Lösen von Gleichungen mit kleinerem Grad zurückzuführen. Gauß' Resultate beweisen sogar, dass all diese Gleichungen vom Grad höchstens 2 sein werden. Wiederholtes Quadrieren von $\xi$ liefert $\xi, \xi^2, \xi^4, \xi^8, \xi^{16}, \xi^{15}, \xi^{13},\xi^9$, und die fehlenden Potenzen ergeben sich durch wiederholtes Quadrieren von $\xi^3$ in der Reihenfolge $xi^3, \xi^6, \xi^{12}, \xi^7, \xi^{14}, \xi^{11}, \xi^5, \xi^{10}$. Seien nun 
\[y_1=\xi+ \xi^2+ \xi^4+ \xi^8+ \xi^{16}+ \xi^{15}+ \xi^{13}+\xi^9\]
und
\[y_2=xi^3+ \xi^6+ \xi^{12}+ \xi^7+ \xi^{14}+ \xi^{11}+ \xi^5+ \xi^{10}\]
Dann ergibt sich:
\begin{align*}
y_1+y_2 &= \sum\limits_{k=1}^{16}{\xi^k}=-1\\
y_1\cdot y_2 &= 4\cdot (y_1+y_2)=-4
\end{align*}
Das heißt, $y_1$ und $y_2$ sind die Nullstellen der Gleichung $y^2+y-4=0$ nach dem Satz von Vieta. Die Lösungen dieser Gleichung sind 
\[\frac{-1\pm\sqrt{17}}{2},\] 
und man kann leicht nachrechnen, dass $y_1$ diejenige mit Plus sein muss, daher ist $y_2$ die mit Minus.\\
Nun definiert man die Variablen $n_1=\xi+\xi^4+\xi^{16}+\xi^{13}$ und $n_2=\xi^2+\xi^8+\xi^{15}+\xi^9$, und offenbar ist $n_1+n_2=y_1$ und $n_1 n_2=-1$. Diese zwei Zahlen sind somit die Lösungen der quadratischen Gleichung$n^2-y_1n-1=0$, und $n_1$ ist die positive, $n_2$ die negative Lösung.\\

Völlig analog definiert man $m_1=\xi^3+\xi^{12}+\xi^{14}+\xi^5$ und $m_2=\xi^6+\xi^7+\xi^{11}+\xi^{10}$. Dann ist $m_1+m_2=y_2$ und $m_1 m_2=-1$, somit sind diese Zahlen die Lösungen der quadratischen Gleichung $m^2-y_2 m-1=0$. $m_2$ ist negativ, $m_1$ positiv.\\
In einem letzten Schritt definiert man die Variablen $v_1:=\xi^2+\xi^{15},~ v_2=\xi^8+\xi^9$. Nachdem $v_1+v_2=n_2$ und $v_1 v_2=m_2$ ist, sind diese Zahlen Lösungen der quadratischen Gleichung $v^2-n_2v+m_2=0$, wobei 
\[v_1=2\cos\frac{4\pi}{17}>0, ~~ v_2=2\cos\frac{16\pi}{17}<0\]
Ziel der Konstruktion muss also gewesen sein, die Zahl $v_1$ zu konstruieren und dann, durch diverse Symmetrieoperationen, das 17-Eck zu vervollständigen.

\begin{proof}
Warum aber funktioniert diese Faltung? Man erinnert sich zunächst daran, dass die Nullstellen einer quadratischen Gleichung $x^2+px+q=0$ die Steigung der Tangenten der Parabel mit Brennpunkt $(0,1)$ und Leitgeraden $l:y=-1$, die durch den Punkt $(-p,q)$ gehen. Diese lassen sich falten, indem man $F$ so auf $l$ faltet, dass die Faltlinien durch $P$ gehen, dafür gibt es im Allgemeinen zwei Möglichkeiten.\footnote{Das folgt aus den Konstruktionen für allgemeine dritte Wurzeln als Spezialfall}. Von dieser Idee wird in der gegebenen Konstruktion nun ausgiebig Gebrauch gemacht. Um $v_1$ als positive Lösung einer Gleichung $v^2-n_2v+m_2=0$ zu bestimmen, muss man also den Punkt mit Koordinaten $(n_2|m_2)$ konstruieren.$n_2$ ist aber die negative Lösung von $n^2-y_1n-1=0$, und $m_2$ ist die negative Lösung von $m^2-y_2m-1=0$, und daher ist es nun zunächst nötig, die Punkte $y_1$ und $y_2$ als die Lösungen von $y^2+y-4=0$ zu bestimmen. Das passiert in den ersten beiden Schritten der Konstruktion. Sei die Seitenlänge des ursprünglichen Quadrates 8, und auf Grund der Halbierungen in Schritt 1 kann man schließen, dass $A=(0|1),B=(-1|-4)$, wobei die x-Achse als nach rechts positiv, die y-Achse als nach oben positiv orientiert angenommen wird.Die Gerade, auf der $B$ liegt, hat somit die Gleichung $y=-1$, und die Faltung von $A$ auf diese Gerade sodass die Faltung durch $B$ geht konstruiert eine Gerade mit Steigung $y_1$.\\
Auf Grund der endlichen Größe des Blattes kann man so keine Gerade mit Steigung $y_2$ falten, aber da die Punkte $C$ und $D$ relativ zueinander genauso liegen wie $A$ und $B$, kann man genauso annehmen, dass $C=(0|1)$ und $D=(-1|-4)$, und $C$ so auf die Gerade $y=-1$ falten, dass die Faltung durch $D$ geht, und daher entsteht so eine Gerade mit Steigung $y_2$.\\
Im dritten Schritt faltet man eine Gerade mit Steigung $n_2$. Da $n_2$ die Gleichung $n^2-y_1n-1=0$ löst, muss man Punkte $F=(0|1),P=(y_1|-1)$ in einem noch festzulegendem Koordinatensystem wählen. Man wählt die x-Achse nach oben orientiert, y-Achse nach links und $B$ als Ursprung. Die Wahl von $E$ wie skizziert liefert einen Punkt $F$ mit x-Koordinate um 2 größer als die von $E$. Man kann daher annehmen, dass $E=(0|1),F=(y_1|-1)$ die Koordinaten der Punkte $E,F$ in diesem System sind. Das Falten von $E$ auf die vertikale Faltlinie durch $F$ so, dass die neue Faltlinie durch $F$ geht, liefert eine Faltung mit Steigung $n_2$ in diesem System.\\
In den Schritten 4 und 5 faltet man eine Gerade mit Steigung $m_2$ in dem Koordinatensystem der Schritte 1 und 2.Durch Falten der Gerade mit Steigung $y_2$ auf sich selbst erhält man eine Gerade mit Steigung $-1/y_2$, die x-Koordinate von $H$ ist daher um $|y_2|$ kleiner als die von $G$, und da $G$ auf der waagerechten Faltung 2 Einheiten höher als die durch $H$ liegt, kann man schließen, dass $G=(0|1),H=(y_2,1)$. $m_2$ ist die negative Lösung der Gleichung $m^2-y_2m-1=0$, und durch Falten von $G$ auf die waagerechte Gerade durch $H$ so, dass die Faltung durch $H$ geht, liefert daher eine Gerade mit Steigung $m_2$.\\
Nun kann man daran gehen, $v_1$ zu bestimmen, diese Zahl ist eine Lösung der quadratischen Gleichung $v^2-n_2v+m_2=0$, und ist daher Steigung einer Geraden, die durch Faltung von $(0|1)$ auf die Gerade $y=-1$ so, dass die Faltlinie durch den Punkt $(n_2|m_2)$ in einem geeigneten Koordinatensystem geht. Dies geschieht in den Schritten 6 bis 9.\\
Die waagerechte Gerade drei Einheiten von der oberen Kante entfernt schneidet die Gerade mit Steigung $m_2$ in $I$ und die rechte Kante des Quadrats in $II$. Die Faltung von $II$ auf $I$ und das Weiterfalten dieser Faltung in Schritt 7 führt zu einer vertikalen Geraden, die $\overleftrightarrow{I~II}$ in $V$ schneidet, und die Gerade mit Steigung $m_2$ in $VI$. Diese Gerade wäre auch ohne das doppelte Falten konstruierbar, nur etwas umständlicher. Da $I$ und $V$ eine Einheit voneinander entfernt sind, müssen $V$ und $VI$ genau $|m_2|$ Einheiten voneinander entfernt sein, und die waagerechte Faltung in Schritt 8 durch $VI$ liegt $m_2$ Einheiten unter $\overleftrightarrow{I~II}$.\\
Diese Gerade $\overleftrightarrow{I~ II}$ schneidet aber die Gerade mit Steigung $n_2$ (in einem gedrehten Koordinatensystem) in $IV$. Die waagerechte Gerade eine Einheit unter $\overleftrightarrow{I\quad II}$ schneidet diese Gerade in $III$, und die vertikale Geraden durch $III$ und $IV$ sind $|n_2|$ Längeneinheiten voneinander entfernt. Daher ist der Punkt $K$, in dem sich die horizontale Gerade durch $VI$ und die vertikale Gerade durch $II$ schneiden, $|n_2|$ Einheiten links und $|m_2|$ Einheiten unterhalb des Punktes $IV$. Wählt man also $IV$ als Ursprung eines neuen Koordinatensystems, mit der x-Achse nach rechts und der y-Achse nach oben orientiert, hat $K$ die Koordinaten $(n_2|m_2)$. Der Punkt $J$, in welchem die vertikale Gerade durch $IV$ die horizontale Gerade eine Einheit unterhalb von $\overleftrightarrow{I~ II}$ schneidet, ist durch die Gleichung $y=-1$ in diesem Koordinatensystem gegeben. Die Faltung, die durch Falten von $J$ auf diese Gerade so, dass die Faltung durch $K$ verläuft, löst daher die Gleichung, und ihre Steigung ist 
\[v_1=2\cos\frac{4\pi}{17}\]
Die Faltung senkrecht zu dieser durch den Mittelpunkt $M$ des Quadrates in Schritt 10 hat also die Steigung $\-\frac{1}{2\cos\frac{4\pi}{17}}$. Da der Umkreis des 17-Ecks der Inkreis des Papierquadrates sein soll, nehme man nun an, dass die Kanten des Quadrates $2$ Einheiten lang sind. Die Entfernung von $M$ zur horizontalen Geraden durch $J$ ist dann eine halbe Längeneinheit, und diese Gerade schneidet die Gerade mit Steigung $\-\frac{1}{2\cos\frac{4\pi}{17}}$ in einem Punkt $L$, dessen Entfernung von der vertikalen Geraden durch den Mittelpunkt des Papierquadrates $\cos\frac{4\pi}{17}$ ist.\\
Nimmt man an, dass der Punkt $1$ in Schritt 11, in dem sich die horizontale Linie durch den Mittelpunkt des Papierquadrates und dessen linke Kante schneiden, eine Ecke des 17-Ecks ist, dann erhält man die Ecken $3$ und $16$ durch Falten von $1$ auf die vertikale Gerade durch $L$ so, dass die Faltlinie durch $M$ geht. Die dabei entstehenden Faltlinien enthalten die Ecken $2$ und $17$. Das Falten von $1$ auf die untere Faltlinie ergibt daher die Ecke $2$. $\Delta(M12)$ ist somit ein Bestimmungsdreieck des 17-Ecks.
\end{proof}
Diese Konstruktion soll das Kapitel abschließen. Im nächsten Kapitel soll ein Einblick in mögliche Erweiterungen des vorgestellten Origami-Axiomensystems gegeben werden.

\newpage
\section{Runde Faltungen und die Konstruktion von $\pi$}\footnote{\cite{hull}}
Über den Verlauf dieser Arbeit war sehr schön zu erkennen, wie die Menge der konstruierbaren Objekte von der Wahl des zugrunde liegenden Werkzeuges abhängt. Natürlich kann man auch das vorgestellte Origami-Axiomensystem um weitere Grundkonstruktionen erweitern.\\
Runde Faltungen erscheinen zunächst sehr spannend, da man sich wohl gar nicht so richtig vorstellen kann, wie man diese realisieren soll. Es ist schon schwierig genug, solche Faltungen überhaupt praktisch zu definieren, gleichwohl die Theorie dahinter relativ wenig Probleme macht. Die runden Faltungen bringen jedoch neue Möglichkeiten ins Origami, zum Beispiel wird es nun einfach, die über $\Q$ transzendente Zahl $\pi$ zu konstruieren.\\

Wie faltet man nun überhaupt einen Kreis? Auf jeden Fall können Kreise nicht flach gefaltet werden, sodass sie nicht einfach durch das Aufeinanderfalten von Geraden oder Punkten entstehen können. Man könnte sich mit Werkzeugen behelfen, zum Beispiel mit einem Stechzirkel die Kreislinie in das Papier einritzen. Aber eigentlich ist Origami nur Falten mit Hilfe der eigenen Finger. Eine mögliche Methode zur Faltung eines Kreises ist die folgende:
\begin{figure}[h]
\includegraphics[width=6cm]{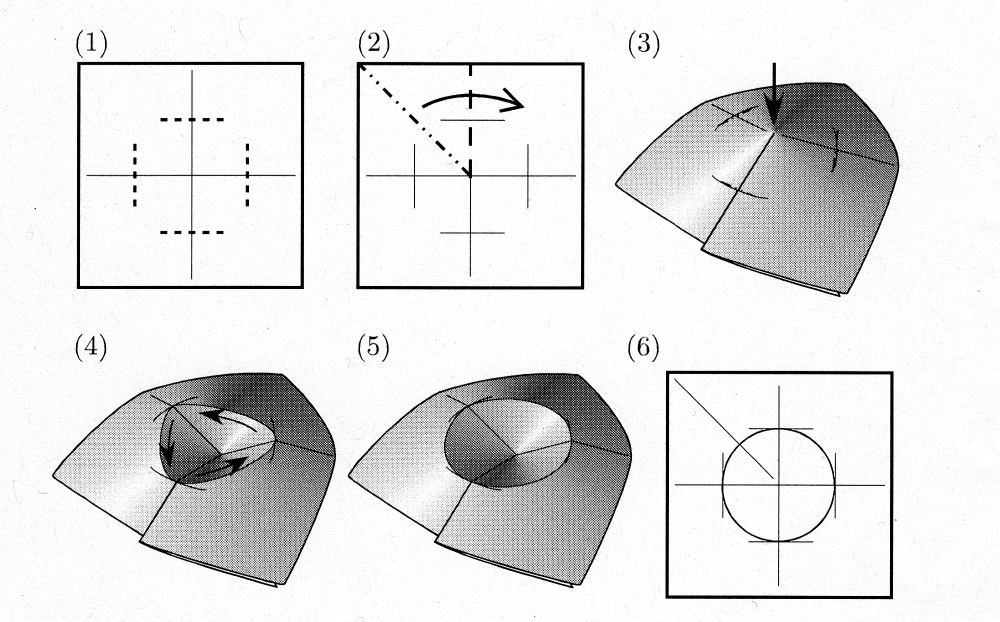}
\caption[\cite{hull}, Seite 4]{Falten eines Kreises}
\label{zul2}
\end{figure}
\bi
\item Man faltet in das quadratische Papier wie dargestellt die Mittellinien als Koordinatenachsen und den Radius entlang der beiden Achsen ein.
\item In einen Quadranten faltet man die Diagonale und das Blatt dann entlang dieser. So entsteht eine Art Kegel.
\item Die Kegelspitze wird nun nach unten gedrückt.
\item Die Faltungen nähert man nun mit den Fingern an einen Kreis mit dem gegebenen Radius an.
\item Mit viel Übung funktioniert das recht gut.
\item Es ergibt sich, wenn man das Blatt wieder flach faltet, der gesuchte Kreis.
\ei
Zugegeben, das Ergebnis sieht erst nach viel Übung nach einem Kreis aus, ein etwas stabileres Papier vereinfacht die Konstruktion beträchtlich. Es erscheint jedoch illusorisch, mit Hilfe nur der Hände einen exakten Kreis falten zu können, alleine eine Sprache zur Beschreibung derartiger Konstruktionen zu entwickeln, dürfte eine größere Aufgabe sein. Aber, angenommen, man könnte einen sauberen Halbkreis falten, dann ergäbe sich sofort eine sehr einfache Möglichkeit zur Konstruktion von $\pi$.

\begin{thm}{Falten von $\pi$}\\
\begin{figure}[h]
\includegraphics[width=8cm]{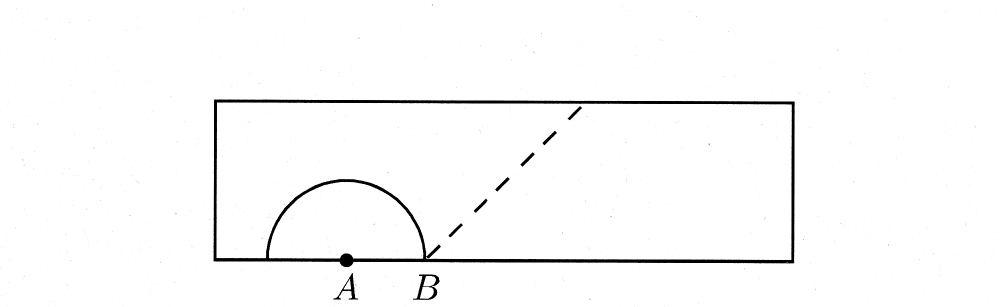}
\caption[\cite{hull}, Seite 2]{Falten von $\pi$ 1}
\label{zul2}
\end{figure}

Man faltet zunächst einen Halbkreis mit Radius 1 als Talfalte um einen Punkt $A$ auf der langen Seite eines Papierstreifens und markiert am rechten Schnittpunkt mit der Blattkante den Punkt $B$. Als nächstes faltet man eine Talfalte von einem Ende des Halbkreises aus, die einen Winkel von $45^\circ$ mit der Blattkante einschließt, und faltet diese nicht auf! Nun faltet man wie in folgender Skizze die Blattkante entlang des Tal-Halbkreises:

\begin{figure}[h]
\includegraphics[width=6cm]{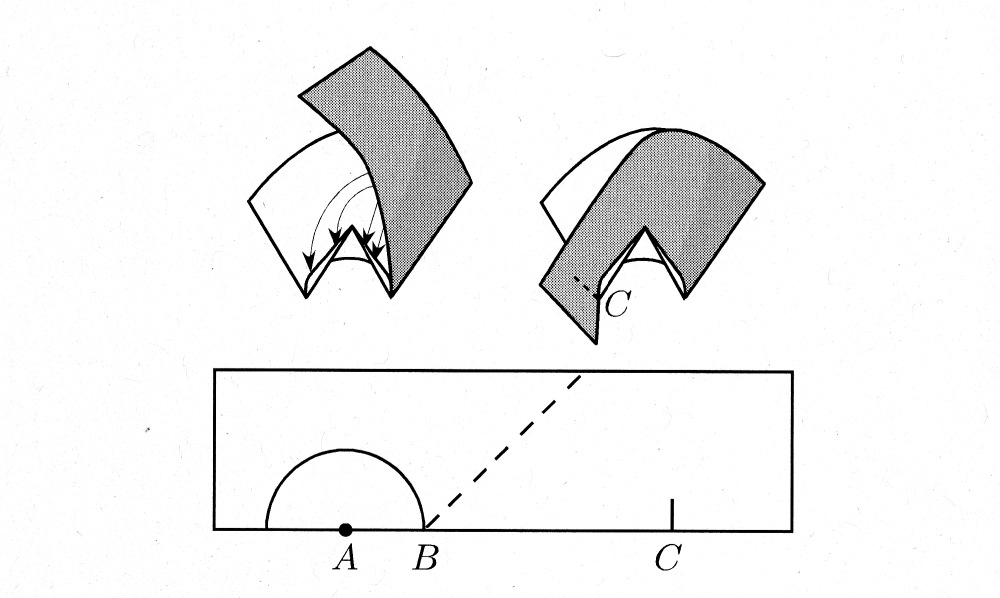}
\caption[\cite{hull}, Seite 3]{Falten von $\pi$ 2}
\label{zul2}
\end{figure}
\newpage
Zur weiteren Verdeutlichung hier eine Fotografie der Faltprozedur an dieser Stelle:
\begin{figure}[h]
\includegraphics[width=8cm]{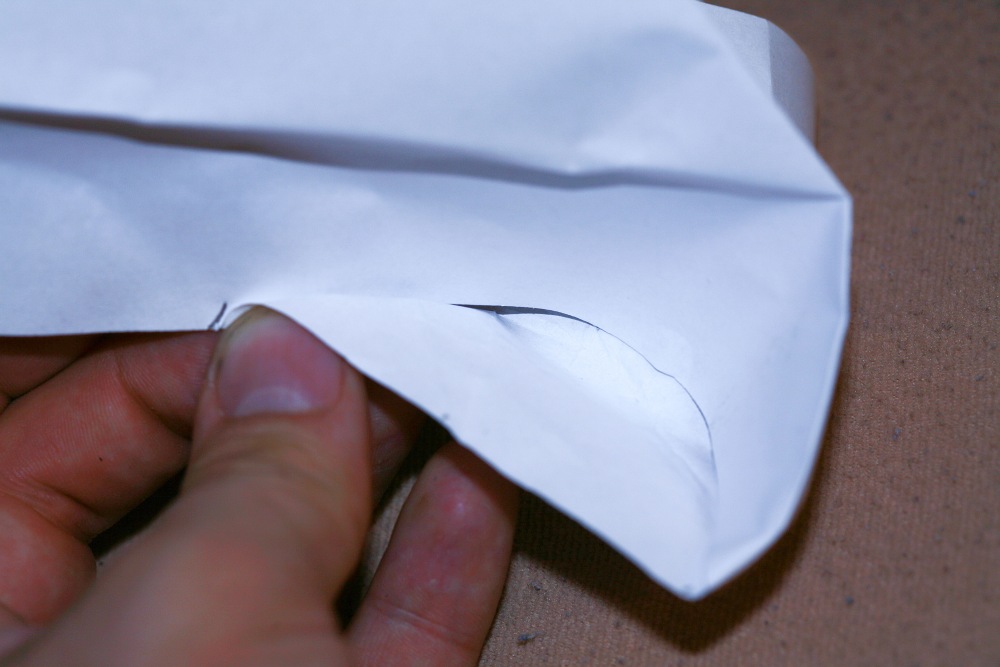}
\caption[Eigene Fotografie]{Falten von $\pi$ 3}
\label{zul2}
\end{figure}

Man folgt also nun mit der Blattkante der Kreislinie, und da, wo die Blattkante  beim linken Schnittpunkt des Kreises mit der Kante, ankommt, markiert man den Punkt $C$. Dann ist $BC=\pi$.
\end{thm}

\begin{proof}
Offenbar ist die Strecke $BC$ der Umfang eines Halbkreises mit Radius 1. Damit hat sie die Länge $\pi$.
\end{proof}
In der Praxis ist diese Faltung absolut nichttrivial in der Durchführung. Bei dem Versuch, der oben fotografiert wurde, erhielt der Autor einen Wert von $\pi\approx 3,03$ und somit einen relativen Fehler von $3,6\%$. Dies verdeutlicht, wie viel Arbeit wohl noch in die Erforschung kreisförmiger Faltungen und in die Systematisierung einer Falttechnik dafür gesteckt werden muss, bis man diese axiomatisch sauber begründen und dann mathematisch untersuchen kann. Wünschenswert wäre das auch insofern, da, sobald man $\pi$ konstruieren könnte, natürlich auch das Problem der Quadratur des Kreises lösbar wäre.

\newpage
\section{Origami und die Schere}

Eines der überraschendsten Ergebnisse der mathematischen Untersuchungen von Origami ist das Fold and One-Cut-Theorem. Es ist so beeindruckend, dass man schon fast an einen Zaubertrick glauben könnte. Zunächst soll der Satz formuliert werden, um danach zu präzisieren, was damit genau gemeint ist, und danach soll ein Beweis skizziert werden.

\begin{defn}
Eine flache Faltung eines Blattes ist eine Faltung, an deren Ende nur noch parallele Papierschichten übereinander liegen, die man sich als idealisiert in eine Ebene gepresst vorstellen kann.\footnote{\cite{rourke}, Seite 57}
\end{defn}

\begin{thm}{Fold and One-Cut}\footnote{\cite{rourke}, Seite 78}\\
Jede Figur, die durch gerade Strecken begrenzt wird (Streckenzug), kann so flach gefaltet werden, dass eine gerade Schnittlinie existiert, sodass die Schere beim Schneiden entlang dieser Linie genau diese Figur(en) ausschneidet (und sonst nichts!).
\end{thm}

Dieser Satz klingt so unglaublich, dass er noch einmal grafisch aufbereitet werden soll:

\begin{figure}[h]
\includegraphics[width=8cm]{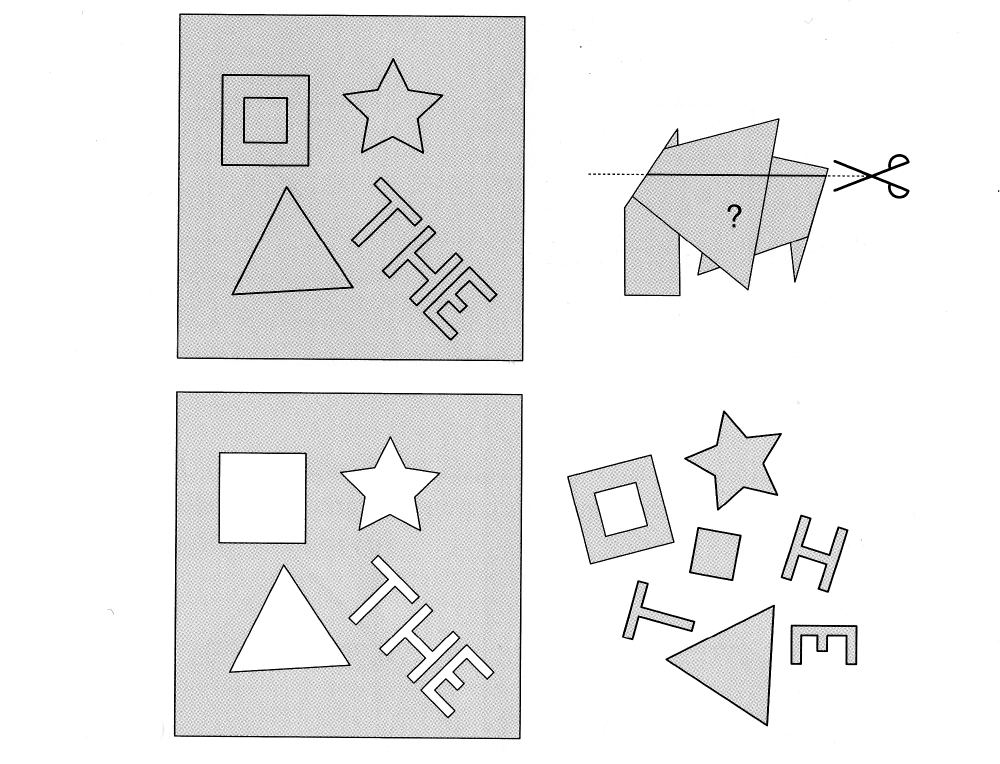}
\caption[\cite{rourke}, Seite 73]{Fold and One Cut}
\label{zul2}
\end{figure}

Man zeichnet also beliebige Figuren, bestehend aus geraden Linien, auf, dann existiert irgendeine Faltung, die es erlaubt, dass man mit einem geraden Schnitt genau die Umrisse, die man aufgezeichnet hat, ausschneidet. Für den Rest der Arbeit ist mit der Abkürzung "`FaoC-Faltung"' eine Faltung gemäß dieses Satzes gemeint.\\
Der Satz an sich ist ein recht neues Resultat aus dem 21. Jahrhundert. Dass eine solche flache Faltung für gewisse Körper möglich ist, ist aber schon länger bekannt. In seinem Buch "`Paper Magic"' aus dem Jahre 1922 beschrieb der bekannte Zauberkünstler Houdini eine FaoC-Faltung für einen fünfzackigen Stern. Dieses Faltmuster ist recht einfach nachzufalten und nachzuvollziehen, da es auf einfachen Symmetrieüberlegungen des {\em regulären} 5-Ecks basiert\footnote{ebd, Seite 73}:
\begin{figure}[h]
\includegraphics[width=8cm]{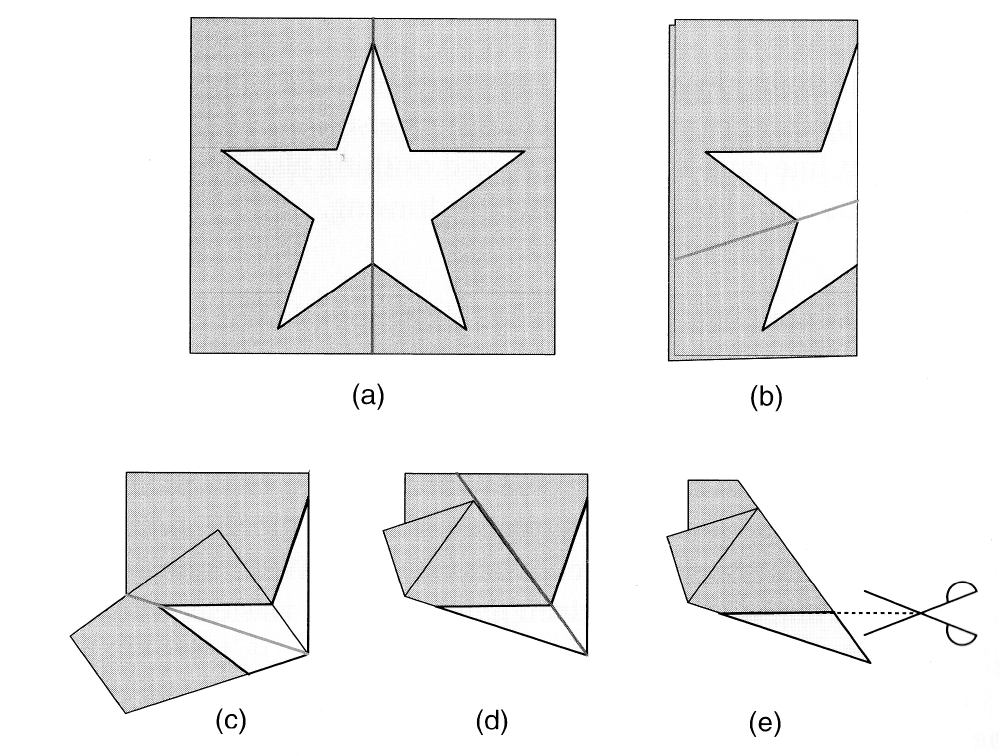}
\caption[\cite{rourke}, Seite 74]{FoaC-Faltung des Pentagramms.}
\label{zul2}
\end{figure}

So faszinierend der Satz klingt, so kompliziert ist sein Beweis. Dieser durchlebte eine längere Entwicklungsgeschichte, bis sämtliche möglichen Fälle abgedeckt waren.\footnote{\cite{rourke}, Seite 144} Eine erste Version des Beweises von Demaine, Martin, Demaine und Lubiw aus dem Jahre 1999\footnote{"`Folding and one straight cut suffice"' in "`Proceedings of the 10th Annual ACM-SIAM Symposium Discrete Algorithms, Seiten 891-892} war nicht vollständig. Ein prinzipiell vollständiger Beweis, der Kreispackungen verwendet, wurde 2002 von Bern, Demaine, Eppstein und Hayes gegeben\footnote{\cite{ber02}}, dieser enthielt jedoch kleinere Fehler, die von Bern und Hayes 2009 in einem Vortrag bei einer Konferenz in Sao Paolo korrigiert werden konnten\footnote{\cite{ber09}}. Den beiden letzteren Artikeln soll im nun anstehenden Beweis gefolgt werden. Dabei bezeichnet im Folgenden das Wort "`Ecke"' gelegentlich auch einen Schnittpunkt mehrerer Faltlinien in der Ebene.

\begin{proof}
Der Beweis folgt zunächst \cite{ber02}. Gegeben ist also ein Polygon $P$, dass zunächst Löcher haben kann oder auch aus mehreren, nicht zusammenhängenden Komponenten bestehen kann, und ein Rechteck $R$, das groß genug ist, um $P$ strikt zu enthalten, das heißt, $P$ soll den Rand von $R$ nicht berühren.. Die Beweisstrategie ist nun, Kreisscheiben so auf $R$ zu packen, dass die Mittelpunkte eine Aufteilung in Drei- und Vierecke des Polygons $P$ ergeben. Man faltet dann jedes Dreieck oder Viereck innerhalb von $P$ nach oben, und jedes außerhalb nach unten aus der Papierebene, so, dass die benachbarten Polygone zu dieser Orientierung der Faltungen passen. Ein geraden Schnitt durch die Papierebene trennt dann Inneres von äußerem.\footnote{\cite{ber02}, Introduction}\\

\subsection{Kreisscheiben-Überdeckung}
Sei nun $PR$ der ebene, aus geraden Linien bestehende Graph, der durch Vereinigung von der Ränder von $P$ und $R$ entsteht. In diesem Absatz soll skizziert werden, wie man Kreisscheiben so auf diesen Graph packt, dass folgende Forderungen erfüllt sind:

\bi
\item Jede Kante von $PR$ ist eine Vereinigung von Radien von Kreisscheiben.
\item Die Kreisscheiben induzieren eine Aufteilung von $R$ in Drei- und Vierecke.
\ei

Man beginnt mit Kreisscheiben, die maximal die Ränder gemeinsam haben. Bezeichne "`Lücke"' eine zusammenhängende Teilmenge von $R$ ohne die Kreisscheiben. Eine Lücke, die von 3 Bögen umrandet ist, nennt man eine 3-Lücke, eine von 4 Bögen beschränkte eine 4-Lücke.\footnote{ebd} Disk-Packing\\
Man beginnt nun damit, eine Kreisscheibe mit dem Mittelpunkt auf jede Ecke $v$, auch die von $R$ zu legen, und wähle als Radius die Hälfte der Entfernung zur $v$ nächstliegenden, nicht auf $v$ liegenden Ecke. Es entsteht somit eine Unterteilungsecke von Grad 2 mit spitzem Winkel bei jedem Schnittpunkt einer Kreisscheibe und einer Kante von $PR$.
Nun betrachtet man die Kanten von $PR$, die noch nicht mit Kreisen bedeckt sind. Man nennt eine solche Kante überfüllt, wenn eine Kreisscheibe mit dieser Kante als Durchmesser eine solche Kreisscheibe einer anderen Kante des Restes von $PR$ schneidet. Man teilt diese nun so weit auf, bis keine Kante mehr überfüllt ist. Dann fügt man die Kreisscheiben mit den Durchmessern dieser neuen Kantensegmente hinzu, so dass jede Kante Vereinigung von Durchmessern ist, wie gefordert. Streng genommen würde es sogar ausreichen, die Kanten von $P$ so zu bedecken. Es ist nun möglich, So Kreise hinzuzufügen, dass alle Lücken zwischen Kreisscheiben 3- oder 4-Lücken sind. Das kann am Computer durch Berechnung des Voroni-Diagramms der bereits vorhandenen Kreisscheiben geschehen, das man wiederholt durch Kreisscheiben mit maximal möglichem Radius füllt und neu berechnet, bis man an dieser Stelle angekommen ist. Eppstein gab 1997 einen Algorithmus hierfür an, der mit $O(n \log n)$ Operationen bei $n$ Kreisscheiben auskommt.
Das nun folgende Bild \ref{foc3} zeigt links ein Beispiel für eine solche Überdeckung, und rechts die dadurch induzierte Zerlegung in Drei-und Vierecke, die entsteht, wenn man die Mittelpunkte sich berührender Kreisscheiben verbindet.
\begin{figure}[h]
\includegraphics[width=8cm]{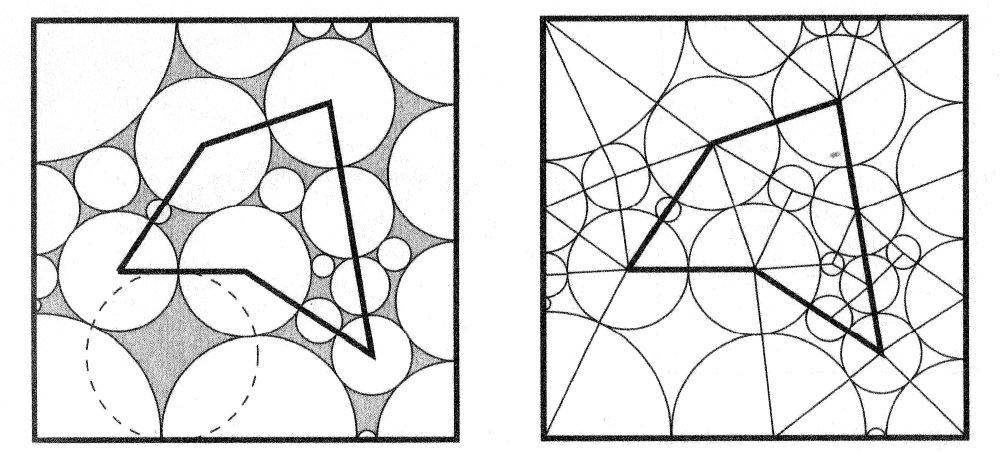}
\caption[\cite{ber02}, Disk Packing]{Beispiel einer Kreisscheiben-Überdeckung}
\label{foc3}
\end{figure}
\subsection{Moleküle}\footnote{\cite{ber02}, Molecules}\\
Ein Molekül ist eine typischerweise flache Faltung eines Polygons, dass als Baustein für größere Origami-Faltungen verwendet werden kann. Die Faltung der Dreiecke unserer Zerlegung, die hier verwendet werden soll, ist das sogenannte Hasenohr-Molekül. In dieser Faltung trifft eine Bergfaltung auf jede Ecke des Dreiecks, und da diese Faltlinien die Winkelhalbierenden des Dreiecks sind, schneiden sie sich in einem Punkt.Weiterhin faltet man Talfalten so, dass diese durch jeweils den Punkt der Dreiecksseiten gehen, in denen sie die Kreisscheiben berühren, dies alles ergibt einen vertikalen Grat senkrecht zur ursprünglichen Papierebene. Der Schnittpunkt der sechs Faltungen, die Spitze dieses Grates ist dann der Inkreismittelpunkt des Dreiecks. Das Bild \ref{foc4} zeigt diese Faltung, und auf der rechten Seite sieht man ein an einen Seestern erinnerndes Zwischenstadium.
\begin{figure}[h]
\includegraphics[width=8cm]{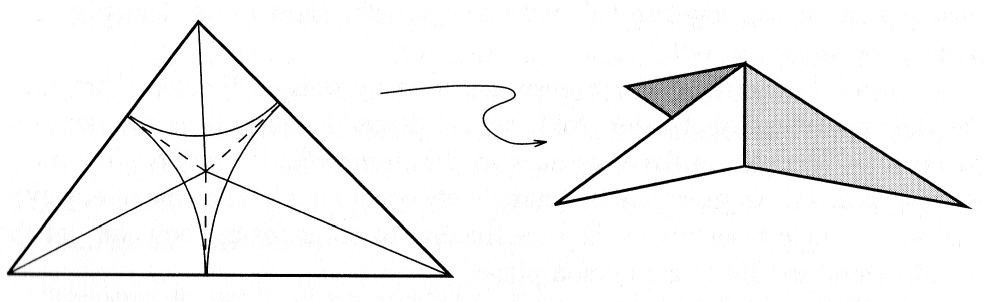}
\caption[\cite{ber02}, Molecules]{Hasenohr-Molekül}
\label{foc4}
\end{figure}
An diesem Punkt des Beweises sind die Orientierungen der Talfalten prinzipiell austauschbar. In größeren Origami-Faltungen mag es nötig sein, die Orientierung umzukehren, um den Satz von Maekawa, der besagt, dass sich an jeder Ecke im Inneren des Papiers die Anzahl der Bergfalten und der Talfalten um $2$ unterscheiden muss, wenn es sich denn um eine flache Faltung handelt, zu erfüllen.\\
Die Vierecke sollen nun ganz ähnlich gefaltet werden, als ein "`vierzackiger Seestern mit Tal in der Mitte"', das man als Kreuzblech-Molekül bezeichnet.
\begin{figure}[h]
\includegraphics[width=8cm]{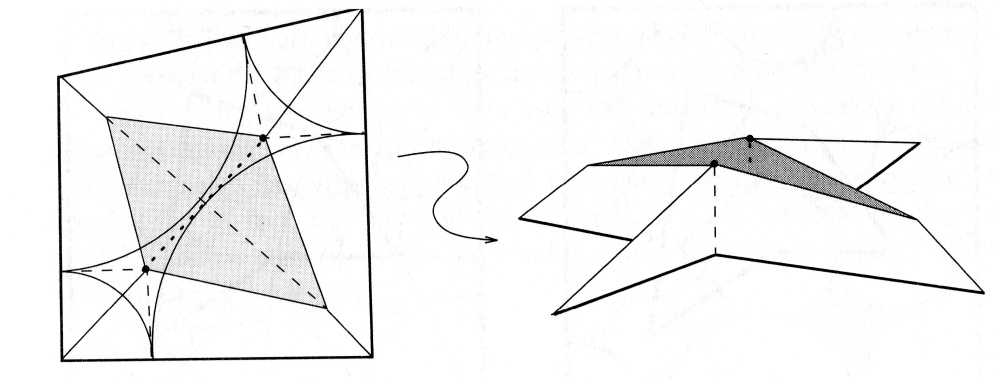}
\caption[\cite{ber02}, Molecules]{Kreuzblech-Molekül}
\label{foc5}
\end{figure}
In dieser Faltung verlaufen Bergfalten ein Stück weit entlang der Winkelhalbierenden, unterbrochen vom Kreuzblech, einem Viereck innerhalb des ursprünglichen Vierecks, wie in Abbildung \ref{foc5} gezeigt. Das Kreuzblech wird durch eine seiner Diagonalen in zwei Dreiecke geteilt, die Diagonalen sind Talfalten, und jede Hälfte des großen, ursprünglichen Vierecks wird so in eine Art Hasenohr-Molekül-Form gefaltet. Auch hier treffen sich die Talfalten der Berührpunkte in der Spitze eines zentralen Grates, senkrecht zur ursprünglichen Papierebene. Auch hier ist die Orientierung der Faltungen prinzipiell austauschbar.\\
Zwei der Ecken des Kreuzbleches sind dadurch vorgegeben, dass die Talfalten vom Berührpunkt mit den Kreisscheiben senkrecht verlaufen müssen. Diese Ecken sollen als {\em Tangenten-Punkte} bezeichnet werden. Die anderen zwei Ecken sind nicht zwingend festgelegt, sie müssen aber auf den Winkelhalbierenden des ursprünglichen Viereckes liegen, damit dessen Grenzen sich in eine gemeinsame Ebene falten lassen. Eine Möglichkeit zum Festlegen dieser Punkte ist das Einschieben eines inneren Vierecks, das aus zu den ursprünglichen Seiten parallelen und äquidistanten Seiten besteht. In Abbildung \ref{foc6} ist das äußere Viereck mit $A,B,C,D$, das neue innere mit $P,Q,R,S$ als Ecken eingezeichnet. 
\begin{figure}[h]
\includegraphics[width=8cm]{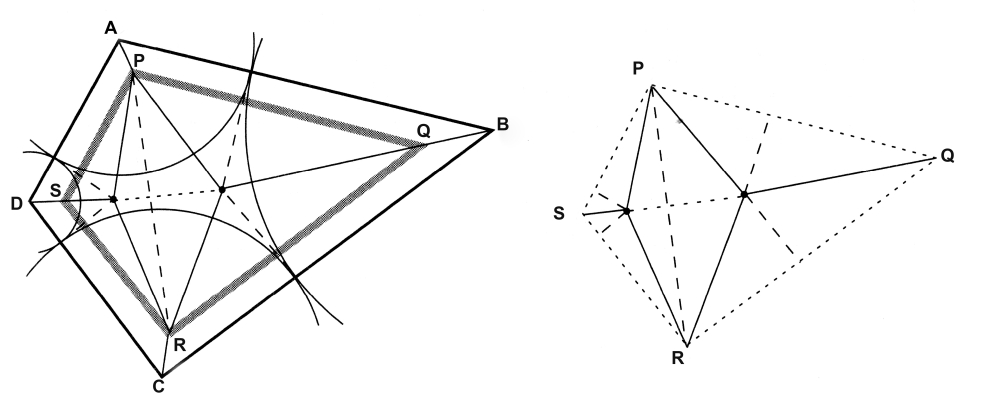}
\caption[Bearbeitung von \cite{ber02}, Molecules]{Kreuzblech-Molekül 2}
\label{foc6}
\end{figure}

Das innere Viereck bildet bei der Faltung des Kreuzblech-Moleküls  einen kleinen Seestern, so, dass $PQRS$ und $PR$ in die gleiche Ebene gefaltet werden. Tatsächlich handelt es sich bei der Kreuzblech-Faltung, eingeschränkt auf $PQRS$, ja nur um zwei Hasenohr-Moleküle, wie in der rechten Skizze dargestellt. Also liegen die Tangenten-Punkte auf den Inkreismittelpunkten der Dreiecke $\Delta(P,Q,R), \Delta(P,R,S)$, und dies bestimmt die Größe von $PQRS$ eindeutig.\\

Was noch zu zeigen ist, ist die Tatsache, dass alle Vierecke, die aus 4-Lücken entstehen, als Kreuzblech-Molekül gefaltet werden können. Dazu ist zu zeigen, dass die Dreiecke $\Delta(P,Q,R), \Delta(P,R,S)$ mit Inkreismittelpunkten bei den Tangenten-Punkten wirklich innerhalb von $ABCD$ liegen, also dass die Anforderungen an das Kreuzblech sich nicht widersprechen.\\
Dazu nehme man zunächst an, dass die Tangenten-Punkte eindeutig sind, und ihre Verbindungsgerade sei $l$. Die Gerade $l$ ist die Menge aller Punkte mit gleicher Potenz\footnote{Die Potenz eines Punktes zu einem Kreis ist das Quadrat der Entfernung vom Mittelpunkt des Kreises minus dem Quadrat des Kreisradius.} von den Kreisscheiben um $A$ und $C$, und verläuft daher zwischen diesen Kreisen. Die Winkelhalbierende des Winkels zwischen $l$ und der Talfalte senkrecht zu $\overleftrightarrow{BC}$ legen dann den Punkt $R$ fest. Da $l$ oberhalb der Scheibe um $C$ verläuft, liegt auch $R$ oberhalb von $C$ auf der Winkelhalbierenden bei $C$. Daher liegt $PQRS$ tatsächlich in $ABCD$. In dem Fall, dass sich die Kreise um $A$ und $C$ berühren, ist $PQRS$ gleich $ABCD$ und das Kreuzblech-Molekül wird zu zwei Hasenohr-Molekülen.\\

Aber was passiert, wenn die beiden Tangenten-Punkte aufeinander liegen? In diesem Extremfall kann man eine spezielle Eigenschaft von 4-Lücken ausnutzen. Die Berührpunkte von vier Kreisscheiben, die sich zyklisch berühren, liegen auf einem Kreis. Abbildung \ref{foc3} zeigt links gestrichelt einen solchen Kreis. Das bedeutet aber, dass sich die Winkelhalbierenden dieses Viereckes in einem Punkt $O$ schneiden, dem Mittelpunkt des Kreises durch die Berührpunkte! In diesem Extremfall zieht sich also $PQRS$ auf den Punkt $O$ zusammen, und die Talfalten von den Berührpunkten  und die Bergfalten entlang der Winkelhalbierenden treffen sich alle in einer flach faltbaren Ecke.\\

\subsection{Verbinden der Moleküle}\footnote{\cite{ber02}, Joining Molecules}\\

Nun bleibt nur noch zu zeigen, welche Orientierung den einzelnen Faltungen zugewiesen werden muss, so dass benachbarte Moleküle zusammenpassen und jede Ecke Maekawa's Satz erfüllt.\\
Man sucht also eine Faltung von $R$, die wie zwei Bücher von Papierlaschen in dreieckiger Form aussieht, eines oberhalb und eines unterhalb der ursprünglichen Papierebene. Die "`Klebekante"' der Seiten des Buches sei als Achse bezeichnet. Die Moleküle, die innerhalb des Polygons $P$ liegen, bilden das obere Bündel, diejenigen außerhalb das untere. Der Rand des Polygons selbst wird nicht gefaltet, und die Teil-Polygone, die die Grenze des Polygons überschreiten, enthalten jeweils ein Dreieck von zwei ursprünglichen Molekülen, und liegen so in beiden Bündeln.\\
Winkelhalbierende Kanten innerhalb des Polygons werden zu Bergfalten, und solche außerhalb werden zu Talfalten. Andere Kanten des Faltmusters erhalten zunächst eine Standard-Orientierung, die sich aber am Ende noch ändern kann. Die Standard-Orientierung einer Kante, die durch einen Berührpunkt verläuft, einer {\em Berührkante}, und die einer {\em Seitenkante}, also einer solchen, die entlang einer Seite eines Drei- oder Vierecks läuft,  ist Tal innerhalb des Polygons, und Berg außerhalb. Seitenkanten entlang des Randes des Polygons werden ja gar nicht gefaltet.\\
Nun hat jede Ecke innerhalb des Faltmusters die gleiche Anzahl von Bergen und Tälern. Die Ecken innerhalb des Polygons $P$ benötigen einen weiteren Berg, und die Ecken außerhalb des Polygons $P$ ein weiteres Tal, damit die Moleküle korrekt nach oben oder unten gefaltet werden.\\
Sei $G$ der ebene Graph, der aus der Zerlegung durch Entfernung aller Winkelhalbierenden-Kanten und aller Kanten entlang der Grenzen des Polygons $P$ entsteht. Man möchte nun eine Menge von Kanten $M$ finden so, dass jede Ecke von $G$ im Inneren des Rechtecks $R$ auf genau einer Kante von $M$ liegt. Durch Umorientierung der Kanten von $M$ kann man sicher stellen, dass jede Ecke Maekawa's Satz erfüllt. Alle Ecken, auch die entlang des Polygons $P$, welche zwei Kanten im Vergleich zur ursprünglichen Zerlegung von $R$ verloren haben, erfüllen auch den Satz von Kawasaki\footnote{Jede Ecke im Inneren des Papiers hat eine alternierende Winkelsumme von $180^\circ$}.\\
Nun wird gezeigt, wie man das Problem des Findens von $M$ lösen kann.\\

Sei dazu $T_C$ der Baum aller Seitenkanten so, dass:
\bi 
\item $T_C$ keine Kanten entlang der Grenzen von $R$ oder $P$ enthält.
\item $T_C$ spannt alle inneren Ecken von Molekülen auf
\item $T_C$ spannt genau eine Ecke auf der Grenze von $R$ auf, die man als Wurzel betrachten kann.
\ei

Wenn man nun das Papier entlang von $T_C$ schneidet, erhält man einen Baum von Molekülen $T_M$ mit einem der Moleküle, die auf der Wurzel von $T_C$ liegen als Wurzel. $M$ enthält nun zwei Arten von Kanten:  Jede Berührkante von der Mitte eines Moleküls zu der Seite seines entsprechenden Elements in $T_M$, und jede Seitenkante von einer Ecke zum zugehörigen Element in $T_C$.\\

Nehme man nun an, man habe bereits enlang der Kanten von $T_C$ geschnitten. Weiterhin stelle man sich vor, man baue die flache Faltung Molekül für Molekül entsprechend $T_M$ auf. Das Wurzel-Molekül von $T_M$ bildet dann ein Buch mit kollinearen Kanten in der ursprünglichen Papierebene. Jedes folgende Molekül fügt einen Abschnitt aus 3 oder 4 Laschen zwischen zwei Laschen des bis dahin konstruierten Buches ein. Die erste und letzte Seite dieses Abschnittes werden an ihre angrenzenden Seiten geklebt, so dass ein Viereck die Dicke zweier alter Laschen um zwei Laschen erhöht.\\
So macht man entlang des Baumes weiter. Wann immer man die Grenzen von $P$ überschreitet, klebt man den nächsten Abschnitt über oder unter das vorhergehende Molekül ein, so dass die Grenze von $P$ selbst nicht gefaltet wird. Wenn man so alle Moleküle verbunden hat, haben wir nun zwei Bücher, eines oberhalb und eines unterhalb der ursprünglichen Papierebene.

Nehme man nun an, man verklebe die Schnittkanten wieder in der umgekehrten Reihenfolge entlang von $T_C$. Vorher bilden die Schnitte, die zu einem Blatt von $T_C$ innerhalb von $P$ führen, die untere Kante zweier "`Armbeugen"', wie in Abbildung \ref{foc7}.
\begin{figure}[h]
\includegraphics[width=8cm]{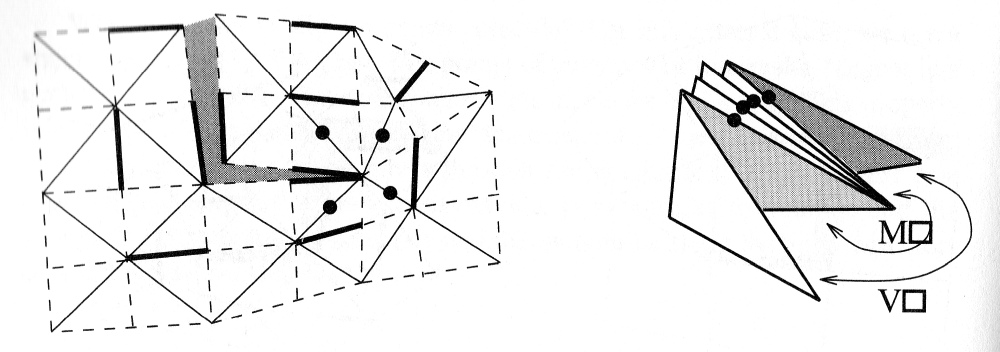}
\caption[Bearbeitung von \cite{ber02}, Joining Molecules]{Zusammenkleben}
\label{foc7}
\end{figure}

Eine solche Armbeuge besteht dabei aus je einer Schicht angrenzender Laschen. Klebt man die erste und letzte Schicht der dazwischen liegenden Laschen zusammen, so entsteht eine Bergfalte, wie in der Standard-Orientierung für Seitenkanten gefordert. Zusammenkleben der verbleibenden zwei Seiten des Schnittes ergibt eine Talfalte, ebenfalls in Übereinstimmung mit unseren Forderungen. Zusammenkleben eines Schnittes, der zu einem Blatt von $T_C$ führt, schließt zwei Armbeugen und reduziert die Anzahl der Laschen im Buch um zwei. 

Ohne weitere Einschränkungen kann es hier aber zu Problemen kommen, denn diese Konstruktion kann sich schneidende Verklebungen erfordern\footnote{\cite{ber09}, Seite 623}. Um die Konstruktion zu reparieren, muss man sicher stellen, dass der Baum der Moleküle die Beziehungen der Grenz-Komponenten untereinander respektiert. Daher müssen die Moleküle innerhalb des Polygons $P$ einen wohldefinierten Unterbaum innerhalb des Baumes aller Moleküle bilden. Wenn das Polygon $P$ ein Loch enthält, das innerhalb einer "`Insel"' innerhalb eines weiteren Loches liegt, dann muss jede Schicht von Molekülen einen wohldefinierten Unterbaum des Baumes darüber bilden. Man versieht dazu jede Grenz-Komponente mit einem Saum der Dicke $\epsilon$, dann falten sich die tiefsten Moleküle zu einem Buch, wie in Abbildung \ref{foc7}. Die zweittiefesten Moleküle falten sich zu Büchern, die die tiefsten Moleküle als Kapitel, also als zusammenhängende Menge von von Laschen innerhalb der größeren Bücher. Der Saum verschiebt die gemeinsamen Achsen der tiefsten Moleküle so, dass sie $\epsilon$ über und parallel zur Achse der zweittiefesten Moleküle. So macht man weiter, und baut Bücher innerhalb von Büchern, wobei jede Schicht zu der vorhergehenden um $\epsilon$ verschoben ist.\\
 Um nun alle Grenzkomponenten auf einmal auszuschneiden, und nichts anderes, faltet man das Innere aller Moleküle nach außerhalb der Achsen falten. Um dies zu tun, reduziert man die Höhe der einzelnen Moleküle so, dass sie alle kleiner als $\epsilon$ sind, entsprechen der in Abbildung \ref{foc8} skizzierten Technik.
\begin{figure}[h]
\includegraphics[width=8cm]{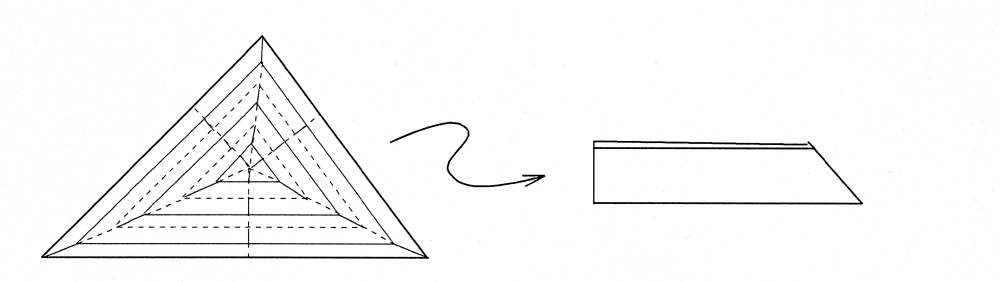}
\caption[Bearbeitung von \cite{ber09}, Seite 625]{Die Höhe einer Buch-Faltung lässt sich beliebig klein machen.}
\label{foc8}
\end{figure}

Nun legt also der Saum die Höhe der ganzen Faltung fest, diese wird für ein einfaches Polygon ohne Löcher in der Gegend von $\epsilon$ sein. Letztendlich kann man das ganze Buch mittel Faltungen parallel zu den Achsen falten, so dass sich alle Achsen auf eine gemeinsame Schnittlinie bringen lassen.
\end{proof}

Damit ist der Satz bewiesen, und der Algorithmus ist nicht zu komplex, hängt er doch linear von der Anzahl der anfangs verwendeten Kreisscheiben ab.\footnote{\cite{ber02}, Seiten 25ff}

Abschließend sollen noch einige FoaC-Faltmuster angegeben werden, um den Leser zur selbständigen Auseinandersetzung mit dieser "`mathematischen Zauberei"' zu bringen. Es ist natürlich mit Hilfe des angegebenen Algorithmus nun jedem selbst möglich, ein Faltmuster für FoaC-Faltungen für einen beliebigen Polygonzug zu erstellen. Ein weiteres Beispiel aus Berns Artikel soll diesen Abschnitt beginnen:

\newpage
\begin{figure}[h]
\includegraphics[width=12cm]{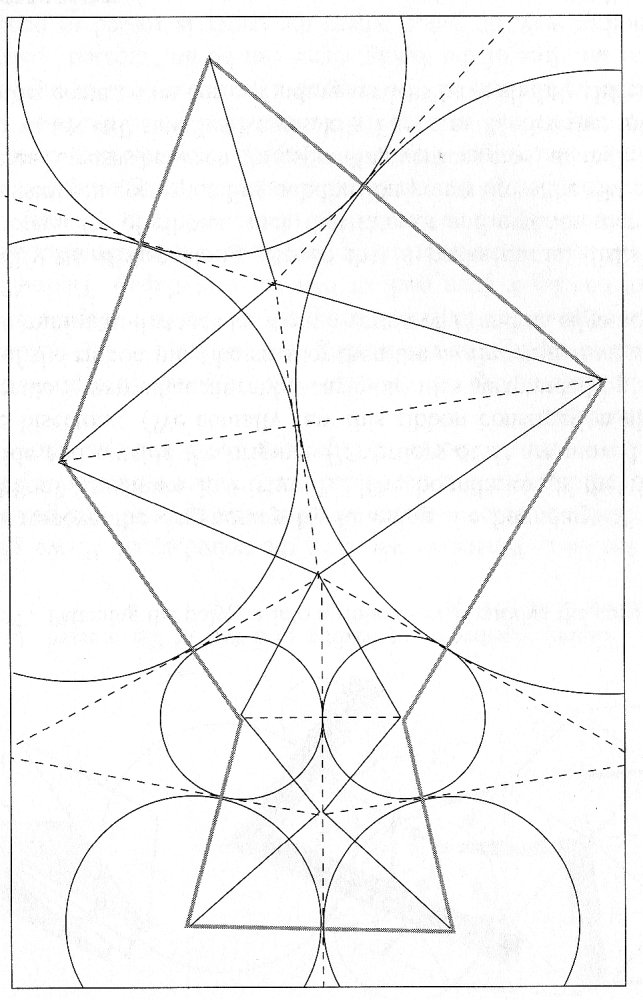}
\caption[\cite{ber02}, Seite 26]{FoaC-Faltmuster für einen Fisch}
\label{foc9}
\end{figure}

\newpage

\begin{figure}[h]
\includegraphics[width=12cm]{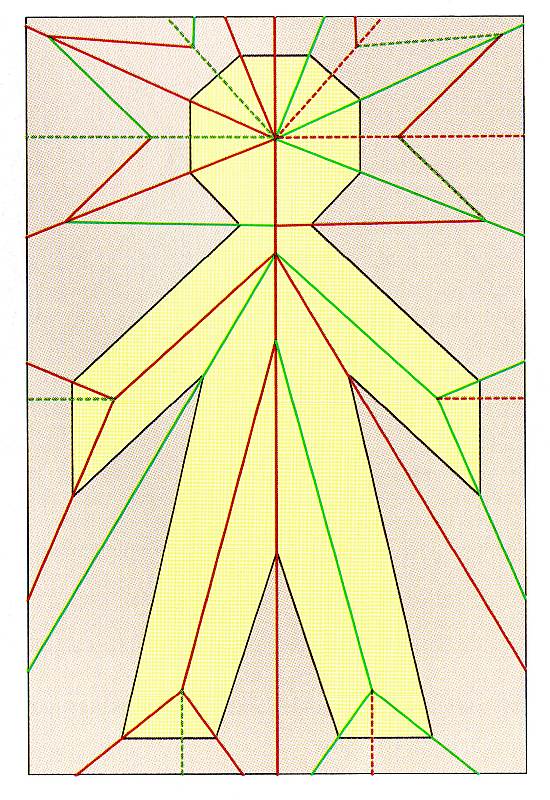}
\caption[\cite{rourke}, Seite 82]{FoaC-Faltmuster für die menschliche Gestalt}
\label{foc9}
\end{figure}

\newpage

\begin{figure}[h]
\includegraphics[width=12cm]{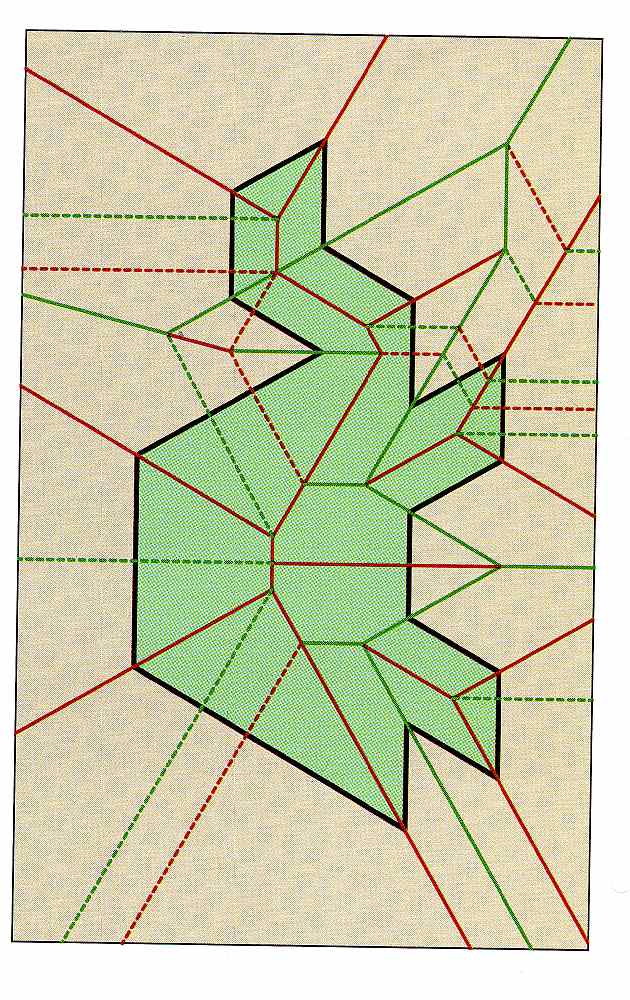}
\caption[\cite{rourke}, Seite 81]{FoaC-Faltmuster für eine Schildkröte}
\label{foc9}
\end{figure}

\newpage
{\huge \bfseries Danksagung}
~\\
~\\
~\\
Ich danke besonders Herrn Professor Knop dafür, dass er mich auf dieses spannende Thema gestoßen hat.

\newpage
\section*{Selbständigkeitserklärung}
Ich erkläre, dass ich diese Arbeit selbständig und nur unter Zuhilfenahme der zitierten Quellen verfasst habe.\\

Erlangen, im Juli 2012\\
~\\
~\\
Kay Paulus

\listoffigures

\end{document}